\documentclass[10pt]{article}

\title{The micromagnetic energy \linebreak with Dzyaloshinskii-Moriya interaction \linebreak in a thin-film regime relevant for boundary vortices}
\author{{\Large Fran\c cois L'Official}\footnote{Institut de Math\' ematiques de Toulouse, UMR 5219, Universit\' e de Toulouse, CNRS, UPS, F-31062 Toulouse Cedex 9, France. Email: Francois.Lofficial@math.univ-toulouse.fr}}

\usepackage[mathscr]{euscript}
\usepackage{pstricks} 
\usepackage{graphicx,pst-plot,pst-math} 
\usepackage{amsfonts,dsfont}
\usepackage{amsmath,latexsym}
\usepackage{amssymb,verbatim}
\usepackage{amsthm}

\theoremstyle{definition}
\newtheorem{definition}{Definition}[section]

\theoremstyle{plain}
\newtheorem{theoreme}[definition]{Theorem}
\newtheorem{lemme}[definition]{Lemma}
\newtheorem{corollaire}[definition]{Corollary}
\newtheorem{proposition}[definition]{Proposition}

\theoremstyle{remark}
\newtheorem{remarque}[definition]{Remark}

\newtheorem{notation}[definition]{Notation}

\usepackage{csquotes}
\usepackage{dsfont}
\usepackage{appendix}

\begin{document}

\maketitle

\begin{abstract}
We consider the three-dimensional micromagnetic model with Dzyaloshinskii-Moriya interaction in a thin-film regime. We prove the Gamma-convergence of the micromagnetic energy in the considered regime, for which the Gamma-limit energy is two-dimensional and relevant for boundary vortices. We then study local minimizers of the Gamma-limit energy and prove a uniqueness result in a certain setting.
\end{abstract}

\tableofcontents

\section{Introduction}

The modeling of small ferromagnetic materials is based on the theory of micromagnetics. The model states that a ferromagnetic device can be described by a three-dimensional vector field with values in $\mathbb{S}^2$, called magnetization, whose stable states correspond to local minimizers of the micromagnetic energy (see e.g.~\cite{HS98}). The associated variational problem is non-convex, non-local and multiscale.
In particular, the relation between length parameters -- as well intrinsic parameters of the magnetic material as geometric parameters -- allows to consider a lot of different asymptotic regimes. This leads to the formation of magnetic patterns, such as domain walls, vortices, etc. We are interested in a special regime relevant for boundary vortices.

\subsection{The general three-dimensional model}

We consider a ferromagnetic sample of cylindrical shape
\begin{equation*}
\Omega
:= \omega_\ell \times (0,t) \subset \mathbb{R}^3,
\end{equation*}
where $\omega_\ell \subset \mathbb{R}^2$ is a smooth bounded open set of typical length $\ell$ (for example, $\omega_\ell$ can be assumed to be a disk of diameter $\ell$).
The magnetization $m$ is a unitary three-dimensional vector field
\begin{equation*}
m \colon \Omega \rightarrow \mathbb{S}^2,
\end{equation*}
where $\mathbb{S}^2$ is the unit sphere in $\mathbb{R}^3$. In particular, the constraint $\left\vert m \right\vert = 1$ yields the non-convexity of the problem. 
The classical micromagnetic energy is defined as
\begin{equation*}
E_\mathrm{class}(m)
:= A^2 \int_{\Omega} \left\vert \nabla m \right\vert^2 \mathrm{d} x
+ \int_{\mathbb{R}^3} \left\vert \nabla u \right\vert^2 \mathrm{d} x
+ Q \int_{\Omega} \Phi(m) \ \mathrm{d} x
- 2 \int_{\Omega} H_\mathrm{ext} \cdot m \ \mathrm{d} x.
\end{equation*}
Let us explain briefly the four terms of $E_\mathrm{class}$. 
The first term is the exchange energy. It is generated by small-distance interactions in the sample ; roughly speaking, this energy favors the alignment of neighboring spins in the sample. The positive parameter $A$ is an intrinsic parameter of the ferromagnetic material, and is typically of the order of nanometers: we call it the exchange length. 
The second term is called magnetostatic or demagnetizing energy. This energy is generated by the large-distance interactions in the sample. It is in fact the energy generated by the magnetic field induced by magnetization. More precisely, the demagnetizing potential $u \in H^1(\mathbb{R}^3,\mathbb{R})$ satisfies
\begin{equation}
\label{eqdistrib-u}
\Delta u = \mathrm{div}(m \mathds{1}_{\Omega}) \ \text{ in the distributional sense in } \mathbb{R}^3,
\end{equation}
where $\mathds{1}_{\Omega}(x)=1$ if $x \in \Omega$, and $\mathds{1}_{\Omega}(x)=0$ elsewhere.
The third term is the anisotropy energy: it takes into account the anisotropy effects resulting from the crystalline structure of the sample. It involves the quality factor $Q>0$ (that is a second intrinsic parameter of the sample) and the function $\Phi \colon \mathbb{S}^2 \rightarrow \mathbb{R}_+$ which has some symmetry properties.
The fourth and last term is the external field or Zeeman energy: it is generated by an external magnetic field. It is the vector field
\begin{equation*}
H_\mathrm{ext} \colon \mathbb{R}^3 \rightarrow \mathbb{R}^3
\end{equation*}
and favours the alignment of the magnetization in the direction of the external magnetic field.

In this paper, we consider an additional term: the Dzyaloshinskii-Moriya interaction. This interaction was introduced in the fifties~\cite{Dzyalo} to describe the magnetization in some materials with few symmetry properties. We assume here that the Dzyaloshinskii-Moriya interaction density in three dimensions is defined as
\begin{equation}
\label{dmi_density}
D : \nabla m \wedge m
:= \sum_{j=1}^3 D_j \cdot \partial_j m \wedge m,
\end{equation}
where $D=(D_1,D_2,D_3) \in \mathbb{R}^{3 \times 3}$, $\cdot$ denotes the inner product in $\mathbb{R}^3$, and $\wedge$ denotes the cross product in $\mathbb{R}^3$.
Hence, we consider the micromagnetic energy $E(m)$ given as
\begin{equation}
\label{exprE}
\begin{split}
E(m)
&
:= A^2 \int_{\Omega} \left\vert \nabla m \right\vert^2 \mathrm{d} x
+ \int_{\Omega} D : \nabla m \wedge m \ \mathrm{d} x
+ \int_{\mathbb{R}^3} \left\vert \nabla u \right\vert^2 \mathrm{d} x
\\
& \quad
+ Q \int_{\Omega} \Phi(m) \ \mathrm{d} x
- 2 \int_{\Omega} H_\mathrm{ext} \cdot m \ \mathrm{d} x,
\end{split}
\end{equation}
with $D \in \mathbb{R}^{3 \times 3}$.
For more details about the components of the micromagnetic energy, especially physical interpretations, we refer to~\cite{Aharoni}, \cite{HS98} or~\cite{Dzyalo}.

The multiscale aspect of the micromagnetic energy \eqref{exprE} is obvious. Indeed, beside the tensor $D$ and the quality factor $Q$, three length paramaters of the ferromagnetic device interact together: the exchange length $A$, the planar diameter $\ell$ and the thickness $t$ of the sample.
From these parameters, we introduce the dimensionless parameters
\begin{equation*}
h := \frac{t}{\ell} \ \ \text{ and } \ \ d := \frac{A}{\ell}.
\end{equation*}
By letting $h$ tend to zero, the relative thickness of the ferromagnetic device tends to zero: it is a thin-film limit. The consequences concerning the magnetization and the micromagnetic energy depend on the relations between $h$ and $d$, i.e. on the thin-film asymptotic regime.

\subsection{A thin-film regime}

\subsubsection{Nondimensionalization in length}

In order to study the micromagnetic energy in a thin-film regime, it is convenient to nondimensionalize it in length ; in particular, we get from the three length parameters $A$, $\ell$ and $t$ only two dimensionless parameters $h$ and $d$ defined above.
We set
\begin{equation*}
\Omega_h := \frac{\Omega}{\ell} = \omega \times (0,h) \subset \mathbb{R}^3,
\end{equation*}
where $\omega := \frac{\omega_\ell}{\ell} \subset \mathbb{R}^2$ is a smooth bounded open set of typical length $1$ (for example, $\omega$ can be assumed to be the unit disk in $\mathbb{R}^2$).
To each $x=(x_1,x_2,x_3) \in \Omega$, we associate $\widehat{x} = \frac{x}{\ell} \in \Omega_h$ and we set $\widehat{D}=\frac{1}{\ell}D$. We also consider the maps $m_h \colon \Omega_h \rightarrow \mathbb{S}^2$, $u_h \colon \mathbb{R}^3 \rightarrow \mathbb{R}$ and $H_{\mathrm{ext},h} \colon \mathbb{R}^3 \rightarrow \mathbb{R}^3$ such that, for every $\widehat{x}=\frac{x}{\ell} \in \Omega_h$,
\begin{equation*}
m_h(\widehat{x})=m(x),
\ \ \
u_h(\widehat{x})=\frac{1}{\ell}u(x),
\end{equation*}
that satisfy
\begin{equation}
\Delta u_h = \mathrm{div}(m_h\mathds{1}_{\Omega_h}) \ \text{ in the distributional sense in } \mathbb{R}^3,
\label{eqdistrib-uh}
\end{equation}
and
\begin{equation*}
H_{\mathrm{ext},h}(\widehat{x})=H_{\mathrm{ext}}(x).
\end{equation*}
The micromagnetic energy \eqref{exprE} can then be written in terms of $m_h$:
\begin{equation}
\label{energymh}
\begin{split}
\widehat{E}(m_h)
& = \ell^3 \left[
d^2 \int_{\Omega_h} \vert \nabla m_h \vert^2 \mathrm{d} \widehat{x}
+ \int_{\Omega_h} \widehat{D} : \nabla m_h \wedge m_h \ \mathrm{d} \widehat{x}
\right.
\\
& \quad \qquad
+ \int_{\mathbb{R}^3} \vert \nabla u_h \vert^2 \mathrm{d} \widehat{x}
\\
& \quad \qquad \left.
+ Q \int_{\Omega_h} \Phi(m_h) \ \mathrm{d} \widehat{x}
- 2 \int_{\Omega_h} H_{\mathrm{ext},h} \cdot m_h \ \mathrm{d} \widehat{x}
\right].
\end{split}
\end{equation}
For simplicity of the notations, we write $x$ instead of $\widehat{x}$ in the following.

\subsubsection{Heuristic approach of a thin-film regime}

The thin-film model is characterized by the assumption $h=t/\ell \rightarrow 0$, i.e. the variations of the magnetization $m$ with respect to the vertical component $x_3$ are strongly penalized.
With this in mind, we assume for a while that $m_h$ does not depend on $x_3$, i.e.
\begin{equation}
\label{hyp_heuristic1}
m_h(x_1,x_2,x_3)=m_h(x_1,x_2) \colon \omega \rightarrow \mathbb{S}^2,
\end{equation}
and that
\begin{equation}
\label{hyp_heuristic2}
m_h \text{ varies on length scales } \gg h.
\end{equation}
Moreover, we also assume that the external magnetic field $H_{\mathrm{ext},h}$ is in-plane, invariant in $x_3$ and independent of $h$, i.e.
\begin{equation}
\label{hyp_heuristic3}
H_{\mathrm{ext},h}(x_1,x_2,x_3) = (H'_{\mathrm{ext},h}(x_1,x_2),0).
\end{equation}
The Maxwell equation \eqref{eqdistrib-uh} implies
\begin{equation*}
\Delta u_h
= \mathrm{div}(m_h\mathds{1}_{\Omega_h})
= \mathrm{div}(m_h)\mathds{1}_{\Omega_h} - (m_h \cdot \nu) \mathds{1}_{\partial\Omega_h}
= \mathrm{div}'(m'_h)\mathds{1}_{\Omega_h} - (m_h \cdot \nu) \mathds{1}_{\partial\Omega_h}
\end{equation*}
in the distributional sense in $\mathbb{R}^3$, where $\nu$ is the outer unit normal vector on $\partial\Omega_h$, $m_h'=(m_{h,1},m_{h,2})$ and $\mathrm{div}'(m'_h)=\partial_1m_{h,1}+\partial_2m_{h,2}$. In other words, $u_h$ is a solution of the problem
\begin{equation*}
\left\lbrace
\begin{array}{rcll}
\Delta u_h & = & \mathrm{div}'(m'_h) & \text{ in } \Omega_h,
\\ \Delta u_h & = & 0 & \text{ in } \mathbb{R}^3 \setminus \Omega_h,
\\ \left[ \frac{\partial u_h}{\partial \nu} \right] & = & m_h \cdot \nu & \text{ on } \partial \Omega_h,
\end{array}
\right.
\end{equation*}
where $[a]=a^+-a^-$ stands for the jump of $a$ with respect to the outer unit normal vector $\nu$ on $\partial \Omega_h$.
From~\cite[Section~2.1.2]{IgnatHDR} (see also~\cite{DKMO04}), by assuming $u_h \in H^1(\mathbb{R}^3)$, we can express the stray-field energy by considering the Fourier transform in the horizontal variables:
\begin{equation*}
\begin{split}
\int_{\mathbb{R}^3} \left\vert \nabla u_h \right\vert^2 \mathrm{d} x
& =
h \int_{\mathbb{R}^2} \frac{\left\vert \xi' \cdot \mathcal{F}(m'_h\mathds{1}_{\omega})(\xi') \right\vert^2}{\left\vert \xi' \right\vert^2} \left( 1 - g_h \left( \left\vert \xi' \right\vert \right) \right) \mathrm{d} \xi'
\\
& \quad
+ h \int_{\mathbb{R}^2} \left\vert \mathcal{F} (m_{h,3} \mathds{1}_{\omega})(\xi') \right\vert^2 g_h \left( \left\vert \xi' \right\vert \right) \mathrm{d} \xi',
\end{split}
\end{equation*}
where $\mathcal{F}$ stands for the Fourier transform in $\mathbb{R}^2$, i.e. for every $f \colon \mathbb{R}^2 \rightarrow \mathbb{C}$ and for every $\xi' \in \mathbb{R}^2$,
\begin{equation*}
\mathcal{F}(f)(\xi')
= \hat{f}(\xi')
= \int_{\mathbb{R}^2} f(x')e^{- 2i\pi x' \cdot \xi'} \mathrm{d} x',
\end{equation*}
and
\begin{equation}
\label{exprgh}
g_h(\vert \xi' \vert) = \frac{1-e^{-2\pi h \vert \xi' \vert}}{2\pi h \vert \xi' \vert}.
\end{equation}

\begin{remarque}
\label{rem_gh}
The function $g_h$ satisfies the following useful properties.
\\
For every $h>0$ and $\xi' \in \mathbb{R}^2$, $e^{-2\pi h \vert \xi' \vert} \in (0,1]$, hence $g_h(\vert \xi' \vert) \geqslant 0$. Moreover, for every $\xi' \in \mathbb{R}^2$,
\begin{equation*}
e^{-2\pi h \vert \xi' \vert}
= 1 - 2\pi h \vert \xi' \vert + 2\pi^2 h^2 \left\vert \xi' \right\vert^2 + o(h^2) \ \ \ \text{ when } h \rightarrow 0,
\end{equation*}
thus $(g_h)_{h>0}$ converges almost everywhere to $1$ in $\mathbb{R}^2$ and, for every $\xi' \in \mathbb{R}^2$, for every $h>0$, $g_h(\vert \xi' \vert)$ \linebreak is bounded independently of $h$.
\end{remarque}

In the asymptotics $h \rightarrow 0$, we have $g_h \left( \left\vert \xi' \right\vert \right) \rightarrow 1$ and $1-g_h \left( \left\vert \xi' \right\vert \right) \rightarrow 0$, hence we can approximate
\begin{equation*}
\int_{\mathbb{R}^3} \left\vert \nabla u_h \right\vert^2 \mathrm{d} x
\approx h^2 \int_{\mathbb{R}^2} \frac{\left\vert \xi' \cdot \mathcal{F}(m'_h\mathds{1}_{\omega})(\xi') \right\vert^2}{\left\vert \xi' \right\vert} \mathrm{d} \xi'
+ h \int_{\omega} m_{h,3}^2 \ \mathrm{d} x'.
\end{equation*}
A more precise approach (see~\cite{DKMO04},~\cite{KS05}) is given by
\begin{equation}
\label{heuristic_strayfield}
\begin{split}
\int_{\mathbb{R}^3} \left\vert \nabla u_h \right\vert^2 \mathrm{d} x
& \approx h^2 \left\Vert \mathrm{div}'(m'_h)_{cont} \right\Vert_{\dot{H}^{-1/2}(\mathbb{R}^2)}^2
\\
& \quad + \dfrac{1}{2\pi} h^2 \left\vert \log h \right\vert \int_{\partial \omega} \left( m'_h \cdot \nu' \right)^2 \mathrm{d} \mathcal{H}^1
+ h \int_{\omega} m_{h,3}^2 \ \mathrm{d} x',
\end{split}
\end{equation}
where $\mathrm{div}'(m'_h)_{cont} = \mathrm{div}'(m'_h)\mathds{1}_\omega$ and $\nu'$ is the outer unit normal vector on $\partial\omega$. Hence, the stray-field energy is asymptotically decomposed in three terms in the thin-film approximation. The first term penalizes the volume charges, as an homogeneous $\dot{H}^{-1/2}$ seminorm, and favors N\'eel walls. The second term takes in account the lateral charges on the cylindrical sample and favors boundary vortices. The third term penalizes the surface charges on the top and bottom of the cylinder, and leads to interior vortices. For more details on the different types of singularities that can occur in thin-film regimes, we refer to~\cite{DKMO04} or~\cite{IgnatHDR}.
Combining~\eqref{heuristic_strayfield} with the assumptions~\eqref{hyp_heuristic1},~\eqref{hyp_heuristic2} and~\eqref{hyp_heuristic3}, we get the following approximation for the reduced two-dimensional thin-film energy:
\begin{equation}
\label{heuristic_energy}
\begin{split}
\widehat{E}(m_h)
& \approx
\ell^3 \left[
d^2h \int_\omega \vert \nabla'm_h \vert^2 \mathrm{d} x'
+ h \int_\omega \widehat{D}' : \nabla'm_h \wedge m_h \ \mathrm{d} x'
\right.
\\
& \quad \qquad
+ \frac{h^2}{2} \left\Vert \mathrm{div}'(m'_h)_{cont} \right\Vert^2_{\dot{H}^{-1/2}(\mathbb{R}^2)}
\\
& \quad \qquad
+ \frac{1}{2\pi}h^2 \left\vert \log h \right\vert \int_{\partial\omega} (m'_h \cdot \nu')^2 \mathrm{d} \mathcal{H}^1
\\
& \quad \qquad
\left.
+ h \int_\omega \left( m_{h,3}^2 +Q\Phi(m_h)-2H_{\mathrm{ext},h}' \cdot m'_h \right) \mathrm{d} x'
\right],
\end{split}
\end{equation}
with $\widehat{D}'=(\widehat{D}_1,\widehat{D}_2)$ and $\widehat{D}' : \nabla'm_h \wedge m_h = \sum_{j=1}^2 \widehat{D}_j \cdot \partial_j m_h \wedge m_h$.

The expression~\eqref{heuristic_energy} is interesting because it allows us to see the diversity of thin-film regimes. By \textit{thin-film regime}, we mean an asymptotic relation between $h$ and $d$ when $h \rightarrow 0$. In the following, we want to consider a thin-film regime that favors boundary vortices, while taking account of the Dzyaloshinskii-Moriya interaction, the anisotropy and the external magnetic field. Hence, regarding~\eqref{heuristic_energy}, we renormalize by $h^2 \left\vert \log h \right\vert$ so that
\begin{equation*}
\frac{d^2}{h \left\vert \log h \right\vert} \rightarrow 1,
\ \ 
\frac{\widehat{D}_{13}}{h \left\vert \log h \right\vert} \rightarrow 1,
\ \ 
\frac{\widehat{D}_{23}}{h \left\vert \log h \right\vert}
\rightarrow 1,
\ \ 
\frac{Q}{h \left\vert \log h \right\vert} \rightarrow 1,
\ \ 
\frac{H'_{\mathrm{ext},h}}{h \left\vert \log h \right\vert} \rightarrow 1,
\end{equation*}
and the remaining components in $\widehat{D}$ are coefficients of terms involving $m_{h,3}$, hence they must vanish at the thin-film limit.

Our assumptions are based on the regime studied by Kohn and Slastikov~\cite{KS05} for which the magnetization develops boundary vortices ; the new point here being that we take in account the Dzyaloshinskii-Moriya interaction. By letting $\frac{d^2}{h \left\vert \log h \right\vert}$ tend to zero, we obtain the regime studied by Kurkze~\cite{Kurzke06} and Ignat-Kurzke~\cite{IK21},~\cite{IK22}. Boundary vortices can occur in other thin-film regimes (see Moser~\cite{Moser04}). We should also mention the recent work of Davoli, Di Fratta, Praetorius and Ruggeri~\cite{DDFPR20} for finding a thin-film limit of the micromagnetic energy with Dzyaloshinskii-Moriya interaction, that is similar to what we do in this paper. Recently, Alama, Bronsard and Golovaty~\cite{ABG20} discovered a special type of boundary vortices, called boojums, that appear in a thin-film model of nematic liquid crystals. That type of boundary vortices costs less energy than the one we study in this paper, so that our "classical" boundary vortices are not present in their model. For studies on the N\'eel walls, we refer to Ignat~\cite{Ignat09}, Ignat-Moser~\cite{IM16},~\cite{IM19}, Ignat-Otto~\cite{IO08},~\cite{IO12} and Melcher~\cite{Melcher03},~\cite{Melcher04}.

Coming back to a general magnetization $m_h$ depending on $x_3$ and $h$, we introduce the $x_3$-average of $m_h$ in the following Section~\ref{subsection_notations}. This quantity will be useful for reducing the three-dimensional general model to a two-dimensional model.

\subsection{Notations}
\label{subsection_notations}

We use the same notation $\cdot$ for the inner product both in $\mathbb{R}^3$ and in $\mathbb{R}^2$. We use the same notation~$\wedge$ for the cross product in $\mathbb{R}^3$, i.e.
\begin{equation*}
\left(
\begin{matrix}
a_1
\\
a_2
\\
a_3
\end{matrix}
\right)
\wedge
\left(
\begin{matrix}
b_1
\\
b_2
\\
b_3
\end{matrix}
\right)
=
\left(
\begin{matrix}
a_2 b_3 - a_3 b_2
\\
a_3 b_1 - a_1 b_3
\\
a_1 b_2 - a_2 b_1
\end{matrix}
\right),
\end{equation*}
and for the determinant in $\mathbb{R}^2$, i.e.
\begin{equation*}
\binom{a_1}{a_2}
\wedge
\binom{b_1}{b_2}
= a_1 b_2 - a_2 b_1.
\end{equation*}

When it is relevant, we use the identification $\mathbb{R}^2 \simeq \mathbb{C}$. We denote by $\Re(\cdot)$ and $\Im(\cdot)$ the real and imaginary parts of a complex number.

\noindent
Note that, for $a=(a_1,a_2)=a_1+ia_2$ and $b=(b_1,b_2)=b_1+ib_2$, we have
$a \wedge b
= \Im(\overline{a}b)$. We denote
\begin{equation*}
\mathbb{R}_+^2 := \mathbb{R} \times (0,+\infty), \ \ \ \text{ and } \ \ \ \overline{\mathbb{R}_+^2} := \mathbb{R} \times [0,+\infty).
\end{equation*}
For $r>0$, we denote by $B_r$ the disk
\begin{equation*}
B_r := B(0,r)= \left\lbrace (x_1,x_2) \in \mathbb{R}^2 : x_1^2+x_2^2 < r \right\rbrace,
\end{equation*}
and by $B_r^+$ the upper-half disk
\begin{equation*}
B_r^+ := \left\lbrace (x_1,x_2) \in \mathbb{R}_+^2 : x_1^2+x_2^2 < r \right\rbrace.
\end{equation*}

For a three-dimensional quantity $q=(q_1,q_2,q_3)$, we use a prime $'$ to call the in-plane component of $q$, i.e. $q'=(q_1,q_2)$ and $q=(q',q_3)$. Note that when the framework is two-dimensional and no confusion is possible, e.g. in Section \ref{section2}, we remove the primes from all notations.

As already defined in \eqref{dmi_density} for the Dzyaloshinskii-Moriya interaction density, we denote
\begin{equation*}
F : \nabla f \wedge f = \sum\limits_{j=1}^3 F_j \cdot \partial_j f \wedge f
\end{equation*}
for $F \in \mathbb{R}^{3 \times 3}$ and $f \colon \mathbb{R}^3 \rightarrow \mathbb{R}^3$. We also denote
\begin{equation*}
G : \nabla' f \wedge f = \sum\limits_{j=1}^2 G_j \cdot \partial_j f \wedge f
\end{equation*}
for $G \in \mathbb{R}^{3 \times 2}$ and $f \colon \mathbb{R}^3 \rightarrow \mathbb{R}^3$.

Most of the time, the constants (often denoted by $C$) can change from line to line in the calculations. Finally, the formulations "a sequence $(m_h)_{h>0}$ / when $h \rightarrow 0$" have to be understood as "a sequence $(m_{h_n})_{n \in \mathbb{N}}$ / with $h_n \rightarrow 0$ as $n \rightarrow +\infty$".

\begin{notation}
\label{notation_mhmean}
For $h>0$, we denote by $\overline{m}_h \colon \omega \rightarrow \mathbb{R}^3$ the $x_3$-average of $m_h$, i.e.
\begin{equation}
\label{def_mhmean}
\overline{m}_h(x') = \frac{1}{h} \int_0^h m_h(x',x_3) \ \mathrm{d} x_3
\end{equation}
for every $x' \in \omega$, and we denote by $\overline{u}_h \colon \mathbb{R}^3 \rightarrow \mathbb{R}$ the associated stray field potential given by
\begin{equation}
\label{eqdistrib-uhbar}
\Delta \overline{u}_h = \mathrm{div}(\overline{m}_h\mathds{1}_{\Omega_h}) \ \text{ in the distributional sense in } \mathbb{R}^3.
\end{equation}
Note that the unit-length constraint is convexified by averaging, thus $\left\vert \overline{m}_h \right\vert \leqslant 1$.
\end{notation}

\subsection{Main results}

Our aim consists in studying the energy $\widehat{E}(m_h)$, given in \eqref{energymh}, in a thin-film regime governed by the main assumptions that $\frac{d^2}{h \left\vert \log h \right\vert}$, $\frac{\widehat{D}_{13}}{h \left\vert \log h \right\vert}$, $\frac{\widehat{D}_{23}}{h \left\vert \log h \right\vert}$, $\frac{Q}{h \left\vert \log h \right\vert}$ and $\frac{H_{\mathrm{ext},h}}{h \left\vert \log h \right\vert}$ are of order $1$. More precisely, assuming that all parameters are functions in $h$, we consider here the regime
\begin{equation}
\begin{split}
& h \ll 1,
\ \ \frac{d^2}{h \left\vert \log h \right\vert} \rightarrow \alpha,
\ \ \frac{Q}{h \left\vert \log h \right\vert} \rightarrow \beta, \ \ \frac{\widehat{D}_{13}}{d^2} \rightarrow 2\delta_1,
\ \ \frac{\widehat{D}_{23}}{d^2} \rightarrow 2\delta_2,
\\
& \frac{1}{h \left\vert \log h \right\vert} \sum_{j,k=1}^2 \vert \widehat{D}_{jk} \vert \ll 1,
\ \ \frac{1}{h \left\vert \log h \right\vert} \sum_{k=1}^3 \vert \widehat{D}_{3k} \vert \ll 1,
\end{split}
\label{regime1.1}
\end{equation}
where $\alpha,\beta>0$, $\delta_1,\delta_2 \in \mathbb{R}$ and $\widehat{D}=(\widehat{D}_{jk})_{(j,k) \in \left\lbrace 1,2,3 \right\rbrace^2} \in \mathbb{R}^{3 \times 3}$.

\subsubsection{Gamma-convergence of the micromagnetic energy with Dzyaloshinskii-Moriya interaction in a thin-film regime for boundary vortices}
\ \\
We consider the rescaled energy
\begin{equation}
\label{Ehmhchp1}
E_h(m_h) := \frac{\widehat{E}(m_h)}{\ell^3 h^2 \left\vert \log h \right\vert}.
\end{equation}
For every $h>0$, we define $\widetilde{m}_h \colon \Omega_1 = \omega \times (0,1) \rightarrow \mathbb{S}^2$ and $\widetilde{H}_{\mathrm{ext},h} \colon \mathbb{R}^3 \rightarrow \mathbb{R}^3$ to be such that, for every $(x',x_3) \in \Omega_1$,  
\begin{equation}
\label{rel_mhtilde_mh}
\widetilde{m}_h(x',x_3)=m_h(x',hx_3),
\end{equation}
and
\begin{equation}
\label{rel_Hexthtilde_Hexth}
\widetilde{H}_{\mathrm{ext},h}(x',x_3)=H_{\mathrm{ext},h}(x',hx_3).
\end{equation}
Moreover, we assume that
\begin{equation}
\left( \frac{\widetilde{H}_{\mathrm{ext},h}}{h \left\vert \log h \right\vert} \right)_{h>0} \text{ converges in } L^1(\Omega_1) \text{ to } \gamma \widetilde{H}_{\mathrm{ext},0},
\label{regime1.2}
\end{equation}
where $\gamma \in \mathbb{R}$ and $\widetilde{H}_{\mathrm{ext},0} \colon \mathbb{R}^3 \rightarrow \mathbb{R}^3$ is independent of $x_3$. Setting finally
\begin{equation*}
\widetilde{E}_h(\widetilde{m}_h)=E_h(m_h),
\end{equation*}
we have (see Section~\ref{subsection1.1} for details):
\begin{equation}
\begin{split}
\widetilde{E}_h(\widetilde{m}_h)
& = \frac{d^2}{h \left\vert \log h \right\vert} \int_{\Omega_1} \left( \vert \nabla'\widetilde{m}_h \vert^2 + \frac{1}{h^2} \left\vert \partial_3\widetilde{m}_h \right\vert^2 \right) \mathrm{d} x
\\ & \quad
+ \frac{1}{h \left\vert \log h \right\vert} \int_{\Omega_1} \widehat{D}' : \nabla' \widetilde{m}_h \wedge \widetilde{m}_h \ \mathrm{d} x
\\ & \quad
+ \frac{1}{h^2 \left\vert \log h \right\vert} \int_{\Omega_1} \widehat{D}_3 \cdot \partial_3 \widetilde{m}_h \wedge \widetilde{m}_h \ \mathrm{d} x
\\ & \quad
+ \frac{1}{h^2 \left\vert \log h \right\vert} \int_{\mathbb{R}^3} \left\vert \nabla u_h \right\vert^2 \mathrm{d} x
\\ & \quad + \frac{Q}{h \left\vert \log h \right\vert} \int_{\Omega_1} \Phi(\widetilde{m}_h) \ \mathrm{d} x
- \frac{2}{h \left\vert \log h \right\vert} \int_{\Omega_1} \widetilde{H}_{\mathrm{ext},h} \cdot \widetilde{m}_h \ \mathrm{d} x.
\end{split}
\label{exprEh}
\end{equation}
We prove the following statement of $\Gamma$-convergence for $\widetilde{E}_h$. It justifies that in the regime~\eqref{regime1.1}+\eqref{regime1.2}, we obtain a reduction from a 3D model to a 2D model by Gamma-convergence.

\begin{theoreme}
\label{GC}
Consider the regime~\eqref{regime1.1}$+$\eqref{regime1.2}. Let $\nu$ be the outer unit normal vector on $\partial \Omega_1$ and~$\delta=(\delta_1,\delta_2)$. Let $\widetilde{E}_h$ be given in~\eqref{exprEh}. Then the following statements hold:
\\
\textbf{(i) Compactness and lower bound.}
\\
Assume that there exists a constant $C>0$ such that, for every $h>0$, $\widetilde{E}_h(\widetilde{m}_h) \leqslant C$. Then, for a subsequence, $(\widetilde{m}_h)_{h>0}$ converges weakly to $\widetilde{m}_0$ in $H^1(\Omega_1,\mathbb{S}^2)$, where $\widetilde{m}_0$ is independent of $x_3$ and satisfies $\widetilde{m}_{0,3} \equiv 0$. Moreover,
\begin{equation*}
\liminf\limits_{h \rightarrow 0} \widetilde{E}_h(\widetilde{m}_h)
\geqslant \widetilde{E}_0(\widetilde{m}_0),
\end{equation*}
with
\begin{equation}
\label{exprE0}
\begin{split}
\widetilde{E}_0(\widetilde{m})
& = \alpha \left[
\int_{\Omega_1} \vert \nabla' \widetilde{m} \vert^2 \mathrm{d} x
+ 2 \int_{\Omega_1} \delta \cdot \nabla' \widetilde{m}' \wedge \widetilde{m}' \ \mathrm{d} x
\right]
\\
& \quad
+ \frac{1}{2\pi} \int_{\partial \omega} \left( \widetilde{m} \cdot \nu \right)^2 \mathrm{d} \mathcal{H}^1
+ \beta \int_{\Omega_1} \Phi(\widetilde{m}) \ \mathrm{d} x
- 2 \gamma \int_{\Omega_1} \widetilde{H}_{\mathrm{ext},0} \cdot \widetilde{m} \ \mathrm{d} x,
\end{split}
\end{equation}
if $\widetilde{m} \in H^1(\Omega_1,\mathbb{S}^2)$ is independent of $x_3$ and satisfies $\widetilde{m}_3 \equiv 0$, and $\widetilde{E}_0(\widetilde{m})=+\infty$ elsewhere.
\\
\textbf{(ii) Upper bound.}
\\
Let $\widetilde{m}_0 \in H^1(\Omega_1,\mathbb{S}^1)$ be such that $\widetilde{m}_0$ is independent of $x_3$ and $\widetilde{m}_{0,3} \equiv 0$. Then there exists a sequence $(\widetilde{m}_h)_{h>0}$ that converges strongly to $\widetilde{m}_0$ in $H^1(\Omega_1,\mathbb{S}^1)$ and satisfies
\begin{equation*}
\lim\limits_{h \rightarrow 0} \widetilde{E}_h(\widetilde{m}_h)
= \widetilde{E}_0(\widetilde{m}_0),
\end{equation*}
where $\widetilde{E}_0$ is given by~\eqref{exprE0}.
\end{theoreme}

The proof of this theorem combines the works of Kohn-Slastikov~\cite{KS05} and Carbou~\cite{Car01}, to which we add here the contribution of the anisotropy energy, the external magnetic field, and the Dzyaloshinskii-Moriya interaction.
\vspace{.5cm}

The three-to-two-dimensions reduction takes sense in the following remark.

\begin{remarque}
\label{3D2D}
We can be more precise concerning the Gamma-limit $\widetilde{E}_0$ of the sequence $(\widetilde{E}_h)_{h>0}$ in Theorem~\ref{GC}. Since $\Omega_1$ has hight $1$, $\widetilde{H}_{\mathrm{ext},0}$ is independent of $x_3$ and $\widetilde{E}_0$ is a functional of \linebreak maps $\widetilde{m} \in H^1(\Omega_1,\mathbb{S}^2)$ such that $\widetilde{m}$ is independent of $x_3$ and $\widetilde{m}_3 \equiv 0$, we have
\begin{equation}
\label{exprE0-2D}
\begin{split}
\widetilde{E}_0(\widetilde{m})
& = \widetilde{E}_0(\widetilde{m}')
\\
& = \alpha \left[
\int_{\omega} \left\vert \nabla' \widetilde{m}' \right\vert^2 \mathrm{d} x'
+ 2 \int_{\omega} \delta \cdot \nabla' \widetilde{m}' \wedge \widetilde{m}' \ \mathrm{d} x'
\right]
\\
& \quad
+ \frac{1}{2\pi} \int_{\partial \omega} (\widetilde{m}' \cdot \nu)^2 \mathrm{d} \mathcal{H}^1
+ \beta \int_{\omega} \Phi(\widetilde{m}') \ \mathrm{d} x'
- 2 \gamma \int_{\omega} \widetilde{H}_{\mathrm{ext},0} \cdot \widetilde{m}' \ \mathrm{d} x',
\end{split}
\end{equation}
for every $\widetilde{m}' \in H^1(\omega,\mathbb{S}^1)$, where $\nu$ denotes here the outer unit normal vector on $\partial \omega$. As a consequence, Theorem \ref{GC} means a reduction from a three-dimensional model to a two-dimensional model.
\end{remarque}

\begin{corollaire}
\label{GC-cor}
Consider the regime~\eqref{regime1.1}$+$\eqref{regime1.2}. Let $\widetilde{E}_h$ be given in~\eqref{exprEh}. For every $h>0$, $\widetilde{E}_h$ admits a minimizer $\widetilde{m}_h \in H^1(\Omega_1,\mathbb{S}^2)$. Furthermore, for a subsequence, $(\widetilde{m}_h)_{h>0}$ converges weakly to a minimizer $\widetilde{m}_0 \in H^1(\Omega_1,\mathbb{S}^2)$ of $\widetilde{E}_0$.
\end{corollaire}

\newpage

\subsubsection{On the local minimizers of the Gamma-limit of the micromagnetic energy in the upper-half plane.}
\ \\
We study local minimizers of the Gamma-limit given in~\eqref{exprE0-2D}, without anisotropy and external magnetic field. To do so, we analyse critical points of this energy in the upper-half plane $\mathbb{R}_+^2$, and we are led to consider a rescaled version of $\widetilde{E}_0$ denoted by
\begin{equation}
\label{exprEepsilondelta}
E_{\varepsilon}^\delta(\varphi ;\omega)
:= \frac{1}{2} \int_{\omega \cap \mathbb{R}_+^2} \left( \left\vert \nabla \varphi \right\vert^2 - 2 \delta \cdot \nabla \varphi \right) \mathrm{d} x
+ \frac{1}{2\varepsilon} \int_{\omega \cap (\mathbb{R} \times \left\lbrace 0 \right\rbrace)} \sin^2 \varphi \ \mathrm{d} \mathcal{H}^1,
\end{equation}
where $\delta = (\delta_1,\delta_2)$, $\varepsilon=2\pi\alpha$, $\omega$ is an open subset of $\mathbb{R}^2$, and $\varphi \colon \omega \rightarrow \mathbb{R}$ is a lifting of $m \colon \omega \rightarrow \mathbb{S}^1$.

\begin{definition}
\label{def_pt_critique}
A function $\varphi \in H^1_\mathrm{loc}(\overline{\mathbb{R}_+^2})$ is a critical point of $E_\varepsilon^\delta$ if
\begin{equation*}
\left. \frac{\mathrm{d}}{\mathrm{d} t} \right\vert_{t=0} E_\varepsilon^\delta(\varphi + t \psi ;\mathrm{Supp}(\psi)) = 0,
\end{equation*}
for every $\psi \in C^1(\mathbb{R}^2)$ with compact support.
\end{definition}

\begin{definition}
\label{def_local_minimizer}
A function $\varphi \in H^1_\mathrm{loc}(\overline{\mathbb{R}_+^2})$ is a local minimizer of $E_{\varepsilon}^\delta$ in $\mathbb{R}_+^2$ in the sense of De Giorgi if
\begin{equation*}
E_{\varepsilon}^\delta(\varphi ;\mathrm{Supp}(\psi)) \leqslant E_{\varepsilon}^\delta(\varphi+\psi ;\mathrm{Supp}(\psi))
\end{equation*}
for every $\psi \in C^1(\mathbb{R}^2)$ with compact support.
\end{definition}

Consider the problem
\begin{equation}
\label{PNlambda}
\tag{PN$_\lambda$}
\left\lbrace \begin{array}{l}
f \in C^\infty(\mathbb{R}_+^2) \cap C^1(\overline{\mathbb{R}_+^2}),
\\
(x_1,x_2) \mapsto f(x_1,x_2)-\lambda x_2 \text{ is bounded in } \mathbb{R}_+^2,
\\
\Delta f=0 \text{ in } \mathbb{R}_+^2,
\\
\partial_2 f-\lambda+\sin f = 0 \text{ on } \mathbb{R} \times \left\lbrace 0 \right\rbrace,
\end{array} \right.
\end{equation}
where $\lambda \in \mathbb{R}$ is a constant parameter. This problem is a modified Peierls-Nabarro problem, that has been studied by Toland and Amick~(\cite{AmickToland},~\cite{Tol97}) in the case $\lambda = 0$. Following their work, we show the following statement.

\begin{theoreme}
\label{thmSolPNlambda}
Let $\lambda \in \mathbb{R}$. Solutions of~\eqref{PNlambda} in $\overline{\mathbb{R}_+^2}$ are given by:
\begin{itemize}
\item[--] for $n \in \mathbb{Z}$, the functions
\begin{equation}
(x_1,x_2) \mapsto n\pi +\lambda x_2,
\end{equation}
\item[--] for $n \in \mathbb{Z}$ and $\alpha \in (1,2)$, the $x_1$-periodic (with period $\frac{\pi}{\sigma}$) functions
\begin{equation}
\label{solPNperiodic}
(x_1,x_2) \mapsto 2n\pi \pm 2\left[ \arctan \left( \frac{\tan(\sigma x_1)}{\Gamma_\alpha(x_2)} \right)-\arctan \left( \Gamma_\alpha(x_2)\tan(\sigma x_1) \right) \right]+\lambda x_2,
\end{equation}
defined for $x_1 \in \mathbb{R} \setminus \left( \frac{\pi}{2\sigma} + \frac{\pi}{\sigma} \mathbb{Z} \right)$ and extended by $(x_1,x_2) \mapsto 2n\pi+\lambda x_2$ elsewhere, and every translation in the variable $x_1$ of this type of functions, where $\sigma$ and $\Gamma_\alpha$ are given by
\begin{equation*}
\sigma=\frac{1}{2}\sqrt{\alpha(2-\alpha)}, \ \ \Gamma_\alpha(x_2)=\frac{\gamma+\tanh(\sigma x_2)}{1+\gamma \tanh(\sigma x_2)} \ \text{ with } \ \gamma=\frac{\alpha}{2\sigma},
\end{equation*}
\item[--] for $n \in \mathbb{Z}$, the non-periodic functions
\begin{equation}
\label{solPNnonperiodic}
(x_1,x_2) \mapsto 2n\pi \pm 2\arctan \left( \frac{x_1}{1+x_2} \right)+\lambda x_2
\end{equation}
and every translation in the variable $x_1$ of this type of functions.
\end{itemize}
\end{theoreme}

\newpage

In Proposition~\ref{propsystFeps} below, we show that if $\varphi_\varepsilon$ is a critical point of $E_{\varepsilon}^\delta$, then the rescaled function
\begin{equation*}
\phi_\varepsilon \colon (x_1,x_2) \in \overline{\mathbb{R}_+^2} \mapsto 2\varphi_\varepsilon(\varepsilon x_1,\varepsilon x_2)+\pi,
\end{equation*}
with the assumption that $(x_1,x_2) \mapsto \phi_\varepsilon(x_1,x_2)-\lambda_\varepsilon x_2$ is bounded in $\mathbb{R}_+^2$, satisfies the problem~(PN$_{\lambda_\varepsilon}$) in the case $\lambda_\varepsilon=2\varepsilon\delta_2$. Coming back to local minimizers of $E_{\varepsilon}^\delta$ in $\mathbb{R}_+^2$ in the sense of De Giorgi, we finally get the following statement.

\begin{theoreme}
\label{thm_phir_minimizer}
For $a \in \mathbb{R}$, the functions
\begin{equation}
\label{localminimizer}
\varphi_\varepsilon \colon (x_1,x_2) \in \overline{\mathbb{R}_+^2} \mapsto \frac{\pi}{2} - \arctan \left( \frac{x_1+\varepsilon a}{x_2+\varepsilon} \right) + \delta_2x_2
\end{equation}
are the only local minimizers of $E_{\varepsilon}^\delta$ in $\mathbb{R}_+^2$ in the sense of De Giorgi such that
\begin{equation}
\label{conditionsminimizer}
\lim\limits_{x_1 \rightarrow +\infty} \varphi_\varepsilon(x_1,0)=0, \ 
\lim\limits_{x_1 \rightarrow -\infty} \varphi_\varepsilon(x_1,0)=\pi \text{ and } \left[ (x_1,x_2) \mapsto \varphi_\varepsilon(x_1,x_2)-\delta_2 x_2 \right] \in L^\infty(\mathbb{R}_+^2).
\end{equation}
\end{theoreme}

\paragraph{Outline of the paper.}
The rest of the paper is organized as follows. In Section~2, we justify the expression of the energy $\widetilde{E}_h$ in~\eqref{exprEh} and prove the coercivity of this energy functional. We then prove Theorem~\ref{GC} and Corollary~\ref{GC-cor}. In Section~3, we analyse critical points of $E_\varepsilon^\delta$ given by~\eqref{exprEepsilondelta} and link them with solutions of a modified Peierls-Nabarro problem with parameter $\lambda$. We then solve this problem~(PN$_\lambda$) and prove Theorem~\ref{thmSolPNlambda}. Finally, we study local minimizers of $E_\varepsilon^\delta$ in $\mathbb{R}_+^2$ in the sense of De Giorgi and prove Theorem~\ref{thm_phir_minimizer}.
 
\paragraph{Acknowledgment.}
The author would like to thank his PhD advisor Radu Ignat for his patience and suggestions.

\section{Gamma-convergence of the micromagnetic energy with Dzyaloshinskii-Moriya interaction in a thin-film regime for boundary vortices}
\label{section1}

\subsection{The rescaled energy $\widetilde{E}_h$}
\label{subsection1.1}

Let $\Omega_h = \omega \times (0,h)$ with $\omega \subset \mathbb{R}^2$ a smooth bounded open set of typical length $1$. For a magnetization~$m_h \colon \Omega_h \rightarrow \mathbb{S}^2$ that satisfies
\begin{equation*}
\Delta u_h = \mathrm{div}(m_h\mathds{1}_{\Omega_h}) \ \text{ in the distributional sense in } \mathbb{R}^3,
\end{equation*}
we define $E_h(m_h)$ as in~\eqref{Ehmhchp1}, using~\eqref{energymh}, as
\begin{equation}
\label{exprEhmh}
\begin{split}
E_h(m_h)
& = \frac{d^2}{h \left\vert \log h \right\vert} \int_{\Omega_h} \left\vert \nabla m_h \right\vert^2 \mathrm{d} x
+ \frac{1}{h^2 \left\vert \log h \right\vert} \int_{\Omega_h} \widehat{D} : \nabla m_h \wedge m_h \ \mathrm{d} x
\\
& \quad
+ \frac{1}{h^2 \left\vert \log h \right\vert} \int_{\mathbb{R}^3} \left\vert \nabla u_h \right\vert^2 \mathrm{d} x
\\
& \quad
+ \frac{Q}{h^2 \left\vert \log h \right\vert} \int_{\Omega_h} \Phi(m_h) \ \mathrm{d} x
- \frac{2}{h^2 \left\vert \log h \right\vert} \int_{\Omega_h} H_{\mathrm{ext},h} \cdot m_h \ \mathrm{d} x,
\end{split}
\end{equation}
where $\Phi \colon \mathbb{S}^2 \rightarrow \mathbb{R}_+$, $H_{\mathrm{ext},h} \colon \mathbb{R}^3 \rightarrow \mathbb{R}^3$ and $\widehat{D} \in \mathbb{R}^{3 \times 3}$.

The energy~\eqref{exprEhmh} has been studied by Kohn and Slastikov in~\cite{KS05} by considering only the exchange and magnetostatic terms. If it seems clear that the exchange energy is one of the leading-order terms in the asymptotic regime~\eqref{regime1.1}+\eqref{regime1.2}, the work of Kohn and Slastikov shows that the limit model also keeps a contribution from the magnetostatic energy.

In order to make the energy $E_h$ easier to study, we remove the dependence on $h$ in the bounds of the involved integrals by a new rescaling. More precisely, for any $h>0$, we set $\widetilde{m}_h \colon \Omega_1 \rightarrow \mathbb{S}^2$ and $\widetilde{H}_{\mathrm{ext},h} \colon \mathbb{R}^3 \rightarrow \mathbb{R}^3$ to be such that, for every $(x',x_3) \in \Omega_1 = \omega \times (0,1)$,  
\begin{equation*}
\widetilde{m}_h(x',x_3)=m_h(x',hx_3)
\ \ \text{ and } \ \ 
\widetilde{H}_{\mathrm{ext},h}(x',x_3)=H_{\mathrm{ext},h}(x',hx_3).
\end{equation*}
For every $(x',x_3) \in \Omega_h$, we then have
\begin{align*}
\int_{\Omega_h} \vert \nabla m_h \vert^2 \mathrm{d} x
& = \int_{\omega} \int_0^h \left\vert \nabla \left( m_h(x',x_3) \right) \right\vert^2 \mathrm{d} x' \mathrm{d} x_3
\\
& = h \int_{\omega} \int_0^1 \left\vert \nabla \left( m_h(x',h\widetilde{x}_3) \right) \right\vert^2 \mathrm{d} x' \mathrm{d} \widetilde{x}_3
\\ & = h \int_{\omega} \int_0^1 \left( \left\vert \nabla'\widetilde{m}_h(x',\widetilde{x}_3) \right\vert^2 + \frac{1}{h^2} \left\vert \partial_3 \widetilde{m}_h(x',\widetilde{x}_3) \right\vert^2 \right) \mathrm{d} x'\mathrm{d} \widetilde{x}_3
\\ & = h \int_{\Omega_1} \left( \vert \nabla'\widetilde{m}_h \vert^2 + \frac{1}{h^2} \left\vert \partial_3\widetilde{m}_h \right\vert^2 \right) \mathrm{d} x,
\end{align*}
by the change of variable $x_3=h\widetilde{x}_3$ and using that $\partial_3m_h(x',h\widetilde{x}_3)=\frac{1}{h}\partial_3\widetilde{m}_h(x',\widetilde{x}_3)$. By the same change of variable,
\begin{align*}
\int_{\Omega_h} \widehat{D} : \nabla m_h \wedge m_h \ \mathrm{d} x
& = h \int_{\Omega_1} \widehat{D}' : \nabla' \widetilde{m}_h \wedge \widetilde{m}_h \ \mathrm{d} x
+ \int_{\Omega_1} \widehat{D}_3 \cdot \partial_3 \widetilde{m}_h \wedge \widetilde{m}_h \ \mathrm{d} x,
\end{align*}
where $\widehat{D}=(\widehat{D}',\widehat{D}_3)$,
\begin{equation*}
\int_{\Omega_h} \Phi(m_h) \ \mathrm{d} x
= h \int_{\Omega_1} \Phi(\widetilde{m}_h) \ \mathrm{d} x,
\ \ \text{ and } \
\int_{\Omega_h} H_{\mathrm{ext},h} \cdot m_h \ \mathrm{d} x
= h \int_{\Omega_1} \widetilde{H}_{\mathrm{ext},h} \cdot \widetilde{m}_h \ \mathrm{d} x.
\end{equation*}
Setting $\widetilde{E}_h(\widetilde{m}_h)=E_h(m_h)$, we eventually get
\begin{equation*}
\begin{split}
\widetilde{E}_h(\widetilde{m}_h)
& = \frac{d^2}{h \left\vert \log h \right\vert} \int_{\Omega_1} \left( \vert \nabla'\widetilde{m}_h \vert^2 + \frac{1}{h^2} \left\vert \partial_3\widetilde{m}_h \right\vert^2 \right) \mathrm{d} x
\\ & \quad
+ \frac{1}{h \left\vert \log h \right\vert} \int_{\Omega_1} \widehat{D}' : \nabla' \widetilde{m}_h \wedge \widetilde{m}_h \ \mathrm{d} x
\\ & \quad
+ \frac{1}{h^2 \left\vert \log h \right\vert} \int_{\Omega_1} \widehat{D}_3 \cdot \partial_3 \widetilde{m}_h \wedge \widetilde{m}_h \ \mathrm{d} x
\\ & \quad
+ \frac{1}{h^2 \left\vert \log h \right\vert} \int_{\mathbb{R}^3} \left\vert \nabla u_h \right\vert^2 \mathrm{d} x
\\ & \quad + \frac{Q}{h \left\vert \log h \right\vert} \int_{\Omega_1} \Phi(\widetilde{m}_h) \ \mathrm{d} x
- \frac{2}{h \left\vert \log h \right\vert} \int_{\Omega_1} \widetilde{H}_{\mathrm{ext},h} \cdot \widetilde{m}_h \ \mathrm{d} x
\end{split}
\end{equation*}
which is~\eqref{exprEh}. The rest of this section is devoted to prove Theorem~\ref{GC}.

\subsection{Coercivity}
\label{subsection1.2}

In this subsection, we provide a statement of coercivity concerning the energy $\widetilde{E}_h$. We begin with giving basic estimates for the Dzyaloshinskii-Moriya interaction energy, that will be useful in the sequel.

\begin{lemme}
\label{lemme-bound-dmi-1-2}
We have
\begin{equation}
\label{bound-D1mD13}
\begin{split}
\left\vert
\int_{\Omega_1} \widehat{D}_1 \cdot \partial_1 \widetilde{m}_h \wedge \widetilde{m}_h \ \mathrm{d} x
\right\vert
& \leqslant
\vert \widehat{D}_{13} \vert \int_{\Omega_1} \left\vert \partial_1 \widetilde{m}'_h \wedge \widetilde{m}'_h \right\vert \mathrm{d} x
\\
& \quad +
\left( \vert \widehat{D}_{11} \vert + \vert \widehat{D}_{12} \vert \right) \int_{\Omega_1} \left( 1+\left\vert \partial_1 \widetilde{m}_h \right\vert^2 \right) \mathrm{d} x,
\end{split}
\end{equation}
and
\begin{equation}
\label{bound-D2mD23}
\begin{split}
\left\vert
\int_{\Omega_1} \widehat{D}_2 \cdot \partial_2 \widetilde{m}_h \wedge \widetilde{m}_h \ \mathrm{d} x
\right\vert
& \leqslant
\vert \widehat{D}_{23} \vert \int_{\Omega_1} \left\vert \partial_2 \widetilde{m}'_h \wedge \widetilde{m}'_h \right\vert \mathrm{d} x
\\
& \quad +
\left( \vert \widehat{D}_{21} \vert + \vert \widehat{D}_{22} \vert \right) \int_{\Omega_1} \left( 1+\left\vert \partial_2 \widetilde{m}_h \right\vert^2 \right) \mathrm{d} x.
\end{split}
\end{equation}
\end{lemme}

\begin{proof}
We denote by $(e_1,e_2,e_3)$ the standard orthonormal basis in $\mathbb{R}^3$.

\noindent
By expanding $\widehat{D}_1$ as $\widehat{D}_1=\widehat{D}_{11}e_1+\widehat{D}_{12}e_2+\widehat{D}_{13}e_3$, we have
\begin{equation}
\label{devD1}
\begin{split}
\int_{\Omega_1} \widehat{D}_1 \cdot \partial_1 \widetilde{m}_h \wedge \widetilde{m}_h \ \mathrm{d} x
& =
\int_{\Omega_1} \widehat{D}_{11}e_1 \cdot \partial_1 \widetilde{m}_h \wedge \widetilde{m}_h \ \mathrm{d} x
\\
& \quad
+ \int_{\Omega_1} \widehat{D}_{12}e_2 \cdot \partial_1 \widetilde{m}_h \wedge \widetilde{m}_h \ \mathrm{d} x
\\
& \quad
+ \int_{\Omega_1} \widehat{D}_{13}e_3 \cdot \partial_1 \widetilde{m}_h \wedge \widetilde{m}_h \ \mathrm{d} x,
\end{split}
\end{equation}
and it is clear that
\begin{equation}
\label{devD13}
\int_{\Omega_1} \widehat{D}_{13} e_3 \cdot \partial_1 \widetilde{m}_h \wedge \widetilde{m}_h \ \mathrm{d} x
= \widehat{D}_{13} \int_{\Omega_1} \partial_1 \widetilde{m}'_h \wedge \widetilde{m}'_h \ \mathrm{d} x.
\end{equation}
Hence, it suffices to find an upper bound for the absolute value of the integrals involving~$\widehat{D}_{11}$ and~$\widehat{D}_{12}$. For the first integral, note that
\begin{align*}
\int_{\Omega_1} \widehat{D}_{11}e_1 \cdot \partial_1 \widetilde{m}_h \wedge \widetilde{m}_h \ \mathrm{d} x
& = \widehat{D}_{11} \left( \int_{\Omega_1} \widetilde{m}_{h,3} \partial_1 \widetilde{m}_{h,2} \ \mathrm{d} x - \int_{\Omega_1} \widetilde{m}_{h,2} \partial_1 \widetilde{m}_{h,3} \ \mathrm{d} x \right).
\end{align*}
By Young's inequality,
\begin{align*}
\widetilde{m}_{h,3} \partial_1 \widetilde{m}_{h,2}
& \leqslant \frac{1}{2} \left( \left\vert \widetilde{m}_{h,3} \right\vert^2 + \left\vert \partial_1 \widetilde{m}_{h,2} \right\vert^2 \right)
\leqslant \frac{1}{2} \left( 1 + \left\vert \partial_1 \widetilde{m}_{h} \right\vert^2 \right),
\end{align*}
and similarly,
\begin{align*}
\widetilde{m}_{h,2} \partial_1 \widetilde{m}_{h,3}
& \leqslant \frac{1}{2} \left( \left\vert \widetilde{m}_{h,2} \right\vert^2 + \left\vert \partial_1 \widetilde{m}_{h,3} \right\vert^2 \right)
\leqslant \frac{1}{2} \left( 1 + \left\vert \partial_1 \widetilde{m}_{h} \right\vert^2 \right).
\end{align*}
We deduce that
\begin{align*}
\left\vert \int_{\Omega_1} \widehat{D}_{11} e_1 \cdot \partial_1 \widetilde{m}_h \wedge \widetilde{m}_h \ \mathrm{d} x \right\vert
& \leqslant \vert \widehat{D}_{11} \vert \int_{\Omega_1} \left( 1 + \left\vert \partial_1 \widetilde{m}_{h} \right\vert^2 \right) \mathrm{d} x.
\end{align*}
By the same arguments, we have
\begin{align*}
\left\vert \int_{\Omega_1} \widehat{D}_{12}e_2 \cdot \partial_1 \widetilde{m}_h \wedge \widetilde{m}_h \ \mathrm{d} x \right\vert
& \leqslant \vert \widehat{D}_{12} \vert \int_{\Omega_1} \left( 1 + \left\vert \partial_1 \widetilde{m}_{h} \right\vert^2 \right) \mathrm{d} x.
\end{align*}
Combining the above inequalities with~\eqref{devD1} and~\eqref{devD13}, we deduce~\eqref{bound-D1mD13}. The proof of~\eqref{bound-D2mD23} is similar.
\end{proof}

\begin{lemme}
\label{lemme-upper-bound-dmi-3}
We have
\begin{equation}
\label{upper-bound-D3}
\begin{split}
& \left\vert \frac{1}{h} \int_{\Omega_1} \widehat{D}_3 \cdot \partial_3 \widetilde{m}_h \wedge \widetilde{m}_h \ \mathrm{d} x \right\vert
\\
& \qquad \qquad \qquad
\leqslant \left( \vert \widehat{D}_{31} \vert + \vert \widehat{D}_{32} \vert + \frac{\vert \widehat{D}_{33} \vert}{2} \right) \int_{\Omega_1} \left( 1+\frac{1}{h^2}  \left\vert \partial_3 \widetilde{m}_h \right\vert^2 \right) \mathrm{d} x.
\end{split}
\end{equation}
\end{lemme}

\begin{proof}
We denote by $(e_1,e_2,e_3)$ the standard orthonormal basis in $\mathbb{R}^3$.

\noindent
By expanding $\widehat{D}_3$ as $\widehat{D}_3=\widehat{D}_{31}e_1+\widehat{D}_{32}e_2+\widehat{D}_{33}e_3$, we have
\begin{equation*}
\begin{split}
\frac{1}{h}\int_{\Omega_1} \widehat{D}_3 \cdot \partial_3 \widetilde{m}_h \wedge \widetilde{m}_h \ \mathrm{d} x
& = \frac{1}{h}\int_{\Omega_1} \widehat{D}_{31}e_1 \cdot \partial_3 \widetilde{m}_h \wedge \widetilde{m}_h \ \mathrm{d} x
\\
& \quad + \frac{1}{h}\int_{\Omega_1} \widehat{D}_{32}e_2 \cdot \partial_3 \widetilde{m}_h \wedge \widetilde{m}_h \ \mathrm{d} x
\\
& \quad + \frac{1}{h}\int_{\Omega_1} \widehat{D}_{33}e_3 \cdot \partial_3 \widetilde{m}_h \wedge \widetilde{m}_h \ \mathrm{d} x .
\end{split}
\end{equation*}
For the first integral in the right-hand side above, note that
\begin{align*}
\frac{1}{h}\int_{\Omega_1} \widehat{D}_{31}e_1 \cdot \partial_3 \widetilde{m}_h \wedge \widetilde{m}_h \ \mathrm{d} x
& = \frac{\widehat{D}_{31}}{h} \left( \int_{\Omega_1} \widetilde{m}_{h,3} \partial_3 \widetilde{m}_{h,2} \ \mathrm{d} x - \int_{\Omega_1} \widetilde{m}_{h,2} \partial_3 \widetilde{m}_{h,3} \ \mathrm{d} x \right).
\end{align*}
By Young's inequality,
\begin{align*}
\frac{1}{h} \widetilde{m}_{h,3} \partial_3 \widetilde{m}_{h,2}
& \leqslant \frac{1}{2} \left( \left\vert \widetilde{m}_{h,3} \right\vert^2 + \frac{1}{h^2} \left\vert \partial_3 \widetilde{m}_{h,2} \right\vert^2 \right)
\leqslant \frac{1}{2} \left( 1 + \frac{1}{h^2} \left\vert \partial_3\widetilde{m}_{h} \right\vert^2 \right),
\end{align*}
and similarly,
\begin{align*}
\frac{1}{h} \widetilde{m}_{h,2} \partial_3 \widetilde{m}_{h,3}
& \leqslant \frac{1}{2} \left( \left\vert \widetilde{m}_{h,2} \right\vert^2 + \frac{1}{h^2} \left\vert \partial_3 \widetilde{m}_{h,3} \right\vert^2 \right)
\leqslant \frac{1}{2} \left( 1 + \frac{1}{h^2} \left\vert \partial_3\widetilde{m}_{h} \right\vert^2 \right).
\end{align*}
We deduce that
\begin{align*}
\left\vert \frac{1}{h} \int_{\Omega_1} \widehat{D}_{31} e_1 \cdot \partial_3 \widetilde{m}_h \wedge \widetilde{m}_h \ \mathrm{d} x \right\vert
& \leqslant \vert \widehat{D}_{31} \vert \int_{\Omega_1} \left( 1 + \frac{1}{h^2} \left\vert \partial_3\widetilde{m}_{h} \right\vert^2 \right) \mathrm{d} x.
\end{align*}
By the same arguments, we have
\begin{align*}
\left\vert \frac{1}{h} \int_{\Omega_1} \widehat{D}_{32} e_2 \cdot \partial_3 \widetilde{m}_h \wedge \widetilde{m}_h \ \mathrm{d} x \right\vert
& \leqslant \vert \widehat{D}_{32} \vert \int_{\Omega_1} \left( 1 + \frac{1}{h^2} \left\vert \partial_3\widetilde{m}_{h} \right\vert^2 \right) \mathrm{d} x.
\end{align*}
Furthermore, we have
\begin{equation*}
\frac{1}{h} \int_{\Omega_1} \widehat{D}_{33} e_3 \cdot \partial_3 \widetilde{m}_h \wedge \widetilde{m}_h \ \mathrm{d} x
= \widehat{D}_{33} \int_{\Omega_1} \frac{1}{h} \partial_3 \widetilde{m}'_h \wedge \widetilde{m}'_h \ \mathrm{d} x,
\end{equation*}
and by Young's inequality,
\begin{align*}
\frac{1}{h} \partial_3 \widetilde{m}'_h \wedge \widetilde{m}'_h = \frac{1}{h} \partial_3 \widetilde{m}'_h \cdot \left( \widetilde{m}'_h \right)^\perp
& \leqslant \frac{1}{2} \left( \frac{1}{h^2} \left\vert \partial_3 \widetilde{m}'_h \right\vert^2+\left\vert \widetilde{m}'_h \right\vert^2 \right)
\\ & \leqslant \frac{1}{2} \left( \frac{1}{h^2} \left\vert \partial_3 \widetilde{m}_h \right\vert^2+1 \right).
\end{align*}
Consequently,
\begin{equation*}
\left\vert \frac{1}{h} \int_{\Omega_1} \widehat{D}_{33} e_3 \cdot \partial_3 \widetilde{m}_h \wedge \widetilde{m}_h \ \mathrm{d} x \right\vert \leqslant \frac{\vert \widehat{D}_{33} \vert}{2} \int_{\Omega_1} \left( 1+\frac{1}{h^2} \left\vert \partial_3 \widetilde{m}_h \right\vert^2 \right) \mathrm{d} x,
\end{equation*}
and the upper bound~\eqref{upper-bound-D3} follows.
\end{proof}

\begin{notation}
In the following, we denote from~\eqref{exprEh}:
\begin{equation}
\label{decomposition-coercivity}
\widetilde{E}_h(\widetilde{m}_h)=\widetilde{E}_h^{(0)}(\widetilde{m}_h)+\widetilde{E}_h^{(1)}(\widetilde{m}_h)+\widetilde{E}_h^{(2)}(\widetilde{m}_h),
\end{equation}
where
\begin{equation}
\label{Eh0}
\begin{split}
\widetilde{E}_h^{(0)}(\widetilde{m}_h)
& = \frac{d^2}{h \left\vert \log h \right\vert} \int_{\Omega_1} \left( \vert \nabla' \widetilde{m}_h \vert^2 + \frac{1}{h^2} \left\vert \partial_3 \widetilde{m}_h \right\vert^2 \right) \mathrm{d} x
\\ & \quad
+ \frac{1}{h^2 \left\vert \log h \right\vert} \int_{\mathbb{R}^3} \left\vert \nabla u_h \right\vert^2 \mathrm{d} x
+ \frac{Q}{h \left\vert \log h \right\vert} \int_{\Omega_1} \Phi(\widetilde{m}_h) \ \mathrm{d} x,
\end{split}
\end{equation}
\begin{equation}
\label{Eh1}
\widetilde{E}_h^{(1)}(\widetilde{m}_h) = - \frac{2}{h \left\vert \log h \right\vert} \int_{\Omega_1} \widetilde{H}_{\mathrm{ext},h} \cdot \widetilde{m}_h \ \mathrm{d} x,
\end{equation}
and
\begin{equation}
\label{Eh2}
\begin{split}
\widetilde{E}_h^{(2)}(\widetilde{m}_h)
& =
\frac{1}{h \left\vert \log h \right\vert} \int_{\Omega_1} \widehat{D}' : \nabla' \widetilde{m}_h \wedge \widetilde{m}_h \ \mathrm{d} x
\\
& \quad +
\frac{1}{h^2 \left\vert \log h \right\vert} \int_{\Omega_1} \widehat{D}_3 \cdot \partial_3 \widetilde{m}_h \wedge \widetilde{m}_h \ \mathrm{d} x.
\end{split}
\end{equation}
\end{notation}

\begin{proposition}[Coercivity]
\label{GC-coercivity}
In the regime~\eqref{regime1.1}$+$\eqref{regime1.2}, there exist constants ${h_0>0}$ and $C>0$ such that, for every $h \in (0,h_0)$,
\begin{equation}
\widetilde{E}_h(\widetilde{m}_h) \geqslant \frac{1}{2} \widetilde{E}_h^{(0)}(\widetilde{m}_h)-C.
\end{equation}
\end{proposition}

\begin{proof}
Let $h>0$. We use the decomposition~\eqref{decomposition-coercivity} of $\widetilde{E}_h(\widetilde{m}_h)$. We clearly have $\widetilde{E}_h^{(0)}(\widetilde{m}_h) \geqslant 0$.

The strategy for the two remaining terms is the following. On the one hand, the energy $\widetilde{E}_h^{(1)}(\widetilde{m}_h)$ being bounded, we will absorb it in the constant $C$. On the other hand, we will distribute the contribution of the energy $\widetilde{E}_h^{(2)}(\widetilde{m}_h)$ in the energy $\widetilde{E}_h^{(0)}(\widetilde{m}_h)$ and in the constant $C$, using Lemma~\ref{lemme-bound-dmi-1-2} and Lemma~\ref{lemme-upper-bound-dmi-3}. Another example of absorbing the DMI into other terms of the micromagnetic energy can be found in~\cite{CI_DMI}. In that article, Ignat and C\^ote absorb the DMI into the exchange and anisotropy energies, in order to prove coercivity and then a $\Gamma$-convergence result.

Using H\"older's inequality with $\left\vert \widetilde{m}_h \right\vert = 1$ in $\Omega_1$, we have
\begin{align*}
\vert \widetilde{E}_h^{(1)}(\widetilde{m}_h) \vert
& =
\left\vert \frac{2}{h \left\vert \log h \right\vert} \int_{\Omega_1} \widetilde{H}_{\mathrm{ext},h} \cdot \widetilde{m}_h \ \mathrm{d} x
\right\vert
\leqslant
2 \left\Vert \frac{\widetilde{H}_{\mathrm{ext},h}}{h \left\vert \log h \right\vert} \right\Vert_{L^1(\Omega_1)}.
\end{align*}
Furthermore, in the regime~\eqref{regime1.1}+\eqref{regime1.2}, $\left\Vert \frac{\widetilde{H}_{\mathrm{ext},h}}{h \left\vert \log h \right\vert} \right\Vert_{L^1(\Omega_1)} \rightarrow \left\vert \gamma \right\vert \Vert \widetilde{H}_{\mathrm{ext},0} \Vert_{L^1(\Omega_1)}$, hence there exists a constant $C(\gamma)>0$ such that, for $h>0$ sufficiently small, we have $\vert \widetilde{E}_h^{(1)}(\widetilde{m}_h) \vert \leqslant C(\gamma)$.

\noindent
By Lemma~\ref{lemme-bound-dmi-1-2} and Lemma~\ref{lemme-upper-bound-dmi-3}, we have
\begin{align*}
\vert \widetilde{E}_h^{(2)}(\widetilde{m}_h) \vert
& \leqslant
\frac{1}{h \left\vert \log h \right\vert} \left[ \sum_{j,k=1}^2 \vert \widehat{D}_{jk} \vert \int_{\Omega_1} \left( 1+\vert \nabla' \widetilde{m}_h \vert^2 \right) \mathrm{d} x \right.
\\
& \qquad \qquad \qquad
+ \vert \widehat{D}_{13} \vert \int_{\Omega_1} \left\vert \partial_1 \widetilde{m}'_h \wedge \widetilde{m}'_h \right\vert \mathrm{d} x
\\
& \qquad \qquad \qquad
+ \vert \widehat{D}_{23} \vert \int_{\Omega_1} \left\vert \partial_2 \widetilde{m}'_h \wedge \widetilde{m}'_h \right\vert \mathrm{d} x
\\
& \qquad \qquad \qquad \left.
+ \sum_{k=1}^3 \vert \widehat{D}_{3k} \vert \int_{\Omega_1} \left( 1+\frac{1}{h^2}  \left\vert \partial_3 \widetilde{m}_h \right\vert^2 \right) \mathrm{d} x \right].
\end{align*}
Let $\varepsilon>0$. By Young's inequality,
\begin{align*}
\frac{\vert \widehat{D}_{13} \vert}{d^2}
\int_{\Omega_1} \left\vert \partial_1 \widetilde{m}'_h \wedge \widetilde{m}'_h \right\vert \mathrm{d} x
& \leqslant
\varepsilon \int_{\Omega_1} \left\vert \partial_1 \widetilde{m}'_h \right\vert^2 \mathrm{d} x
+ \frac{1}{4\varepsilon} \left( \frac{\vert \widehat{D}_{13} \vert}{d^2} \right)^2 \int_{\Omega_1} \left\vert \widetilde{m}'_h \right\vert^2 \mathrm{d} x
\\
& \leqslant
\varepsilon \int_{\Omega_1} \left\vert \partial_1 \widetilde{m}'_h \right\vert^2 \mathrm{d} x
+ \frac{1}{4\varepsilon} \left( \frac{\vert \widehat{D}_{13} \vert}{d^2} \right)^2 \left\vert \Omega_1 \right\vert,
\end{align*}
since $\left\vert \widetilde{m}_h' \right\vert \leqslant \left\vert \widetilde{m}_h \right\vert = 1$. Using the same arguments,
\begin{align*}
\frac{\vert \widehat{D}_{23} \vert}{d^2}
\int_{\Omega_1} \left\vert \partial_2 \widetilde{m}'_h \wedge \widetilde{m}'_h \right\vert \mathrm{d} x
& \leqslant
\varepsilon \int_{\Omega_1} \left\vert \partial_2 \widetilde{m}'_h \right\vert^2 \mathrm{d} x
+ \frac{1}{4\varepsilon} \left( \frac{\vert \widehat{D}_{23} \vert}{d^2} \right)^2 \left\vert \Omega_1 \right\vert.
\end{align*}
Furthermore, we note that, in the regime~\eqref{regime1.1}$+$\eqref{regime1.2},
\begin{align*}
& \frac{\vert \widehat{D}_{13} \vert}{d^2} \rightarrow 2\delta_1,
\ \
\frac{\vert \widehat{D}_{23} \vert}{d^2} \rightarrow 2\delta_2,
\ \
\frac{1}{h \left\vert \log h \right\vert} \sum_{j,k=1}^2 \vert \widehat{D}_{jk} \vert \ll 1,
\ \
\frac{1}{h \left\vert \log h \right\vert} \sum_{k=1}^3 \vert \widehat{D}_{3k} \vert \ll 1.
\end{align*}
We deduce that, for $h>0$ sufficiently small,
\begin{align*}
\vert \widetilde{E}_h^{(2)}(\widetilde{m}_h) \vert
& \leqslant
o_h(1) \int_{\Omega_1} \left( 1+\left\vert \nabla' \widetilde{m}_h \right\vert^2 \right) \mathrm{d} x
+ \frac{d^2\varepsilon}{h \left\vert \log h \right\vert} \int_{\Omega_1} \left( \left\vert \partial_1 \widetilde{m}'_h \right\vert^2 + \left\vert \partial_2 \widetilde{m}'_h \right\vert^2 \right) \mathrm{d} x
\\
& \quad
+ \frac{d^2}{h \left\vert \log h \right\vert} \frac{\left\vert \Omega_1 \right\vert}{\varepsilon} \left( \delta_1^2+\delta_2^2+1 \right)
+ o_h(1) \int_{\Omega_1} \left( 1+\frac{1}{h^2}  \left\vert \partial_3 \widetilde{m}_h \right\vert^2 \right) \mathrm{d} x
\\
& \leqslant
\left( \frac{d^2 \varepsilon}{h \left\vert \log h \right\vert}
+o_h(1) \right) \int_{\Omega_1} \left( \left\vert \nabla' \widetilde{m}_h \right\vert^2 + \frac{1}{h^2} \left\vert \partial_3 \widetilde{m}_h \right\vert^2 \right) \mathrm{d} x
\\
& \quad
+ \frac{d^2}{h \left\vert \log h \right\vert} \frac{\left\vert \Omega_1 \right\vert}{\varepsilon} \left( \delta_1^2+\delta_2^2+1 \right)
+ o_h(1).
\end{align*}

\noindent
We eventually combine our estimates on $\widetilde{E}_h^{(1)}(\widetilde{m}_h)$ and $\widetilde{E}_h^{(2)}(\widetilde{m}_h)$: for $h>0$ sufficiently small, we have
\begin{align*}
\widetilde{E}_h(\widetilde{m}_h)
& \geqslant \widetilde{E}_h^{(0)}(\widetilde{m}_h)-C(\gamma)
\\
& \quad - \left( \frac{d^2}{h \left\vert \log h \right\vert}  \varepsilon+o_h(1) \right) \int_{\Omega_1} \left( \left\vert \nabla' \widetilde{m}_h \right\vert^2 + \frac{1}{h^2} \left\vert \partial_3 \widetilde{m}_h \right\vert^2 \right) \mathrm{d} x
\\
& \quad - \frac{d^2}{h \left\vert \log h \right\vert} \frac{\left\vert \Omega_1 \right\vert}{\varepsilon} \left( \delta_1^2+\delta_2^2+1 \right)
- o_h(1).
\end{align*}
Since $\frac{d^2}{h \left\vert \log h \right\vert} \rightarrow \alpha$ in the regime~\eqref{regime1.1}$+$\eqref{regime1.2}, then choosing $\varepsilon>0$ so small that $1-\varepsilon -o_h(1) \geqslant \frac{3}{4}$ and setting $C(\alpha,\delta_1,\delta_2)=(\alpha+1) \left\vert \Omega_1 \right\vert (\delta_1^2+\delta_2^2+1) + 1$, we get
\begin{align*}
\widetilde{E}_h(\widetilde{m}_h)
& \geqslant \frac{1}{2} \widetilde{E}_h^{(0)}(\widetilde{m}_h) - C(\gamma) - \frac{1}{\varepsilon}C(\alpha,\delta_1,\delta_2).
\end{align*}
Setting $C=C(\gamma)+\frac{1}{\varepsilon}C(\alpha,\delta_1,\delta_2)$, we get the expected estimate.
\end{proof}

\subsection{Gamma-convergence}
\label{subsection1.3}

As mentioned in the introduction, the stray-field energy in $\widetilde{E}_h(\widetilde{m}_h)$ has been studied by Kohn and Slastikov~\cite{KS05} in the regime we are considering. We cite~\cite[Lemma~4]{KS05} and the estimate~(33) from~\cite{KS05}:

\begin{theoreme}
\label{limI}
For $h>0$, set
\begin{equation}
\label{exprIh}
I(h) = \int_0^h \int_{\partial \omega} \int_0^h \int_{\partial \omega} \frac{\left( \overline{m}_h \cdot \nu \right)(x') \left( \overline{m}_h \cdot \nu  \right)(y')}{\sqrt{\vert x'-y' \vert^2+(s-t)^2}} \mathrm{d} x' \mathrm{d} s \mathrm{d} y' \mathrm{d} t,
\end{equation}
where $\nu(x')$ is the outer unit normal vector at a point $x' \in \partial \omega$.

If $(\overline{m}_h)_{h>0}$ converges weakly to $\overline{m}_0$ in $H^1(\omega)$, then
\begin{equation}
\lim\limits_{h \rightarrow 0} \frac{I(h)}{h^2\left\vert \log h \right\vert}
= 2 \int_{\partial \omega} \left( \overline{m}_0 \cdot \nu \right)^2 \mathrm{d} \mathcal{H}^1.
\end{equation}
\end{theoreme}

In particular, the constant~$2$ in the above limit is computed by using an integral operator in~\cite{Car01}.

\begin{theoreme}
\label{devstrayfield}
For $h>0$, consider $g_h$ as in~\eqref{exprgh}, $m_h \colon \Omega_h \rightarrow \mathbb{S}^2$, $u_h$ as in~\eqref{eqdistrib-uh}, $\overline{m}_h$ as in~\eqref{def_mhmean} \linebreak and $\widetilde{m}_h \colon \Omega_1 \rightarrow \mathbb{S}^2$ satisfying~\eqref{rel_mhtilde_mh}. We have, as $h$ tends to zero,
\begin{equation}
\begin{split}
& \frac{1}{h^2 \left\vert \log h \right\vert} \int_{\mathbb{R}^3} \left\vert \nabla u_h \right\vert^2 \mathrm{d} x
\\ & \qquad \qquad = \frac{1}{h \left\vert \log h \right\vert} \int_{\mathbb{R}^2} \left\vert \mathcal{F}(\overline{m}_h \cdot e_3)(\xi') \right\vert^2 g_h(\left\vert \xi' \right\vert) \mathrm{d} \xi'
\\ & \qquad \qquad \quad + \frac{I(h)}{4\pi h^2 \left\vert \log h \right\vert}
\\ & \qquad \qquad \quad + \left( 1+\left\Vert \mathrm{div}'(\widetilde{m}_h) \right\Vert_{L^2(\Omega_1)}^2
 + \frac{1}{h^2} \left\Vert \partial_3 \widetilde{m}_h \right\Vert_{L^2(\Omega_1)}^2 \right) O \left( \frac{1}{\left\vert \log h \right\vert} \right),
\end{split}
\end{equation}
where $I(h)$ is defined in~\eqref{exprIh}, $\mathcal{F}$ denotes the Fourier transformation and $e_3$ is the third unit vector of the standard orthonormal basis in $\mathbb{R}^3$.
\end{theoreme}

Both previous theorems will be useful for proving Theorem~\ref{GC}. We now present and prove three theorems (compactness, lower bound and upper bound) on the Gamma-convergence of~$\widetilde{E}_h$ that will lead to Theorem~\ref{GC}.

\begin{theoreme}[Compactness]
\label{GC-compactness}
Consider the regime~\eqref{regime1.1}$+$\eqref{regime1.2}. Assume that there exists a \linebreak constant $C>0$ such that, for every $h>0$, $\widetilde{E}_h(\widetilde{m}_h) \leqslant C$. Then, for a subsequence, $(\widetilde{m}_h)_{h>0}$ converges weakly to $\widetilde{m}_0$ in $H^1$, where $\widetilde{m}_0 \in H^1(\Omega_1,\mathbb{S}^2)$ is independent of $x_3$ and satisfies $\widetilde{m}_{0,3} \equiv 0$.
\end{theoreme}

\begin{proof}
By Proposition~\ref{GC-coercivity} and since we assumed that $\widetilde{E}_h(\widetilde{m}_h)$ is bounded, there exist $h_0>0$   \linebreak and $C_0>0$ such that, for every $h \in (0,h_0)$, $\widetilde{E}_h^{(0)}(\widetilde{m}_h) \leqslant C_0$. In particular, for every $h \in (0,h_0)$, we have
\begin{equation*}
\int_{\Omega_1} \left\vert \nabla' \widetilde{m}_h \right\vert^2 \mathrm{d} x
\leqslant \frac{h \left\vert \log h \right\vert}{d^2} C_0,
\end{equation*}
and
\begin{equation}
\label{eq-thm-compactness}
\int_{\Omega_1} \left\vert \partial_3 \widetilde{m}_h \right\vert^2 \mathrm{d} x
\leqslant h^2 \cdot \frac{h \left\vert \log h \right\vert}{d^2} C_0.
\end{equation}
In the regime~\eqref{regime1.1}+\eqref{regime1.2}, $\lim\limits_{h \rightarrow 0} \frac{h \left\vert \log h \right\vert}{d^2}=\alpha$ and $\lim\limits_{h \rightarrow 0} h^2 \cdot \frac{h \left\vert \log h \right\vert}{d^2} = 0$. Hence, $(\nabla'\widetilde{m}_h)_{h>0}$ is bounded in $L^2$. Moreover, $\left\vert \widetilde{m}_h \right\vert=1$ for every $h>0$, thus the sequence $(\widetilde{m}_h)_{h>0}$ is bounded in $H^1$. By the Banach-Alaoglu theorem~(\cite[Theorem~3.15]{Rudin}), for a subsequence, $(\widetilde{m}_{h})_{h>0}$ converges weakly to $\widetilde{m}_0$ in $H^1$, for some $\widetilde{m}_0 \in H^1(\Omega_1,\mathbb{R}^3)$.

It remains to show the stated properties of $\widetilde{m}_0$. By the Rellich-Kondrachov compactness theorem~(\cite[Section~5.7]{Evans}), up to take a further subsequence, we can assume that $(\widetilde{m}_{h})_{h>0}$ converges strongly to $\widetilde{m}_0$ in $L^2$ and almost everywhere in$\Omega_1$. In particular, $\left\vert \widetilde{m}_0 \right\vert = 1$ almost everywhere in~$\Omega_1$. By~\eqref{eq-thm-compactness}, $(\partial_3 \widetilde{m}_{h})_{h>0}$ tends to zero in $L^2$, but we also know that $(\partial_3\widetilde{m}_{h})_{h>0}$ tends to $\partial_3\widetilde{m}_0$ weakly in $L^2$. By uniqueness of the weak limit, we have $\partial_3\widetilde{m}_0 \equiv 0$ in $L^2$. It follows that $\widetilde{m}_0$ is independent of $x_3$. By Proposition~\ref{GC-coercivity} and since we assumed that $\widetilde{E}_h(\widetilde{m}_h)$ is bounded, there \linebreak exist $h_1>0$ and $C_1>0$ such that, for every $h \in (0,h_1)$,
\begin{equation*}
\frac{1}{h^2 \left\vert \log h \right\vert} \int_{\mathbb{R}^2} \left\vert \nabla u_h \right\vert^2 \mathrm{d} x \leqslant C_1.
\end{equation*}
Using Remark~\ref{rem_gh} and Theorem~\ref{devstrayfield}, it follows that, for every $h \in (0,h_1)$,
\begin{equation*}
0 \leqslant
\int_{\mathbb{R}^2} \vert \mathcal{F}(\overline{m}_h \cdot e_3)(\xi') \vert^2 g_h(\vert \xi' \vert ) \ \mathrm{d} \xi'
\leqslant \left( C_1 - \frac{I(h)}{4\pi h^2 \left\vert \log h \right\vert} \right) h \left\vert \log h \right\vert.
\end{equation*}
By~\eqref{rel_mhtilde_mh}, we have $\overline{\widetilde{m}}_h=\overline{m}_h$. By weak convergence of $(\widetilde{m}_h)_{h>0}$ to $\widetilde{m}_0$ in $H^1$, by Fubini's theorem, and since $\widetilde{m}_0$ is independent of $x_3$, then $(\overline{m}_h)_{h>0}$ converges weakly to $\widetilde{m}_0$ in $H^1$. Hence, we can use Theorem~\ref{limI} and we get
\begin{equation*}
\lim\limits_{h \rightarrow 0} \left( C_1-\frac{I(h)}{4\pi h^2 \left\vert \log h \right\vert} \right)
= C_1-\frac{1}{2\pi} \int_{\partial\omega} (\widetilde{m}_0 \cdot \nu)^2 \mathrm{d} \mathcal{H}^1 < +\infty,
\end{equation*}
since $\left\vert \widetilde{m}_0 \right\vert = 1$ almost everywhere, where $\nu$ is the outer unit normal vector on $\partial\omega$. We deduce that
\begin{equation*}
\lim\limits_{h \rightarrow 0} \int_{\mathbb{R}^2} \vert \mathcal{F}(\overline{m}_h \cdot e_3)(\xi') \vert^2 g_h(\vert \xi' \vert ) \ \mathrm{d} \xi' = 0.
\end{equation*}
Since $m_h \equiv 0$ in $\mathbb{R}^3 \setminus \Omega_1$, then for every $h>0$,
\begin{equation*}
\int_{\mathbb{R}^2} \vert \mathcal{F}(\overline{m}_h \cdot e_3)(\xi') \vert^2 g_h(\vert \xi' \vert ) \ \mathrm{d} \xi'
=
\int_{\mathbb{R}^2} \vert \mathcal{F}(\overline{\widetilde{m}}_{h,3} \mathds{1}_\omega)(\xi') \vert^2 g_h(\vert \xi' \vert ) \ \mathrm{d} \xi'.
\end{equation*}
Since $(g_h)_{h>0}$ converges almost everywhere to $1$ in $\mathbb{R}^2$ and $g_h(\vert \xi' \vert)$ is bounded for almost \linebreak every $\xi' \in \mathbb{R}^2$ (see Remark~\ref{rem_gh}), and $(\mathcal{F}(\overline{\widetilde{m}}_{h,3}))_{h>0}$ is bounded and converges almost everywhere to $\mathcal{F}(\overline{\widetilde{m}}_{0,3})$ in $\omega$, we deduce from the dominated convergence theorem and the two above relations that
\begin{equation*}
\int_{\mathbb{R}^2} \vert \mathcal{F}(\overline{\widetilde{m}}_{0,3} \mathds{1}_\omega)(\xi') \vert^2 \mathrm{d} \xi'
= 0.
\end{equation*}
By Plancherel's formula, we get
\begin{equation*}
\int_\omega \vert \overline{\widetilde{m}}_{0,3}(x') \vert^2 \mathrm{d} x'
= \int_{\mathbb{R}^2} \vert \mathcal{F}(\overline{\widetilde{m}}_{0,3} \mathds{1}_\omega)(\xi') \vert^2 \mathrm{d} \xi'
= 0.
\end{equation*}
We deduce that $\overline{\widetilde{m}}_{0,3} \equiv 0$ in $\omega$, but since $\widetilde{m}_0$ is independent of $x_3$, we firstly have $\widetilde{m}_{0,3}=\overline{\widetilde{m}}_{0,3} \equiv 0$ in $\omega$, and secondly $\widetilde{m}_{0,3} \equiv 0$ in $\Omega_1$.
\end{proof}

\begin{theoreme}[Lower bound]
\label{GC-lower-bound}
Consider the regime~\eqref{regime1.1}$+$\eqref{regime1.2}. Consider a sequence $(\widetilde{m}_h)_{h>0}$ in $H^1(\Omega_1,\mathbb{S}^2)$ and $\widetilde{m}_0 \in H^1(\Omega_1,\mathbb{S}^2)$ such that $\widetilde{m}_0$ is independent of $x_3$, $\widetilde{m}_{0,3} \equiv 0$ and $(\widetilde{m}_h)_{h>0}$ converges weakly to $\widetilde{m}_0$ in $H^1$. Then
\begin{equation}
\label{liminfEh}
\liminf\limits_{h \rightarrow 0} \widetilde{E}_h(\widetilde{m}_h) \geqslant \widetilde{E}_0(\widetilde{m}_0),
\end{equation}
where $\widetilde{E}_0$ is given by~\eqref{exprE0}.
\end{theoreme}

\begin{proof}
We denote by $\nu$ the outer unit normal vector on $\partial \Omega_1$.
\\
If $\liminf\limits_{h \rightarrow 0} \widetilde{E}_h(\widetilde{m}_h)=+\infty$, then the expected inequality is obvious. Assume that there exists a constant $C>0$ such that $\liminf\limits_{h \rightarrow 0} \widetilde{E}_h(\widetilde{m}_h) \leqslant C$. Let $h>0$. By~\eqref{exprEh}, we clearly have
\begin{equation}
\label{equa-lower-bound-0}
\begin{split}
\widetilde{E}_h(\widetilde{m}_h)
& \geqslant \frac{d^2}{h \left\vert \log h \right\vert} \int_{\Omega_1} \vert \nabla'\widetilde{m}_h \vert^2 \mathrm{d} x
\\ & \quad
+ \frac{1}{h \left\vert \log h \right\vert} \int_{\Omega_1} \widehat{D}' : \nabla' \widetilde{m}_h \wedge \widetilde{m}_h \ \mathrm{d} x
\\ & \quad
+ \frac{1}{h^2 \left\vert \log h \right\vert} \int_{\Omega_1} \widehat{D}_3 \cdot \partial_3 \widetilde{m}_h \wedge \widetilde{m}_h \ \mathrm{d} x
\\ & \quad
+ \frac{1}{h^2 \left\vert \log h \right\vert} \int_{\mathbb{R}^3} \left\vert \nabla u_h \right\vert^2 \mathrm{d} x
\\ & \quad + \frac{Q}{h \left\vert \log h \right\vert} \int_{\Omega_1} \Phi(\widetilde{m}_h) \ \mathrm{d} x
- 2 \int_{\Omega_1} \frac{\widetilde{H}_{\mathrm{ext},h}}{h \left\vert \log h \right\vert} \cdot \widetilde{m}_h \ \mathrm{d} x.
\end{split}
\end{equation}
Let us examine each term of this inequality in order to prove~\eqref{liminfEh}.
Since $\frac{d^2}{h \left\vert \log h \right\vert} \rightarrow \alpha$ and $(\nabla \widetilde{m}_h)_{h>0}$ converges weakly to $\nabla \widetilde{m}_0$ in $L^2$, then by weak lower semicontinuity of the Dirichlet integral, we have
\begin{equation}
\label{equa-lower-bound-1}
\liminf\limits_{h \rightarrow 0}
\frac{d^2}{h \left\vert \log h \right\vert}
\int_{\Omega_1} \left\vert \nabla' \widetilde{m}_h \right\vert^2 \mathrm{d} x
\geqslant \alpha \int_{\Omega_1} \left\vert \nabla' \widetilde{m}_0 \right\vert^2 \mathrm{d} x.
\end{equation}
Recall that
\begin{align*}
\frac{1}{h \left\vert \log h \right\vert} \int_{\Omega_1} \widehat{D}' : \nabla' \widetilde{m}_h \wedge \widetilde{m}_h \ \mathrm{d} x
& = \frac{1}{h \left\vert \log h \right\vert} \sum_{j=1}^2 \sum_{k=1}^3 \int_{\Omega_1} \widehat{D}_{jk} e_k \cdot \partial_j \widetilde{m}_h \wedge \widetilde{m}_h \ \mathrm{d} x
\\
& = \sum_{j,k=1}^2 \frac{\widehat{D}_{jk}}{h \left\vert \log h \right\vert} \ e_k \cdot \int_{\Omega_1} \partial_j \widetilde{m}_h \wedge \widetilde{m}_h \ \mathrm{d} x
\\
& \quad
+ \frac{d^2}{h \left\vert \log h \right\vert} \frac{\widehat{D}_{13}}{d^2} \ e_3 \cdot \int_{\Omega_1} \partial_1 \widetilde{m}_h \wedge \widetilde{m}_h \ \mathrm{d} x
\\
& \quad
+ \frac{d^2}{h \left\vert \log h \right\vert} \frac{\widehat{D}_{23}}{d^2} \ e_3 \cdot \int_{\Omega_1} \partial_2 \widetilde{m}_h \wedge \widetilde{m}_h \ \mathrm{d} x.
\end{align*}
Since $(\widetilde{m}_h)_{h>0}$ converges weakly to $\widetilde{m}_0$ in $H^1$, then up to a subsequence (thanks to the Rellich-Kondrachov compactness theorem~\cite[Section~5.7]{Evans}), we can assume that $(\widetilde{m}_h)_{h>0}$ converges strongly to $\widetilde{m}_0$ in $L^2$. As a consequence, we deduce that, for $j \in \left\lbrace 1,2 \right\rbrace$,
\begin{equation*}
\lim\limits_{h \rightarrow 0} \int_{\Omega_1} \partial_j \widetilde{m}_h \wedge \widetilde{m}_h \ \mathrm{d} x = \int_{\Omega_1} \partial_j \widetilde{m}_0 \wedge \widetilde{m}_0 \ \mathrm{d} x.
\end{equation*}
Combining this with the assumptions~\eqref{regime1.1}$+$\eqref{regime1.2}, we deduce that
\begin{equation}
\label{equa-lower-bound-2}
\lim\limits_{h \rightarrow 0}
\frac{1}{h \left\vert \log h \right\vert} \int_{\Omega_1} \widehat{D}' : \nabla' \widetilde{m}_h \wedge \widetilde{m}_h \ \mathrm{d} x
= 2 \alpha \int_{\Omega_1} \delta \cdot \nabla' \widetilde{m}'_0 \wedge \widetilde{m}'_0 \ \mathrm{d} x.
\end{equation}
Similarly, as $\left( \frac{1}{h}\partial_3 \widetilde{m}_h \right)_{h>0}$ converges weakly to some $M$ in $L^2$,
\begin{equation*}
\lim\limits_{h \rightarrow 0} \frac{1}{h} \int_{\Omega_1} \partial_3 \widetilde{m}_h \wedge \widetilde{m}_h \ \mathrm{d} x = \int_{\Omega_1} M \wedge \widetilde{m}_0 \ \mathrm{d} x.
\end{equation*}
But $\frac{1}{h \left\vert \log h \right\vert} \sum_{k=1}^3 \vert \widehat{D}_{3k} \vert \ll 1$, hence
\begin{equation}
\label{equa-lower-bound-3}
\lim\limits_{h \rightarrow 0}
\frac{1}{h^2 \left\vert \log h \right\vert} \int_{\Omega_1} \widehat{D}_3 \cdot \partial_3 \widetilde{m}_h \wedge \widetilde{m}_h \ \mathrm{d} x
= 0.
\end{equation}
By Theorem~\ref{devstrayfield} and since $g_h \geqslant 0$ (see Remark~\ref{rem_gh}),
\begin{equation*}
\frac{1}{h^2 \left\vert \log h \right\vert} \int_{\mathbb{R}^3} \left\vert \nabla u_h \right\vert^2 \mathrm{d} x
\geqslant \frac{I(h)}{4\pi h^2 \left\vert \log h \right\vert} +o_h(1).
\end{equation*}
By~\eqref{rel_mhtilde_mh}, we have $\overline{\widetilde{m}}_h=\overline{m}_h$. By weak convergence of $(\widetilde{m}_h)_{h>0}$ to $\widetilde{m}_0$ in $H^1$, by Fubini's theorem, and since $\widetilde{m}_0$ is independent of $x_3$, then $(\overline{m}_h)_{h>0}$ converges weakly to $\widetilde{m}_0$ in $H^1$. Hence, we can use Theorem~\ref{limI}, from which it follows
\begin{equation}
\label{equa-lower-bound-4}
\begin{split}
\lim\limits_{h \rightarrow 0} \frac{1}{h^2 \left\vert \log h \right\vert} \int_{\mathbb{R}^3} \left\vert \nabla u_h \right\vert^2 \mathrm{d} x
& \geqslant \frac{1}{2\pi} \int_{\partial \omega} ( \widetilde{m}_0 \cdot \nu )^2 \mathrm{d} \mathcal{H}^1.
\end{split}
\end{equation}
Since $(\widetilde{m}_h)_{h>0}$ converges (up to a subsequence) almost everywhere to $\widetilde{m}_0$ in $\Omega_1$ and $\Phi$ is continuous in $\mathbb{S}^2$, then $(\Phi(\widetilde{m}_h))_{h>0}$ converges almost everywhere to $\Phi(\widetilde{m}_0)$ and is bounded (because $\mathbb{S}^2$ is compact). Thus, by the dominated convergence theorem and since $\frac{Q}{h \left\vert \log h \right\vert} \rightarrow \beta$,
\begin{equation}
\label{equa-lower-bound-5}
\lim\limits_{h \rightarrow 0} \frac{Q}{h \left\vert \log h \right\vert} \int_{\Omega_1} \Phi(\widetilde{m}_h) \ \mathrm{d} x
= \beta \int_{\Omega_1} \Phi(\widetilde{m}_0) \ \mathrm{d} x.
\end{equation}
Finally, we have
\begin{equation*}
\begin{split}
\left\vert \int_{\Omega_1} \left( \frac{\widetilde{H}_{\mathrm{ext},h}}{h \left\vert \log h \right\vert} \cdot \widetilde{m}_h - \gamma \widetilde{H}_{\mathrm{ext},0} \cdot \widetilde{m}_0 \right) \mathrm{d} x \right\vert
& \leqslant
\left\vert \int_{\Omega_1} \left( \frac{\widetilde{H}_{\mathrm{ext},h}}{h \left\vert \log h \right\vert} - \gamma \widetilde{H}_{\mathrm{ext},0} \right) \cdot \widetilde{m}_h \ \mathrm{d} x \right\vert
\\
& \quad +
\left\vert \int_{\Omega_1} \gamma \widetilde{H}_{\mathrm{ext},0} \cdot \left( \widetilde{m}_h - \widetilde{m}_0 \right) \mathrm{d} x \right\vert.
\end{split}
\end{equation*}
The first term in the right-hand side above tends to zero by H\"older's inequality, since $\left( \frac{\widetilde{H}_{\mathrm{ext},h}}{h \left\vert \log h \right\vert} \right)_{h>0}$ converges to $\gamma \widetilde{H}_{\mathrm{ext},0}$ in $L^1$ and $\left\Vert \widetilde{m}_h \right\Vert_{L^\infty} = 1$. The second term also tends to zero, by dominated convergence theorem: indeed, up to a subsequence, $\widetilde{m}_h \rightarrow \widetilde{m}_0$ almost everywhere in $\Omega_1$, \linebreak and $\vert \gamma \widetilde{H}_{\mathrm{ext},0} \cdot \left( \widetilde{m}_h-\widetilde{m}_0 \right) \vert \leqslant C \Vert \widetilde{H}_{\mathrm{ext},0} \Vert_{L^\infty} \leqslant C$. We deduce that
\begin{equation}
\label{equa-lower-bound-6}
\begin{split}
\lim\limits_{h \rightarrow 0} \int_{\Omega_1} \frac{\widetilde{H}_{\mathrm{ext},h}}{h \left\vert \log h \right\vert} \cdot \widetilde{m}_h \ \mathrm{d} x
& = \gamma \int_{\Omega_1} \widetilde{H}_{\mathrm{ext},0} \cdot \widetilde{m}_0 \ \mathrm{d} x.
\end{split}
\end{equation}
Taking the $\liminf$ in~\eqref{equa-lower-bound-0} and using~\eqref{equa-lower-bound-1}, \eqref{equa-lower-bound-2}, \eqref{equa-lower-bound-3}, \eqref{equa-lower-bound-4}, \eqref{equa-lower-bound-5} and~\eqref{equa-lower-bound-6}, we get~\eqref{liminfEh} as expected.
\end{proof}

\begin{theoreme}[Upper bound]
\label{GC-upper-bound}
Consider the regime~\eqref{regime1.1}$+$\eqref{regime1.2}. Consider a map $\widetilde{m}_0 \in H^1(\Omega_1,\mathbb{S}^1)$ such that $\widetilde{m}_0$ is independent of $x_3$ and $\widetilde{m}_{0,3} \equiv 0$. Then there exists a sequence $(\widetilde{m}_h)_{h>0}$ \linebreak in $H^1(\Omega_1,\mathbb{S}^1)$ such that $(\widetilde{m}_h)_{h>0}$ converges strongly to $\widetilde{m}_0$ in $H^1$ and satisfies
\begin{equation}
\label{limsupEh}
\lim\limits_{h \rightarrow 0} \widetilde{E}_h(\widetilde{m}_h)=\widetilde{E}_0(\widetilde{m}_0),
\end{equation}
where $\widetilde{E}_0$ is given by~\eqref{exprE0}.
\end{theoreme}

\begin{proof}
We denote by $\nu$ the outer unit normal vector on $\partial \Omega_1$.
\\
We consider the constant sequence $(\widetilde{m}_h)_{h>0}=(\widetilde{m}_0)_{h>0}$. By Theorem~\ref{devstrayfield} and Theorem~\ref{limI}, we have in this case
\begin{equation*}
\lim\limits_{h \rightarrow 0} \frac{1}{h^2 \left\vert \log h \right\vert} \int_{\mathbb{R}^3} \left\vert \nabla u_h \right\vert^2 \mathrm{d} x
= \frac{1}{2\pi} \int_{\partial \omega} ( \overline{\widetilde{m}}_0 \cdot \nu )^2 \ \mathrm{d} \mathcal{H}^1
= \frac{1}{2\pi} \int_{\partial \omega} ( \widetilde{m}_0 \cdot \nu )^2 \ \mathrm{d} \mathcal{H}^1,
\end{equation*}
since $\widetilde{m}_0$ is independent of $x_3$ and $\widetilde{m}_{0,3} \equiv 0$. Using the decomposition~\eqref{decomposition-coercivity}, we have, in the regime~\eqref{regime1.1}+\eqref{regime1.2},
\begin{equation}
\label{eq-upper-bound-1}
\lim\limits_{h \rightarrow 0} \widetilde{E}_h^{(0)}(\widetilde{m}_0)
= \alpha \int_{\Omega_1} \left\vert \nabla' \widetilde{m}_0 \right\vert^2 \mathrm{d} x
+ \frac{1}{2\pi} \int_{\partial \omega} \left( \widetilde{m}_0 \cdot \nu \right)^2 \mathrm{d} \mathcal{H}^1
+ \beta \int_{\Omega_1} \Phi(\widetilde{m}_0) \ \mathrm{d} x,
\end{equation}
by the above convergence result and using again that $\widetilde{m}_0$ is independent of $x_3$. We also have
\begin{equation}
\label{eq-upper-bound-2}
\lim\limits_{h \rightarrow 0} \widetilde{E}_h^{(1)}(\widetilde{m}_0)
 = -2 \gamma \int_{\Omega_1} \widetilde{H}_{\mathrm{ext},0} \cdot \widetilde{m}_0 \ \mathrm{d} x,
\end{equation}
since by H\"older's inequality,
\begin{equation*}
\left\vert \int_{\Omega_1} \frac{\widetilde{H}_{\mathrm{ext},h}}{h \left\vert \log h \right\vert} \cdot \widetilde{m}_0 \ \mathrm{d} x
- \int_{\Omega_1} \gamma \widetilde{H}_{\mathrm{ext},0} \cdot \widetilde{m}_0 \ \mathrm{d} x \right\vert
\leqslant \left\Vert \frac{\widetilde{H}_{\mathrm{ext},h}}{h \left\vert \log h \right\vert} - \gamma \widetilde{H}_{\mathrm{ext},0} \right\Vert_{L^1(\Omega_1)}
\end{equation*}
and $\left( \frac{\widetilde{H}_{\mathrm{ext},h}}{h \left\vert \log h \right\vert} \right)_{h>0}$ converges to $\gamma \widetilde{H}_{\mathrm{ext},0}$ in $L^1$. Furthermore, using that $\widetilde{m}_0$ is independent of $x_3$, we get
\begin{align*}
\widetilde{E}_h^{(2)}(\widetilde{m}_0)
& = \frac{1}{h \left\vert \log h \right\vert} \int_{\Omega_1} \widehat{D}' : \nabla' \widetilde{m}_0 \wedge \widetilde{m}_0 \ \mathrm{d} x
\\
& = \frac{1}{h \left\vert \log h \right\vert} \int_{\Omega_1} \left(
\widehat{D}_1 \cdot \partial_1 \widetilde{m}_0 \wedge \widetilde{m}_0
+ \widehat{D}_2 \cdot \partial_2 \widetilde{m}_0 \wedge \widetilde{m}_0
\right) \mathrm{d} x,
\end{align*}
and since $\widetilde{m}_{0,3} \equiv 0$ (and thus $\partial_j \widetilde{m}_{0,3} \equiv 0$ for $j \in \left\lbrace 1,2 \right\rbrace$),
\begin{align*}
\widetilde{E}_h^{(2)}(\widetilde{m}_0)
& = \frac{1}{h \left\vert \log h \right\vert} \int_{\Omega_1} 
\left(
\widehat{D}_{13} e_3 \cdot \partial_1 \widetilde{m}_0 \wedge \widetilde{m}_0
+ \widehat{D}_{23} e_3 \cdot \partial_2 \widetilde{m}_0 \wedge \widetilde{m}_0
\right) \mathrm{d} x
\\
& = \frac{d^2}{h \left\vert \log h \right\vert} \frac{\widehat{D}_{13}}{d^2} \int_{\Omega_1} \partial_1 \widetilde{m}'_0 \wedge \widetilde{m}'_0 \ \mathrm{d} x
+ \frac{d^2}{h \left\vert \log h \right\vert} \frac{\widehat{D}_{23}}{d^2} \int_{\Omega_1} \partial_2 \widetilde{m}'_0 \wedge \widetilde{m}'_0 \ \mathrm{d} x.
\end{align*}
Then, in the regime~\eqref{regime1.1}+\eqref{regime1.2}, we have
\begin{equation}
\label{eq-upper-bound-3}
\lim\limits_{h \rightarrow 0} \widetilde{E}_h^{(2)}(\widetilde{m}_0)
= 2 \alpha\delta_1 \int_{\Omega_1} \partial_1 \widetilde{m}'_0 \wedge \widetilde{m}'_0 \ \mathrm{d} x
+ 2 \alpha\delta_2 \int_{\Omega_1} \partial_2 \widetilde{m}'_0 \wedge \widetilde{m}'_0 \ \mathrm{d} x.
\end{equation}
Combining~\eqref{eq-upper-bound-1}, \eqref{eq-upper-bound-2} and~\eqref{eq-upper-bound-3}, we get~\eqref{limsupEh}.
\end{proof}

Theorem~\ref{GC} is now a direct consequence from Theorems~\ref{GC-compactness}, \ref{GC-lower-bound} and \ref{GC-upper-bound}.

Corollary~\ref{GC-cor} is a consequence from Proposition \ref{GC-coercivity} and the direct method in the calculus of variations on the one hand, and from Theorem \ref{GC} and properties of Gamma-limits (see \cite[Proposition~7.8]{DalMaso}) on the other hand.

\section{On the local minimizers of the Gamma-limit of the micromagnetic energy in the upper-half plane.}
\label{section2}

This section is devoted to look for local minimizers of the Gamma-limit $\widetilde{E}_0$ given in~\eqref{exprE0-2D}. We assume here that the anisotropy $\Phi$ and the external magnetic field $\widetilde{H}_{\mathrm{ext},0}$ are equal to zero.

\noindent
\textbf{Moreover, since all quantities in this section are two-dimensional quantities, we drop the primes $'$ in the notations.}

The energy that we will study in this section, resulting from~\eqref{exprE0-2D}, is
\begin{equation*}
\widetilde{E}_0(m;\omega \cap \mathbb{R}_+^2)
:= \alpha \left[
\int_{\omega \cap \mathbb{R}_+^2} \left\vert \nabla m \right\vert^2 \mathrm{d} x
+ 2 \int_{\omega \cap \mathbb{R}_+^2} \delta \cdot \nabla m \wedge m \ \mathrm{d} x
\right]
+ \frac{1}{2\pi} \int_{\omega \cap (\mathbb{R} \times \left\lbrace 0 \right\rbrace)} (m \cdot \nu)^2 \ \mathrm{d} \mathcal{H}^1,
\end{equation*}
for every $m \in H^1(\omega,\mathbb{S}^1)$, where $\omega$ is a smooth bounded open subset of $\mathbb{R}^2$ and $\nu$ is the outer unit normal vector on $\partial \omega$.

\subsection{The energy in the upper-half plane and its critical points}
\label{subsection2.1}

By making a blow-up near the boundary $\partial \omega$, we are led to consider localized functionals with the integrals defined on sets of the form $\omega \cap \mathbb{R}_+^2$, where $\omega$ is a smooth bounded open subset of $\mathbb{R}^2$.
\linebreak
Let $\omega \subset \mathbb{R}^2$ be such a set. For $m \in H^1(\omega,\mathbb{S}^1)$, there exists~(see~\cite{BBM}) a lifting $\varphi \in H^1(\omega,\mathbb{R})$ of $m$, i.e. $m=e^{i\varphi}$. Note that $\nu=-e_2$, where $e_2$ is the second unit vector of the standard orthonormal basis in $\mathbb{R}^2$. Since
\begin{equation*}
\nabla m = \nabla \left( e^{i\varphi} \right) = i\nabla \varphi e^{i\varphi},
\end{equation*}
\begin{equation*}
\nabla m \wedge m = \Im \left( -i\nabla \varphi e^{-i\varphi} e^{i\varphi} \right) = -\nabla \varphi,
\end{equation*}
and
\begin{equation*}
m \cdot \nu = -m_2=-\sin \varphi,
\end{equation*}
we can introduce
\begin{equation*}
\widetilde{E}_0(m ;\omega \cap \mathbb{R}_+^2)
:= \alpha \int_{\omega \cap \mathbb{R}_+^2} \left( \left\vert \nabla \varphi \right\vert^2 -2 \delta \cdot \nabla \varphi \right) \mathrm{d} x
+ \frac{1}{2\pi} \int_{\omega \cap (\mathbb{R} \times \left\lbrace 0 \right\rbrace)} \sin^2 \varphi \ \mathrm{d} \mathcal{H}^1.
\end{equation*}
Setting $\varepsilon :=2\pi \alpha$ and $E_{\varepsilon}^\delta(\varphi ;\omega):=\frac{1}{2\alpha}\widetilde{E}_0(m ;\omega \cap \mathbb{R}_+^2)$, we get
\begin{equation}
\label{exprEeps}
E_{\varepsilon}^\delta(\varphi;\omega)
= \frac{1}{2} \int_{\omega \cap \mathbb{R}_+^2} \left( \left\vert \nabla \varphi \right\vert^2-2\delta \cdot \nabla \varphi \right) \mathrm{d} x
+ \frac{1}{2\varepsilon} \int_{\omega \cap (\mathbb{R} \times \left\lbrace 0 \right\rbrace)} \sin^2 \varphi \ \mathrm{d} \mathcal{H}^1.
\end{equation}

In the case $\delta=0$, the energy $E_{\varepsilon}^\delta=E_{\varepsilon}^0$ has been deeply studied by Kurzke~\cite{Kurzke06} and Ignat-Kurzke~\cite{IK21}.

\begin{proposition}
\label{propformulationvariationnelle}
If $\varphi \in H^1_\mathrm{loc}(\overline{\mathbb{R}_+^2})$ is a critical point of $E_{\varepsilon}^\delta$, then
\begin{equation}
\label{formulationvariationnellephir}
\int_{\mathbb{R}_+^2 \cap \mathrm{Supp}(\psi)} (\nabla \varphi-\delta) \cdot \nabla \psi \ \mathrm{d} x
+ \frac{1}{2\varepsilon} \int_{(\mathbb{R} \times \left\lbrace 0 \right\rbrace) \cap \mathrm{Supp}(\psi)} \sin(2\varphi)\psi \ \mathrm{d} \mathcal{H}^1
= 0,
\end{equation}
for every $\psi \in H^1(\mathbb{R}^2)$ with compact support.
\end{proposition}

\begin{proof}
Let $\varphi \in H^1_\mathrm{loc}(\overline{\mathbb{R}_+^2})$, $\psi \in C^1(\mathbb{R}^2)$ with compact support, and $t \in \mathbb{R}$. We have
\begin{equation*}
\begin{split}
E_{\varepsilon}^\delta(\varphi+t\psi;\mathrm{Supp}(\psi))
& = \frac{1}{2} \int_{\mathbb{R}_+^2 \cap \mathrm{Supp}(\psi)}
\left( 
\left\vert \nabla \varphi + t \nabla \psi \right\vert^2
- 2 \delta \cdot \nabla (\varphi +t \psi)
\right)
\mathrm{d} x
\\
& \quad + \frac{1}{2\varepsilon} \int_{(\mathbb{R} \times \left\lbrace 0 \right\rbrace) \cap \mathrm{Supp}(\psi)}
\sin^2(\varphi+t\psi)
\ \mathrm{d} \mathcal{H}^1.
\end{split}
\end{equation*}
On the one hand, in $\mathbb{R}_+^2 \cap \mathrm{Supp}(\psi)$,
\begin{equation*}
\left\vert \nabla \varphi + t \nabla \psi \right\vert^2
- 2 \delta \cdot \nabla (\varphi +t \psi)
= \left\vert \nabla \varphi \right\vert^2 +2t(\nabla \varphi-\delta)\cdot \nabla \psi +O(t^2) \ \ \text{ as } t \rightarrow 0.
\end{equation*}
On the other hand, in $\mathbb{R} \times \left\lbrace 0 \right\rbrace$,
\begin{equation*}
\sin^2(\varphi+t\psi)
= \sin^2(\varphi) + t \psi \sin(2\varphi) + O(t^2)
\ \ \text{ as } t \rightarrow 0.
\end{equation*}
Hence, as $t \rightarrow 0$,
\begin{align*}
E_\varepsilon^\delta(\varphi+t\psi;\mathrm{Supp}(\psi))
= E_\varepsilon^\delta(\varphi;\mathrm{Supp}(\psi))
& + t\int_{\mathbb{R}_+^2 \cap \mathrm{Supp}(\psi)} (\nabla \varphi-\delta) \cdot \nabla \psi \ \mathrm{d} x
\\
& 
+ \frac{t}{2\varepsilon} \int_{(\mathbb{R} \times \left\lbrace 0 \right\rbrace) \cap \mathrm{Supp}(\psi)} \sin(2\varphi)\psi \ \mathrm{d} \mathcal{H}^1
+ O(t^2),
\end{align*}
and we deduce that
\begin{align*}
\left. \frac{\mathrm{d}}{\mathrm{d} t} \right\vert_{t=0} E_\varepsilon^\delta(\varphi + t \psi;\mathrm{Supp}(\psi))
& =
\int_{\mathbb{R}_+^2 \cap \mathrm{Supp}(\psi)} (\nabla \varphi-\delta) \cdot \nabla \psi \ \mathrm{d} x
+ \frac{1}{2\varepsilon} \int_{(\mathbb{R} \times \left\lbrace 0 \right\rbrace) \cap \mathrm{Supp}(\psi)} \sin(2\varphi)\psi \ \mathrm{d} \mathcal{H}^1.
\end{align*}
By density and by Definition~\ref{def_pt_critique}, we deduce~\eqref{formulationvariationnellephir}.
\end{proof}

\begin{proposition}
\label{propsystFeps}
Any critical point $\varphi$ of $E_{\varepsilon}^\delta$ belongs to $C^\infty(\overline{\mathbb{R}_+^2})$ and satisfies
\begin{equation}
\label{systFeps}
\left\lbrace \begin{array}{rcll} \Delta \varphi&=&0& \text{ in } \mathbb{R}_+^2, \\ \partial_2\varphi &=&\frac{1}{2\varepsilon} \sin 2\varphi + \delta_2 & \text{ on } \mathbb{R} \times \left\lbrace 0 \right\rbrace.
\end{array} \right.
\end{equation}
\end{proposition}

\begin{proof}
Let $\varphi \in H^1_\mathrm{loc}(\overline{\mathbb{R}_+^2})$ be a critical point of $E_\varepsilon^\delta$.

\textit{Step 1:} We begin with proving~\eqref{systFeps} and we first assume that $\varphi \in H^2_\mathrm{loc}(\overline{\mathbb{R}_+^2})$.

For every $\psi \in H^1(\mathbb{R}^2)$ with compact support, we have, thanks to~\eqref{formulationvariationnellephir},
\begin{equation*}
\int_{\mathbb{R}_+^2 \cap \mathrm{Supp}(\psi)} (\nabla \varphi-\delta) \cdot \nabla \psi \ \mathrm{d} x
+ \frac{1}{2\varepsilon} \int_{(\mathbb{R} \times \left\lbrace 0 \right\rbrace) \cap \mathrm{Supp}(\psi)} \sin( 2\varphi)\psi \ \mathrm{d} \mathcal{H}^1
= 0.
\end{equation*}
Integrating by parts, we get
\begin{equation*}
\begin{split}
& -\int_{\mathbb{R}_+^2 \cap \mathrm{Supp}(\psi)} (\Delta \varphi)\psi \ \mathrm{d} x
+ \int_{\partial(\mathbb{R}_+^2 \cap \mathrm{Supp}(\psi))} \left( \partial_\nu\varphi -\delta\cdot \nu \right) \psi \ \mathrm{d} \mathcal{H}^1
\\
& \qquad\qquad\qquad\qquad\qquad\qquad\qquad + \frac{1}{2\varepsilon} \int_{(\mathbb{R} \times \left\lbrace 0 \right\rbrace) \cap \mathrm{Supp}(\psi)}  \sin(2\varphi)\psi \ \mathrm{d} \mathcal{H}^1
= 0.
\end{split}
\end{equation*}
But $\psi \equiv 0$ on $\partial(\mathrm{Supp}(\psi))$ and $\nu=-e_2$ on $(\mathbb{R} \times \left\lbrace 0 \right\rbrace) \cap \mathrm{Supp}(\psi)$, thus
\begin{equation*}
-\int_{\mathbb{R}_+^2 \cap \mathrm{Supp}(\psi)} (\Delta \varphi)\psi \ \mathrm{d} x
+ \int_{(\mathbb{R} \times \left\lbrace 0 \right\rbrace) \cap \mathrm{Supp}(\psi)}  \left( -\partial_2\varphi +\delta_2 + \frac{1}{2\varepsilon} \sin(2\varphi) \right)\psi \ \mathrm{d} \mathcal{H}^1
= 0.
\end{equation*}
We can choose $\psi$ such that $\psi \equiv 0$ on $\mathbb{R} \times \left\lbrace 0 \right\rbrace$, so that $\Delta \varphi = 0$ in $\mathbb{R}_+^2 \cap \mathrm{Supp}(\psi)$. We then deduce that $\partial_2\varphi = \frac{1}{2\varepsilon}\sin(2\varphi)+\delta_2$ on $(\mathbb{R} \times \left\lbrace 0 \right\rbrace) \cap \mathrm{Supp}(\psi)$. This equalities being true for every $\psi \in H^1(\mathbb{R}^2)$ with compact support, we deduce~\eqref{systFeps}.

\textit{Step 2:} Let us prove that $\varphi \in H^2_\mathrm{loc}(\overline{\mathbb{R}_+^2})$.

For the interior regularity, by~\eqref{systFeps} we have $\Delta\varphi=0$ in the distributional sense in $\mathbb{R}_+^2$. By Weyl's lemma, $\varphi \in C^\infty(\mathbb{R}_+^2)$.

For the boundary regularity, we introduce tangential difference quotients as defined in \cite[Section~5.8.2]{Evans} (see also \cite[Section~7.11]{GT01}), and we use the ideas in \cite[Section~6.3]{Evans}. Our strategy consists in proving that $\varphi$ is $H^2$ around $(x_1,x_2)=(0,0)$, and then around any point $(x_1,0)$ \linebreak with $x_1 \in \mathbb{R}$, by translation.

For $\left\vert h \right\vert >0$ and $(x_1,x_2) \in \overline{\mathbb{R}_+^2}= \mathbb{R} \times [0,+\infty)$, set
\begin{equation*}
\Delta_h\varphi(x_1,x_2) := \frac{1}{h} \left( \varphi(x_1+h,x_2)-\varphi(x_1,x_2) \right).
\end{equation*}
Let $r \in (0,1)$ be fixed. Let $\zeta \in C^\infty(\overline{B_r})$ be a function that satisfies $0 \leqslant \zeta \leqslant 1$, $\zeta \equiv 1$ in $\overline{B_{r/2}}$, $\zeta \equiv 0$ in $\overline{B_r} \setminus \overline{B_{3r/4}}$ and $\left\vert \nabla \zeta \right\vert \leqslant \frac{C}{r}$ for some constant $C>0$.
For $\left\vert h \right\vert >0$, set $\psi := \Delta_{-h}(\zeta^2 \Delta_h\varphi)$, which is in $H^1(\mathbb{R}^2)$ with compact support. Thus, we can input this function $\psi$ in~\eqref{formulationvariationnellephir}. By use of the integration by parts property of difference quotients (see equation~(16) in \cite[Section~6.3]{Evans}), we get
\begin{equation*}
\int_{B_r^+} \sum_{j=1}^2
\Delta_h(\partial_j\varphi-\delta_j) \partial_j(\zeta^2\Delta_h\varphi) \mathrm{d} x
+ \frac{1}{2\varepsilon} \int_{(-r,r) \times \left\lbrace 0 \right\rbrace}
\zeta^2(\Delta_h\varphi) \Delta_h(\sin(2\varphi)) \mathrm{d} \mathcal{H}^1
= 0.
\end{equation*}
Since $\delta_1,\delta_2$ are constants and expanding the derivative of $\zeta^2\Delta_h\varphi$, the above relation can be written as $I_1+I_2+I_3=0$ with
\begin{equation*}
I_1 := \int_{B_r^+} \zeta^2 \left\vert \Delta_h(\nabla\varphi) \right\vert^2 \mathrm{d} x,
\end{equation*}
\begin{equation*}
I_2 := 2 \int_{B_r^+} \sum_{j=1}^2 \zeta (\partial_j\zeta) (\Delta_h\varphi) \partial_j(\Delta_h\varphi) \mathrm{d} x,
\end{equation*}
and
\begin{equation*}
I_3 := \frac{1}{2\varepsilon} \int_{(-r,r) \times \left\lbrace 0 \right\rbrace}
\zeta^2(\Delta_h\varphi) \Delta_h(\sin(2\varphi)) \mathrm{d} \mathcal{H}^1.
\end{equation*}
Let $\varepsilon_1>0$, that will be precised later, and $j \in \left\lbrace 1,2 \right\rbrace$.  Using the properties of $\zeta$ and Young's inequality, we have
\begin{align*}
\left\vert 2 \int_{B_r^+} \zeta (\partial_j\zeta) (\Delta_h\varphi) \partial_j(\Delta_h\varphi) \mathrm{d} x \right\vert
& =
\left\vert 2 \int_{B_{3r/4}^+} \zeta (\partial_j\zeta) (\Delta_h\varphi) \partial_j(\Delta_h\varphi) \mathrm{d} x \right\vert
\\
& \leqslant
\varepsilon_1 \int_{B_r^+} \zeta^2 \left\vert \Delta_h(\partial_j\varphi) \right\vert^2 \mathrm{d} x
\\
& \quad + \frac{1}{\varepsilon_1}\int_{B_{3r/4}^+} \left\vert \partial_j\zeta \right\vert^2 \left\vert \Delta_h \varphi \right\vert^2 \mathrm{d} x.
\end{align*}
In particular, using the properties of $\zeta$ and \cite[Section~5.8.2, Theorem~3(i)]{Evans}, we have
\begin{align*}
\int_{B_{3r/4}^+} \left\vert \partial_j\zeta \right\vert^2 \left\vert \Delta_h \varphi \right\vert^2 \mathrm{d} x
& \leqslant \frac{C}{r^2} \int_{B_{3r/4}^+} \left\vert \Delta_h \varphi \right\vert^2 \mathrm{d} x
\leqslant \frac{C}{r^2} \int_{B_r^+} \left\vert \nabla \varphi \right\vert^2 \mathrm{d} x.
\end{align*}
We deduce that
\begin{equation}
\label{diffq_estimI2}
\left\vert I_2 \right\vert
\leqslant
\varepsilon_1 \int_{B_r^+} \zeta^2 \left\vert \Delta_h(\nabla\varphi) \right\vert^2 \mathrm{d} x
+
\frac{C}{\varepsilon_1 r^2} \int_{B_r^+} \left\vert \nabla \varphi \right\vert^2 \mathrm{d} x.
\end{equation}
By the mean-value theorem, $\left\vert \Delta_h(\sin(2\varphi)) \right\vert \leqslant 2 \left\vert \Delta_h \varphi \right\vert$. Combining this with Green's formula,
\begin{equation*}
\left\vert I_3 \right\vert
\leqslant \frac{1}{\varepsilon} \int_{(-r,r) \times \left\lbrace 0 \right\rbrace} \zeta^2 \left\vert \Delta_h \varphi \right\vert^2 \mathrm{d} \mathcal{H}^1
= \frac{1}{\varepsilon} \int_{B_r^+} -\partial_2 \left( (\zeta \Delta_h\varphi)^2 \right) \mathrm{d} x
= -\frac{2}{\varepsilon} \int_{B_r^+} \zeta \Delta_h\varphi \partial_2 \left( \zeta \Delta_h\varphi \right) \mathrm{d} x.
\end{equation*}
We deduce that
\begin{equation*}
\left\vert I_3 \right\vert
\leqslant \frac{2}{\varepsilon} \int_{B_r^+} \zeta \left\vert \Delta_h\varphi \right\vert \left\vert \nabla(\zeta\Delta_h\varphi) \right\vert \mathrm{d} x.
\end{equation*}
Let $\varepsilon_2 \in (0,\frac{1}{2})$, that will be precised later. By Young's inequality,
\begin{equation*}
\left\vert I_3 \right\vert
\leqslant \frac{2r\varepsilon_2}{\varepsilon} \int_{B_r^+} \left\vert \nabla (\zeta \Delta_h\varphi) \right\vert^2 \mathrm{d} x
+ \frac{2}{r\varepsilon\varepsilon_2} \int_{B_r^+} \zeta^2 \left\vert \Delta_h\varphi \right\vert^2 \mathrm{d} x.
\end{equation*}
On the one hand, using the properties of $\zeta$, Young's inequality and \cite[Section~5.8.2, Theorem~3(i)]{Evans},
\begin{align*}
\int_{B_r^+} \left\vert \nabla (\zeta \Delta_h\varphi) \right\vert^2 \mathrm{d} x
& =
\int_{B_{3r/4}^+} \left\vert \nabla (\zeta \Delta_h\varphi) \right\vert^2 \mathrm{d} x
\\
& =
\int_{B_{3r/4}^+} \left\vert (\Delta_h\varphi) \nabla \zeta + \zeta \Delta_h(\nabla \varphi) \right\vert^2 \mathrm{d} x
\\
& \leqslant
2 \int_{B_{3r/4}^+} \left\vert \nabla \zeta \right\vert^2 \left\vert \Delta_h \varphi \right\vert^2 \mathrm{d} x
+ 2 \int_{B_{3r/4}^+} \zeta^2 \left\vert \Delta_h(\nabla \varphi) \right\vert^2 \mathrm{d} x
\\
& \leqslant
\frac{C}{r^2} \int_{B_r^+} \left\vert \nabla \varphi \right\vert^2 \mathrm{d} x
+ 2 \int_{B_r^+} \zeta^2 \left\vert \Delta_h(\nabla \varphi) \right\vert^2 \mathrm{d} x.
\end{align*}
On the other hand, using the properties of $\zeta$ and \cite[Section~5.8.2, Theorem~3(i)]{Evans},
\begin{align*}
\int_{B_r^+} \zeta^2 \left\vert \Delta_h\varphi \right\vert^2 \mathrm{d} x
& \leqslant
\int_{B_{3r/4}^+} \left\vert \Delta_h\varphi \right\vert^2 \mathrm{d} x
\leqslant
 C\int_{B_r^+} \left\vert \nabla \varphi \right\vert^2 \mathrm{d} x.
\end{align*}
We deduce that
\begin{equation}
\label{diffq_estimI3}
\left\vert I_3 \right\vert
\leqslant
C(\varepsilon)r\varepsilon_2 \int_{B_r^+} \zeta^2 \left\vert \Delta_h(\nabla \varphi) \right\vert^2 \mathrm{d} x
+
\frac{C(\varepsilon,\varepsilon_2)}{r} \int_{B_r^+} \left\vert \nabla \varphi \right\vert^2 \mathrm{d} x.
\end{equation}
Combining~\eqref{diffq_estimI2},~\eqref{diffq_estimI3} with the definition of $I_1$ and the relation $I_1+I_2+I_3=0$, we deduce that
\begin{equation*}
\left( 1-\varepsilon_1-C(\varepsilon)r\varepsilon_2 \right)
\int_{B_r^+} \zeta^2 \left\vert \Delta_h(\nabla \varphi) \right\vert^2 \mathrm{d} x
\leqslant \frac{(1+r)C(\varepsilon,\varepsilon_1,\varepsilon_2)}{r^2} \int_{B_r^+} \left\vert \nabla \varphi \right\vert^2 \mathrm{d} x.
\end{equation*}
Using that $r \in (0,1)$, and choosing $\varepsilon_1$ and $\varepsilon_2$ so small that $1-\varepsilon_1-C(\varepsilon)\varepsilon_2 \geqslant \frac{1}{2}$, we finally get
\begin{equation*}
\int_{B_{r/2}^+} \left\vert \Delta_h (\nabla \varphi) \right\vert^2 \mathrm{d} x
\leqslant \int_{B_r^+} \zeta^2 \left\vert \Delta_h(\nabla \varphi) \right\vert^2 \mathrm{d} x
\leqslant \frac{C}{r^2} \int_{B_r^+} \left\vert \nabla \varphi \right\vert^2 \mathrm{d} x,
\end{equation*}
with the first inequality coming from the definition of $\zeta$. By \cite[Section~5.8.2, Theorem~3(ii)]{Evans}, taking the limits when $h$ tends to zero, we get
\begin{equation*}
\int_{B_{r/2}^+} \left\vert \partial_1(\nabla \varphi) \right\vert^2 \mathrm{d} x
\leqslant
\frac{C}{r^2} \int_{B_r^+} \left\vert \nabla \varphi \right\vert^2 \mathrm{d} x.
\end{equation*}
We deduce that
\begin{equation*}
\int_{B_{r/2}^+} \left( \left\vert \partial_{11}\varphi \right\vert^2+\left\vert \partial_{12}\varphi \right\vert^2 \right) \mathrm{d} x
\leqslant 
\frac{C}{r^2} \int_{B_r^+} \left\vert \nabla \varphi \right\vert^2 \mathrm{d} x.
\end{equation*}
Since $\varphi \in C^\infty(\mathbb{R}_+^2)$, we get $\partial_{21}\varphi=\partial_{12}\varphi$ by Schwarz's lemma, and $\partial_{22}\varphi=-\partial_{11}\varphi$ because $\Delta \varphi = 0$ in $\mathbb{R}_+^2$. It follows that $\varphi \in H^2(B_{r/2}^+)$. Translating the support of $\zeta$ with respect to $x_1$ in the previous calculations, we deduce that $\varphi \in H^2_\mathrm{loc}(\overline{\mathbb{R}_+^2})$.

\textit{Step 3:} Let us prove that $\varphi \in C^\infty(\overline{\mathbb{R}_+^2})$.

For the interior regularity, we already noticed in Step 2 that $\varphi \in C^\infty(\mathbb{R}_+^2)$ as a harmonic function.

For the boundary regularity, we repeat the arguments of Step 2 for $\nabla^2\varphi$, etc. and prove by induction that $\varphi \in H^m_\mathrm{loc}(\overline{\mathbb{R}_+^2})$ for every $m \in \mathbb{N}^\ast$, as in \cite[Section~6.3, Theorem~5]{Evans}. Using Sobolev embeddings (see \cite[Section~5.6.3, Theorem~6(ii)]{Evans}), we deduce that $\varphi \in C^\infty(\overline{U})$ for every bounded open subset $U \subset \overline{\mathbb{R}_+^2}$. It follows that $\varphi \in C^\infty(\overline{\mathbb{R}_+^2})$.
\end{proof}

Given a critical point $\varphi_\varepsilon$ of $E_\varepsilon^\delta$ that satisfies~\eqref{systFeps} by Proposition~\ref{propsystFeps}, we will consider in the following the rescaled functions
\begin{equation}
\label{phirepsilon}
\phi_\varepsilon \colon (x_1,x_2) \in \overline{\mathbb{R}_+^2} \mapsto 2\varphi_\varepsilon(\varepsilon x_1, \varepsilon x_2)+\pi,
\end{equation}
that are harmonic in $\mathbb{R}_+^2$ and satisfy $\partial_2\phi_\varepsilon = -\sin(\phi_\varepsilon)+2 \varepsilon \delta_2$ on $\mathbb{R} \times \left\lbrace 0 \right\rbrace$, i.e.
\begin{equation}
\label{systphirepsilon}
\left\lbrace \begin{array}{rcll}
\Delta \phi_\varepsilon&=&0&\text{ in } \mathbb{R}_+^2,
\\ \partial_2\phi_\varepsilon-\lambda_\varepsilon+\sin(\phi_\varepsilon) &=&0 & \text{ on } \mathbb{R} \times \left\lbrace 0 \right\rbrace,
\end{array} \right.
\end{equation}
where $\lambda_\varepsilon=2 \varepsilon \delta_2$. We now look for explicit solutions of this problem under the boundedness condition $\left[ (x_1,x_2) \mapsto \phi_\varepsilon(x_1,x_2)-\lambda_\varepsilon x_2 \right] \in L^\infty(\mathbb{R}_+^2)$ in order to get a modified Peierls-Nabarro problem. More precisely, we will especially look for nonconstant, nonperiodic and bounded solutions, since this type of solutions is expected to minimize $E_{\varepsilon}^\delta$ as for $\delta=0$ (see~\cite{Kurzke06}).

\subsection{A modified Peierls-Nabarro problem}
\label{subsection2.2}

For any $\lambda \in \mathbb{R}$, we consider the problem
\begin{equation}
\tag{PN$_\lambda$}
\left\lbrace \begin{array}{l}
f \in C^\infty(\mathbb{R}_+^2) \cap C^1(\overline{\mathbb{R}_+^2}),
\\
(x_1,x_2) \mapsto f(x_1,x_2)-\lambda x_2 \text{ is bounded in } \mathbb{R}_+^2,
\\
\Delta f=0 \text{ in } \mathbb{R}_+^2,
\\
\partial_2 f-\lambda+\sin f = 0 \text{ on } \mathbb{R} \times \left\lbrace 0 \right\rbrace.
\end{array} \right.
\end{equation}

\noindent
The problem~\eqref{PNlambda} is a generalization of the classical Peierls-Nabarro problem (which is in fact the case $\lambda=0$). In~\cite{Tol97}, Toland shows a link between solutions of the Peierls-Nabarro problem and solutions of the Benjamin-Ono problem, that is given by
\begin{equation}
\tag{BO}
\label{BO}
\left\lbrace \begin{array}{l}
u \in C^\infty(\mathbb{R}_+^2) \cap C^1(\overline{\mathbb{R}_+^2}),
\\
u \text{ is bounded in } \mathbb{R}_+^2,
\\
\Delta u=0 \text{ in } \mathbb{R}_+^2,
\\
\partial_2 u+u^2- u = 0 \text{ on } \mathbb{R} \times \left\lbrace 0 \right\rbrace.
\end{array} \right.
\end{equation}

\noindent
The main point in finding solutions of the Peierls-Nabarro problem is based on the fact that all solutions of the Benjamin-Ono problem were classified by Amick and Toland in~\cite{AmickToland}. The rest of this subsection is devoted to determine solutions of~\eqref{PNlambda}, using the ideas in Toland~\cite{Tol97}. A first observation before linking solutions of~\eqref{PNlambda} with solutions of~\eqref{BO} is the following remark.

\begin{remarque}
\label{remPNlambda_odd}
Note that the problem \eqref{PNlambda} is "odd" with respect to $\lambda$, in the sense that $f$ satisfies~\eqref{PNlambda} if and only if $-f$ satisfies~(PN$_{-\lambda}$). In particular, $f$ satisfies~(PN$_{0}$) if and only if $-f$ satisfies~(PN$_{0}$).
\end{remarque}

\newpage

\begin{flushleft}
\textbf{Linking solutions of \eqref{PNlambda} to solutions of \eqref{BO}.}
\end{flushleft}

Let us formally explain the strategy for establishing Theorem~\ref{thmPNlambdaBO}. Given $\lambda \in \mathbb{R}$ and a solution $f$ of~\eqref{PNlambda}, we look for a relation between $f$ and a solution of~\eqref{BO}. To do so, a crucial observation -- due to Toland~\cite{Tol97} -- is that the function
\begin{equation*}
(x_1,x_2) \in \overline{\mathbb{R}_+^2} \mapsto 2 \arctan \left( \frac{x_1}{1+x_2} \right)
\end{equation*}
satisfies~(PN$_0$), and the function
\begin{equation*}
(x_1,x_2) \in \overline{\mathbb{R}_+^2} \mapsto \frac{2(1+x_2)}{x_1^2+(1+x_2)^2},
\end{equation*}
that satisfies~\eqref{BO}, is the $x_1$-derivative of the first one. It is clear that the latter function is also the $x_1$-derivative of the function
\begin{equation*}
(x_1,x_2) \in \overline{\mathbb{R}_+^2} \mapsto 2 \arctan \left( \frac{x_1}{1+x_2} \right) + \lambda x_2,
\end{equation*}
that satisfies~\eqref{PNlambda}. This is a motivation for setting $u=\partial_1f$ in $\mathbb{R}_+^2$, where $u$ is a solution of~\eqref{BO}.

On the one hand, using $\eqref{BO}$, $\partial_2u=u-u^2=u(1-u)$ on $\mathbb{R} \times \left\lbrace 0 \right\rbrace$. On the other hand, using~$\eqref{PNlambda}$,
\begin{equation*}
\partial_2u
=\partial_{21}f
=\partial_1(\partial_2f-\lambda)
=\partial_1(-\sin f)
=-\partial_1f\cos f
=-u\cos f
\end{equation*}
on $\mathbb{R} \times \left\lbrace 0 \right\rbrace$. Hence, $u(1+\cos f-u)=0$ on $\mathbb{R} \times \left\lbrace 0 \right\rbrace$ and this identity motivates the \linebreak relation $u=1+\cos f$ on $\mathbb{R} \times \left\lbrace 0 \right\rbrace$ (in particular since~\eqref{BO} has nontrivial solutions, see~\cite{Tol97} or Theorem~\ref{thmSolBO} below). Moreover,
\begin{equation*}
\partial_1u=\partial_1(1+\cos f)=-\partial_1f\sin f=\partial_1f(\partial_2f-\lambda)
\end{equation*}
on $\mathbb{R} \times \left\lbrace 0 \right\rbrace$. By harmonicity and boundedness, we can extend this equality to $\mathbb{R}_+^2$ using the Phragm\'en-Lindel\"of principle~\cite[Theorem~2.3.2]{Ransford}. We observe that $u=\partial_1f$ and $\partial_1u=\partial_1f(\partial_2f-\lambda)$ in $\mathbb{R}_+^2$, so that $\partial_1u=\frac{1}{2}w_{f,\lambda}$ where
\begin{equation*}
w_{f,\lambda}=\partial_{11}f+\partial_1f(\partial_2f-\lambda).
\end{equation*} 

Hence, the idea of Theorem~\ref{thmPNlambdaBO} consists in setting $w_{f,\lambda}$ as above, and re-construct $u$ as an integral of $\frac{1}{2} w_{f,\lambda}$ with respect to $x_1$ such that $u$ satisfies~\eqref{BO}. More precisely, given $w_{f,\lambda}$, we first introduce a function $W_{f,\lambda}$ that satisfies $w_{f,\lambda} = \partial_1 W_{f,\lambda}$:
\begin{equation*}
W_{f,\lambda}(x_1,x_2)=W_{f,\lambda}(0,x_2)+\int_0^{x_1} w_{f,\lambda} (s,x_2) \ \mathrm{d} s,
\end{equation*}
where $W_{f,\lambda}(0,x_2)$ is chosen such that $W_{f,\lambda}$ is harmonic. Then, we set
\begin{equation*}
u(x_1,x_2)=\frac{1}{2}W_{f,\lambda}(x_1,x_2)+G_{f,\lambda}(x_2),
\end{equation*}
where $G_{f,\lambda}$ has to be precised. Since $W_{f,\lambda}$ is harmonic and $u$ must be harmonic for satisfying~\eqref{BO}, then $G_{f,\lambda}$ must be affine. The above relation and the calculation of $W_{f,\lambda}(x_1,0)$ (see~\eqref{W(x_1,0)} below or~\cite[Section~4]{Tol97}) give
\begin{equation}
\label{Gflambda0}
G_{f,\lambda}(0)
=u(x_1,0)-C+\frac{1}{2}\partial_1f(x_1,0)-\frac{1}{2}\cos f(x_1,0).
\end{equation}
However, as $u$ is expected to satisfy $u(x_1,0)=1+\cos f(x_1,0)$, we see that $G_{f,\lambda}(0)$ is a priori not constant.

For solving this difficulty, we rely on the odd symmetry of the problem~\eqref{PNlambda} (see Remark~\ref{remPNlambda_odd}). More precisely, we consider
\begin{equation*}
v(x_1,x_2)=\frac{1}{2}W_{-f,-\lambda}(x_1,x_2)+G_{-f,-\lambda}(x_2),
\end{equation*}
so that $\partial_1v=\frac{1}{2}\partial_1W_{-f,-\lambda}=\frac{1}{2}w_{-f,-\lambda}$. Similarly than $u$ before, $v$ is expected to satisfy the \linebreak relation $v(x_1,0)=1+\cos f(x_1,0)$, so that adding~\eqref{Gflambda0} and the analogous identity for $G_{-f,-\lambda}(0)$, the terms $\partial_1f(x_1,0)$ cancel and we get:
\begin{equation*}
G_{f,\lambda}(0)+G_{-f,-\lambda}(0)
= C+\underbrace{\cos f(x_1,0)}_{\text{from } u(x_1,0)} + \underbrace{\cos f(x_1,0)}_{\text{from } v(x_1,0)} - \cos f(x_1,0).
\end{equation*}
In order to cancel the cosines completely, we replace the functions $u$ and $v$ defined above by their half, so that the sum $u(x_1,0)+v(x_1,0)$ is expected to be equal to $1+\cos f(x_1,0)$. This replacements do not change harmonicity or boundedness, and it justifies the expected relations~\eqref{v-u},~\eqref{dr_1u+dr_1v} and~\eqref{u0+v0} below.

\begin{theoreme}
\label{thmPNlambdaBO}
Let $\lambda \in \mathbb{R}$ and $f$ be a solution of~$\eqref{PNlambda}$. Then there exist two solutions $u$ and $v$ of~$\eqref{BO}$ such that
\begin{equation}
u-v=\partial_1f \ \ \ \ \text{ in } \mathbb{R}_+^2,
\label{v-u}
\end{equation}
\begin{equation}
\partial_1u+\partial_1v=\partial_1f(\partial_2f-\lambda) \ \ \ \ \text{ in } \mathbb{R}_+^2,
\label{dr_1u+dr_1v}
\end{equation}
and
\begin{equation}
u(x_1,0)+v(x_1,0)=1 + \cos f(x_1,0) \ \ \ \ \forall x_1 \in \mathbb{R}.
\label{u0+v0}
\end{equation}
\end{theoreme}

\begin{proof}
Let $\lambda \in \mathbb{R}$ and $f \colon \overline{\mathbb{R}_+^2} \rightarrow \mathbb{R}$ be a solution of~\eqref{PNlambda}.

\textit{Step 1:} The function $w_{f,\lambda} \colon \overline{\mathbb{R}_+^2} \rightarrow \mathbb{R}$, defined as
\begin{equation*}
w_{f,\lambda} := \partial_{11}f+\partial_1f(\partial_2f-\lambda),
\end{equation*}
is harmonic in $\mathbb{R}_+^2$ and satisfies
\begin{equation*}
w_{f,\lambda}(x_1,0) = \left. \left[ \partial_1 \left( \partial_1f(x_1,x_2)+\cos f(x_1,x_2) \right) \right]\right\vert_{x_2=0} \ \ \ \ \forall x_1 \in \mathbb{R},
\label{w(x_1,0)}
\end{equation*}
\begin{equation*}
\partial_2w_{f,\lambda}(x_1,0)= \left. \left[ -\frac{1}{2}\partial_1\left( (\partial_1f(x_1,x_2)+\cos f(x_1,x_2))^2 \right) \right] \right\vert_{x_2=0} \ \ \ \ \forall x_1 \in \mathbb{R},
\label{dr_2w(x_1,0)}
\end{equation*}
and
\begin{equation*}
-\partial_1w_{f,\lambda}=\partial_{22}\left( \partial_1f-\lambda f \right) +\partial_{22}f\partial_2f-\partial_{12}f\partial_1f \ \ \ \ \text{ in } \mathbb{R}_+^2.
\label{-dr_1w}
\end{equation*}
For the proofs of the harmonicity and of the three identities above, that are elementary calculations, we refer to Toland~\cite[Section~4]{Tol97}.

\textit{Step 2:} The function $W_{f,\lambda} \colon \overline{\mathbb{R}_+^2} \rightarrow \mathbb{R}$, defined as
\begin{equation*}
\begin{split}
W_{f,\lambda}(x_1,x_2)
& :=
\partial_1f(0,x_2)
- \frac{1}{2} \int_{x_2}^{+\infty} \left( (\partial_2f(0,t)-\lambda)^2-\partial_1f(0,t)^2 \right) \mathrm{d} t
\\
& \qquad
+ \int_0^{x_1} w_{f,\lambda}(s,x_2) \ \mathrm{d} s,
\end{split}
\end{equation*}
for every $(x_1,x_2) \in \overline{\mathbb{R}_+^2}$, is harmonic and bounded in $\mathbb{R}_+^2$ and satisfies
\begin{equation}
W_{f,\lambda}(x_1,0)
= A_{\lambda}
+ \partial_1f(x_1,0)
+ \cos f(x_1,0)
\ \ \ \ \forall x_1 \in \mathbb{R},
\label{W(x_1,0)}
\end{equation}
and
\begin{equation}
\partial_2W_{f,\lambda}(x_1,0)=\frac{1}{2} \left( 1+W_{f,\lambda}(x_1,0)-A_\lambda \right) \left( 1-W_{f,\lambda}(x_1,0)+A_\lambda \right) \ \ \ \ \forall x_1 \in \mathbb{R},
\label{dr_2W(x_1,0)}
\end{equation}
where
\begin{equation*}
A_{\lambda} = - \cos f(0,0) - \frac{1}{2} \int_0^{+\infty} \left( (\partial_2f(0,t)-\lambda)^2-\partial_1f(0,t)^2 \right) \mathrm{d} t.
\end{equation*}
Once again, we refer to Toland~\cite[Section~4]{Tol97} for the proofs of this properties of $W_{f,\lambda}$, using Step~1.

\textit{Step 3:} Let $u \colon \overline{\mathbb{R}_+^2} \rightarrow \mathbb{R}$ be defined as
\begin{equation*}
u(x_1,x_2) := \frac{1}{2} \left( 1+W_{f,\lambda}(x_1,x_2)-A_\lambda \right),
\end{equation*}
for every $(x_1,x_2) \in \overline{\mathbb{R}_+^2}$. By Step~2, $u$ is harmonic and bounded in $\mathbb{R}_+^2$, and of class $C^1$ in $\overline{\mathbb{R}_+^2}$. Using~\eqref{dr_2W(x_1,0)}, for every $x_1 \in \mathbb{R}$,
\begin{align*}
\partial_2u(x_1,0)
& = \frac{1}{2} \partial_2 W_{f,\lambda} (x_1,0)
\\
& = \frac{1}{4}(1+W_{f,\lambda}(x_1,x_2)-A_\lambda)(1-W_{f,\lambda}(x_1,x_2)+A_\lambda)
\\
& = u(x_1,0)(1-u(x_1,0)),
\end{align*}
so that $u$ is a solution of~\eqref{BO}.

Let $v \colon \overline{\mathbb{R}_+^2} \rightarrow \mathbb{R}$ be defined as
\begin{equation*}
v(x_1,x_2) := \frac{1}{2} \left( 1+W_{-f,-\lambda}(x_1,x_2)-A_{-\lambda} \right),
\end{equation*}
for every $(x_1,x_2) \in \overline{\mathbb{R}_+^2}$. Similarly than before, $v$ is harmonic and bounded in $\mathbb{R}_+^2$, of class $C^1$ \linebreak in $\overline{\mathbb{R}_+^2}$ and satisfies~\eqref{BO}.

For every $x_1 \in \mathbb{R}$, using~\eqref{W(x_1,0)},
\begin{align*}
u(x_1,0)+v(x_1,0)
& = 1 + \frac{1}{2}(W_{f,\lambda}(x_1,0)-A_\lambda) + \frac{1}{2}(W_{-f,-\lambda}(x_1,0)-A_{-\lambda})
\\
& = 1 + \cos f(x_1,0),
\end{align*}
which shows~\eqref{u0+v0}. Using~\eqref{W(x_1,0)} again, for every $x_1 \in \mathbb{R}$,
\begin{align*}
u(x_1,0)-v(x_1,0)
& = \frac{1}{2}(W_{f,\lambda}(x_1,0)-A_\lambda)-\frac{1}{2}(W_{-f,-\lambda}(x_1,0)-A_{-\lambda})
= \partial_1f(x_1,0).
\end{align*}
Since $(x_1,x_2) \mapsto f(x_1,x_2)-\lambda x_2$ is harmonic and bounded in $\mathbb{R}_+^2$, as a solution of~\eqref{PNlambda}, then $\partial_1f$ is bounded (see~\cite[Theorem~2.10]{GT01}). Hence, $u-v$ and $\partial_1f$ are both bounded harmonic functions which coincide on $\mathbb{R} \times \left\lbrace 0 \right\rbrace = \partial(\mathbb{R}_+^2) \setminus \left\lbrace \infty \right\rbrace$. By the Phragm\' en-Lindel\" of principle~\cite[Theorem~2.3.2]{Ransford}, the functions $u-v$ and $\partial_1f$ coincide in $\mathbb{R}_+^2$. This proves~\eqref{v-u}. Finally, in $\mathbb{R}_+^2$,
\begin{align*}
\partial_1u+\partial_1v
& = \frac{1}{2}\left( \partial_1 W_{f,\lambda} + \partial_1 W_{-f,-\lambda} \right)
= \frac{1}{2} \left( w_{f,\lambda}+w_{-f,-\lambda} \right)
= \partial_1f(\partial_2f-\lambda),
\end{align*}
which gives~\eqref{dr_1u+dr_1v}.
\end{proof}

\newpage

We now quote two statements from~\cite{AmickToland}.

\begin{theoreme}
\label{thmSolBO}
Solutions of~\eqref{BO} in $\overline{\mathbb{R}_+^2}$ are:
\begin{itemize}
\item[--] the constant function $u_0 \equiv 0$,
\item[--] for $\alpha \in [1,2)$, the functions
\begin{equation}
\label{solBOperiodic}
u_\alpha \colon (x_1,x_2) \mapsto \frac{2\sigma \Gamma_\alpha(x_2)}{\cos^2(\sigma x_1)+\Gamma_\alpha(x_2)^2\sin^2(\sigma x_1)}
\end{equation}
where
\begin{equation}
\label{deltaGammagamma}
\sigma=\frac{1}{2}\sqrt{\alpha(2-\alpha)}, \ \ \Gamma_\alpha(x_2)=\frac{\gamma+\tanh(\sigma x_2)}{1+\gamma \tanh(\sigma x_2)} \ \text{ and } \ \gamma=\frac{\alpha}{2\sigma},
\end{equation}
which are non-constant periodic functions of the variable $x_1$, and every translation of $u_\alpha$ in the $x_1$-direction,
\item[--] the function
\begin{equation}
\label{solBOnonperiodic}
u_2 \colon (x_1,x_2) \mapsto \frac{2(1+x_2)}{x_1^2+(1+x_2)^2}
\end{equation}
which is non-constant and non-periodic in $x_1$, and every translation of $u_2$ in the $x_1$-direction.
\end{itemize}
\end{theoreme}

\begin{remarque}
The solution $u_\alpha$ given in~\eqref{solBOperiodic} is not well-defined for $\alpha=2$, because in this case $\sigma=0$. However, $(u_\alpha)_{\alpha \in [1,2)}$ converges pointwise to $u_2$ when $\alpha \rightarrow 2$.
\end{remarque}

\begin{proposition}
\label{properties_solBO}
Solutions of~\eqref{BO} have the following properties:
\begin{itemize}
\item[i)] For every $\alpha \in [1,2]$,
\begin{equation}
\tag{P1$_{\mathrm{BO}}$}
\label{PBO1}
u_\alpha>0.
\end{equation}
\item[ii)] For $\alpha \in (1,2)$, for every $x_2>0$,
\begin{equation}
\tag{P2$_{\mathrm{BO}}$}
\label{PBO2}
u_\alpha(\cdot,x_2) \text{ is } \frac{\pi}{\sigma}\text{-periodic and } u_\alpha(0,0)=\alpha=\max\limits_\mathbb{R} u_\alpha(\cdot,x_2).
\end{equation}
\item[iii)] For every $\alpha \in (1,2]$,
\begin{equation}
\tag{P3$_{\mathrm{BO}}$}
\label{PBO3}
\partial_1u_\alpha(\cdot,0)^2=\alpha(\alpha-2) u_\alpha(\cdot,0)^2+2u_\alpha(\cdot,0)^3-u_\alpha(\cdot,0)^4 \text{ in } \mathbb{R}.
\end{equation}
\item[iv)] For every $\alpha \in (1,2]$,
\begin{equation}
\tag{P4$_{\mathrm{BO}}$}
\label{PBO4}
\sup\limits_\mathbb{R} u_\alpha(\cdot,0)+\inf\limits_\mathbb{R} u_\alpha(\cdot,0)=2.
\end{equation}
\end{itemize}
\end{proposition}

\begin{flushleft}
\textbf{Finding solutions of \eqref{PNlambda}.}
\end{flushleft}

We now prove Theorem~\ref{thmSolPNlambda} by using Theorem~\ref{thmPNlambdaBO}, Theorem~\ref{thmSolBO} and the properties of the solutions of~\eqref{BO} given in Proposition~\ref{properties_solBO}.

Let $\lambda \in \mathbb{R}$ and $f$ be a solution of~\eqref{PNlambda}. By~Theorem \ref{thmPNlambdaBO}, there exist two solutions $u$ and $v$ of~\eqref{BO} that are related to $f$. Since solutions of~\eqref{BO} are entirely known thanks to Theorem~\ref{thmSolBO}, we proceed by testing all possible cases. Each case consists in choosing $v$ (the possible choices being $u_0$, $u_1$, $u_2$ and $u_\alpha$, with $\alpha \in (1,2)$, from Theorem~\ref{thmSolBO}), and then testing each choice for $u$ (if needed). For each case, we will obtain either a candidate function $f$ to be a solution of~\eqref{PNlambda}, or a contradiction, that means that case cannot occur. Note that, because of~\eqref{v-u}, we also have to consider the candidate functions $f$ obtained by inverting $u$ and $v$.

\newpage

We need the following lemmas for the first two cases.

\begin{lemme}
\label{lemmeintegrabilite}
For every $\alpha \in \left\lbrace 0 \right\rbrace \cup [1,2]$, $u_\alpha(\cdot,0)$ is integrable on $\mathbb{R}$ if and only if $\alpha \in \left\lbrace 0,2 \right\rbrace$.
\end{lemme}

\begin{proof}
It suffices to prove that $u_0(\cdot,0)$ and $u_2(\cdot,0)$ are integrable on $\mathbb{R}$ and that the other options for $u_\alpha(\cdot,0)$ are not integrable on $\mathbb{R}$.
Since $u_0 \equiv 0$ and $u_1 \equiv 1$, it is clear that $u_0(\cdot,0)$ is integrable on $\mathbb{R}$ and $u_1(\cdot,0)$ is not integrable on $\mathbb{R}$. For every $x_1 \in \mathbb{R}$,
\begin{equation*}
u_2(x_1,0)=\frac{2}{x_1^2+1},
\end{equation*}
thus $u_2(\cdot,0)$ is clearly integrable on $\mathbb{R}$. For every $\alpha \in (1,2)$, for every $x_1 \in \mathbb{R}$,
\begin{equation*}
u_\alpha(x_1,0)=\frac{2\sigma \gamma}{\cos^2(\sigma x_1)+\gamma^2\sin^2(\sigma x_1)} \geqslant \frac{2\sigma \gamma}{1+\gamma^2},
\end{equation*}
thus $u_\alpha(\cdot,0)$ is not integrable on $\mathbb{R}$.
\end{proof}

\begin{lemme}
\label{lemmeintegraleualpha}
For every $\alpha \in (1,2)$ and $x_2 \geqslant 0$,
\begin{equation}
\int_0^{\pi/\sigma} u_\alpha(x_1,x_2) \ \mathrm{d} x_1=2\pi.
\end{equation}
\end{lemme}

\begin{proof}
Let $\alpha \in (1,2)$ and $x_2 \geqslant 0$ be fixed. By change of variable and by the dominated convergence theorem,
\begin{align*}
\int_0^{\pi/\sigma} u_\alpha(x_1,x_2) \ \mathrm{d} x_1
& = \lim\limits_{\varepsilon \rightarrow 0} \int_\varepsilon^{\pi-\varepsilon} \frac{2\Gamma_\alpha(x_2)}{\cos^2(x_1)+\Gamma_\alpha(x_2)^2\sin^2(x_1)} \ \mathrm{d} x_1.
\end{align*}
Moreover, $\Gamma_\alpha(x_2)>0$ by \eqref{deltaGammagamma}, and for every $\varepsilon>0$,
\begin{align*}
\int_\varepsilon^{\pi-\varepsilon} \frac{2\Gamma_\alpha(x_2)}{\cos^2(x_1)+\Gamma_\alpha(x_2)^2\sin^2(x_1)} \ \mathrm{d} x_1
& = -2\int_\varepsilon^{\pi-\varepsilon} \frac{g_{\alpha,x_2}'(x_1)}{1+g_{\alpha,x_2}(x_1)^2} \ \mathrm{d} x_1
\end{align*}
where $g_{\alpha,x_2}(x_1)=\frac{\cot(x_1)}{\Gamma_\alpha(x_2)}$, hence
\begin{align*}
& \int_\varepsilon^{\pi-\varepsilon} \frac{2\Gamma_\alpha(x_2)}{\cos^2(x_1)+\Gamma_\alpha(x_2)^2\sin^2(x_1)} \ \mathrm{d} x_1
\\
& \qquad \qquad = -2 \left[ \arctan \left( \frac{\cot(\pi-\varepsilon)}{\Gamma_\alpha(x_2)} \right)-\arctan \left( \frac{\cot(\varepsilon)}{\Gamma_\alpha(x_2)} \right) \right].
\end{align*}
Taking the limits when $\varepsilon$ tends to zero, we get
\begin{align*}
\int_0^{\pi/\sigma} u_\alpha(x_1,x_2) \ \mathrm{d} x_1
& = 2\pi.
\end{align*}
\end{proof}

\newpage

We come back to the proof of Theorem~\ref{thmSolPNlambda}.

\begin{flushleft}
\textbf{Case~1: $v=u_0$.}
\end{flushleft}
By~\eqref{v-u}, $\partial_1f=u$ in $\mathbb{R}_+^2$. Note that this equality extends to $\overline{\mathbb{R}_+^2}$ by continuity of $\partial_1f$ and $v$.
\vspace{.1cm}

-- Subcase~1: $u=u_0$.
\vspace{.1cm}

\noindent
Then $\partial_1f \equiv 0$ in $\overline{\mathbb{R}_+^2}$ and by \eqref{u0+v0}, $\cos f(\cdot,0) \equiv -1$ in $\mathbb{R}$, i.e. there exists an integer $n \in \mathbb{Z}$ such \linebreak that $f(\cdot,0)=(2n+1)\pi$ in $\mathbb{R}$. Since $(x_1,x_2) \mapsto f(x_1,x_2)-\lambda x_2$ and $(x_1,x_2) \mapsto (2n+1)\pi$ are bounded harmonic functions in $\mathbb{R}_+^2$ which coincide on $\mathbb{R} \times \left\lbrace 0 \right\rbrace = \partial(\mathbb{R}_+^2) \setminus \left\lbrace \infty \right\rbrace$, they coincide in $\overline{\mathbb{R}_+^2}$ by the Phragm\' en-Lindel\" of principle~\cite[Theorem~2.3.2]{Ransford}. As a consequence, for every $(x_1,x_2) \in \overline{\mathbb{R}_+^2}$,
\begin{equation*}
f(x_1,x_2)=(2n+1)\pi+\lambda x_2.
\end{equation*}

-- Subcase~2: $u=u_1$.
\vspace{.1cm}

\noindent
Then $\partial_1f \equiv 1$ in $\overline{\mathbb{R}_+^2}$, so that $f(\cdot,x_2)$ is not bounded in $\mathbb{R}$ for every $x_2>0$. This contradicts the fact that, for every fixed $x_2>0$, $f(\cdot ,x_2)=\left( f(\cdot ,x_2)-\lambda x_2 \right) + \lambda x_2$ is a bounded function since $f$ is a solution of~\eqref{PNlambda}. Hence, the configuration $(u,v)=(u_1,u_0)$ is not possible.
\vspace{.1cm}

-- Subcase~3: $u=u_2$.
\vspace{.1cm}

\noindent
Up to a translation with respect to the variable $x_1$, we assume that for every $x_1 \in \mathbb{R}$,
\begin{equation*}
\partial_1f(x_1,0)=\frac{2}{x_1^2+1},
\end{equation*}
using~\eqref{solBOnonperiodic} in Theorem~\ref{thmSolBO}, from which it follows
\begin{equation*}
f(x_1,0)=f(0,0)+2\arctan(x_1).
\end{equation*}
Furthermore, by~\eqref{u0+v0}, we have $\cos f(0,0)=u_2(0,0)-1=1$ and thus, there exists an integer $n \in \mathbb{Z}$ such that $f(0,0)=2n\pi$. Since $(x_1,x_2) \mapsto f(x_1,x_2)-\lambda x_2$ and $(x_1,x_2) \mapsto 2n\pi + 2 \arctan \left( \frac{x_1}{1+x_2} \right)$ are bounded harmonic functions in $\mathbb{R}_+^2$ which coincide on $\mathbb{R} \times \left\lbrace 0 \right\rbrace = \partial(\mathbb{R}_+^2) \setminus \left\lbrace \infty \right\rbrace$, they coincide in $\overline{\mathbb{R}_+^2}$ by the Phragm\' en-Lindel\" of principle. As a consequence, for every $(x_1,x_2) \in \overline{\mathbb{R}_+^2}$,
\begin{equation*}
f(x_1,x_2)=2n\pi+2\arctan \left( \frac{x_1}{1+x_2} \right)+\lambda x_2.
\end{equation*}

-- Subcase~4: $u=u_\alpha$ with $\alpha \in (1,2)$.
\vspace{.1cm}

\noindent
Then $u>0$ by~\eqref{PBO1} and, since $u=\partial_1f$ and $f$ is bounded on $\mathbb{R} \times \left\lbrace 0 \right\rbrace$, $u$ is integrable on $\mathbb{R} \times \left\lbrace 0 \right\rbrace$. This is a contradiction with Lemma \ref{lemmeintegrabilite}, thus the configuration $(u,v)=(u_\alpha,u_0)$ with $\alpha \in (1,2)$ is not possible.

\begin{flushleft}
\textbf{Case~2: $v=u_1$.}
\end{flushleft}
By~\eqref{v-u}, $\partial_1f=u-1$ in $\mathbb{R}_+^2$. Note that this equality extends to $\overline{\mathbb{R}_+^2}$ by continuity of $\partial_1f$ and $u$.
\vspace{.1cm}

-- Subcase~1: $u=u_0$.
\vspace{.1cm}

\noindent
Then $\partial_1f \equiv -1$ in $\overline{\mathbb{R}_+^2}$, and we get a contradiction as in the Subcase~2 of Case~1. Hence, the configuration $(u,v)=(u_0,u_1)$ is not possible.
\vspace{.1cm}

-- Subcase~2: $u=u_1$.
\vspace{.1cm}

\noindent
Then $\partial_1f \equiv 0$ in $\overline{\mathbb{R}_+^2}$ and by~\eqref{u0+v0}, $\cos f(\cdot,0) \equiv 1$ in $\mathbb{R}$, i.e. there exists an integer $n \in \mathbb{Z}$ such \linebreak that $f(\cdot,0)=2n\pi$ in $\mathbb{R}$. Proceeding as in the Subcase~1 of Case~1, we get, for every $(x_1,x_2) \in \overline{\mathbb{R}_+^2}$,
\begin{equation*}
f(x_1,x_2)=2n\pi+\lambda x_2.
\end{equation*}

\newpage

-- Subcase~3: $u=u_2$.
\vspace{.1cm}

\noindent
Then $u>0$ by~\eqref{PBO1} and, since $u=\partial_1f+1$ and $f$ is bounded on $\mathbb{R} \times \left\lbrace 0 \right\rbrace$, $u$ is not integrable \linebreak on $\mathbb{R} \times \left\lbrace 0 \right\rbrace$. This is a contradiction with Lemma~\ref{lemmeintegrabilite}, thus the configuration $(u,v)=(u_2,u_1)$ is not possible.
\vspace{.1cm}

-- Subcase~4: $u=u_\alpha$ with $\alpha \in (1,2)$.
\vspace{.1cm}

\noindent
Let $\sigma$ be given by~\eqref{deltaGammagamma}. Since $u=\partial_1f+1$ in $\overline{\mathbb{R}_+^2}$, we have
\begin{align*}
\int_0^{N\pi/\sigma} u(x_1,x_2) \ \mathrm{d} x_1
& = \int_0^{N\pi/\sigma} (\partial_1f(x_1,x_2)+1) \ \mathrm{d} x_1
\\ & = f \left( \frac{N\pi}{\sigma},x_2 \right) - f \left( 0,x_2 \right) + \frac{N\pi}{\sigma}
\end{align*}
for every $x_2>0$ and $N \in \mathbb{N}^\ast$. By Lemma~\ref{lemmeintegraleualpha} and since $u_\alpha$ is $\frac{\pi}{\sigma}$-periodic in $x_1$ by~\eqref{PBO2}, we get
\begin{align*}
2\pi N = N \int_0^{\pi/\sigma} u(x_1,x_2) \ \mathrm{d} x_1
& = f \left( \frac{N\pi}{\sigma},x_2 \right)-f(0,x_2)+\frac{N\pi}{\sigma},
\end{align*}
i.e.
\begin{align*}
2\pi-\frac{\pi}{\sigma} = \frac{1}{N} \left[ \left( f \left( \frac{N\pi}{\sigma},x_2 \right)-\lambda x_2 \right) - \left( f(0,x_2) -\lambda x_2 \right) \right]
\end{align*}
for every $x_2>0$ and $N \in \mathbb{N}^\ast$. But $f(\cdot,x_2)-\lambda x_2$ is bounded for every $x_2>0$ so that, letting $N$ tend to $+\infty$, we get
\begin{equation*}
2\pi=\frac{\pi}{\sigma}
\end{equation*}
that is to say $\alpha=1$, thanks to~\eqref{deltaGammagamma}. This is a contradiction, since we assumed $\alpha \in (1,2)$. Hence, the configuration $(u,v)=(u_\alpha,u_1)$ with $\alpha \in (1,2)$ is not possible.

\begin{flushleft}
\textbf{Case~3: $v=u_2$ (or a translation of $u_2$ in the variable $x_1$).}
\end{flushleft}
Up to a translation with respect to the variable $x_1$, we assume that for every $x_1 \in \mathbb{R}$,
\begin{equation*}
v(x_1,0)=\frac{2}{x_1^2+1}
\end{equation*}
using~\eqref{solBOnonperiodic} in Theorem~\ref{thmSolBO}. We clearly have $u(0,0)=2$. By~\eqref{u0+v0}, we have
\begin{equation*}
u(\cdot,0)=1+\cos f(\cdot,0)-v(\cdot,0) \leqslant 2-v(\cdot,0) \ \ \text{ in } \mathbb{R}.
\end{equation*}
In particular, $u(0,0) \leqslant 0$ and then, by~\eqref{PBO1}, we necessarily have $u = u_0$. By~\eqref{v-u} and using the odd symmetry of the problem~\eqref{PNlambda} as explained in Remark~\ref{remPNlambda_odd}, we can proceed similarly to the Subcase~3 of Case~1. We get, for every $(x_1,x_2) \in \overline{\mathbb{R}_+^2}$,
\begin{equation*}
f(x_1,x_2)=2n\pi-2\arctan \left( \frac{x_1}{1+x_2} \right)+\lambda x_2.
\end{equation*}

\newpage

\begin{flushleft}
\textbf{Case~4: $v=u_\alpha$ (or a translation of $u_\alpha$ in the variable $x_1$) with $\alpha \in (1,2)$.}
\end{flushleft}
Let $\alpha \in (1,2)$. Up to a translation with respect to the variable $x_1$, we assume that
\begin{equation*}
v(x_1,x_2)=\frac{2\sigma \Gamma_\alpha(x_2)}{\cos^2(\sigma x_1)+\Gamma_\alpha(x_2)^2\sin^2(\sigma x_1)}
\end{equation*}
for every $(x_1,x_2) \in \overline{\mathbb{R}_+^2}$, using~\eqref{solBOperiodic} in Theorem~\ref{thmSolBO}. By~\eqref{v-u} and using the odd symmetry of the problem~\eqref{PNlambda} as explained in Remark~\ref{remPNlambda_odd}, we can exclude the three configurations $(u_0,u_\alpha)$, $(u_1,u_\alpha)$ and $(u_2,u_\alpha)$. Indeed, we can proceed similarly to the Subcase~4 of Case~1 (for $u=u_0$), the Subcase~4 of Case~2 (for $u=u_1$) and the Case~3 (for $u=u_2$). Hence, it remains to consider the case $u=u_{\alpha'}$ for some $\alpha' \in (1,2)$.

\begin{lemme}
\label{lemtvi}
There exists $x_0 \in \mathbb{R}$ such that $\partial_1v(x_0,0) \neq 0$ and
\begin{equation}
\label{lienuv}
v(x_0-x_1,0)=u(x_0+x_1,0)
\end{equation}
for every $x_1 \in \mathbb{R}$.
\end{lemme}

\begin{lemme}
\label{lemmaximizer}
Let $x_0 \in \mathbb{R}$ be given by Lemma~\ref{lemtvi}. Then there exists a unique maximizer $x_M \in \mathbb{R}$ of $v(\cdot,0)$ that is closest to $x_0$. Moreover, we have
\begin{equation}
\label{x0-xM}
\left\vert x_0-x_M \right\vert = \frac{\pi}{4\sigma}
\end{equation}
and
\begin{equation}
\label{cosfxM}
\cos f(x_M,0)=1.
\end{equation}
\end{lemme}

We refer to Toland~\cite{Tol97} -- equations~(5.9) to~(5.13) -- for the proofs of both lemmas.
\vspace{.5cm}

\noindent
Let $x_0 \in \mathbb{R}$ be given by Lemma \ref{lemtvi}. By Lemma \ref{lemmaximizer}, there exists a unique maximizer $x_M$ of $u(\cdot,0)$ such that the distance between $x_0$ and $x_M$ is $\frac{\pi}{4\sigma}$. This gives us two configurations: either $x_M<x_0$ or $x_M>x_0$. Using the translation invariance of $v(\cdot,0)$, we can reduce this two subcases~to:
\begin{itemize}
\item[--] Subcase 1: $x_M=0$ and $x_0=\frac{\pi}{4\sigma}$,
\item[--] Subcase 2: $x_M=0$ and $x_0=-\frac{\pi}{4\sigma}$.
\end{itemize}

\noindent
From now on, we assume $x_M=0$ and we keep the notation $x_0$ (with $x_0 =\pm \frac{\pi}{4\sigma}$). By~\eqref{v-u} and~\eqref{lienuv}, for every $x_1 \in \mathbb{R}$,
\begin{align*}
\partial_1f(x_0+x_1,0)
& = u(x_0+x_1,0)-v(x_0+x_1,0)
= v(x_0-x_1,0)-v(x_0+x_1,0),
\end{align*}
i.e., after a change of variable,
\begin{equation}
\label{dr1f_sign}
\partial_1f(x_1,0)
= v(2x_0-x_1,0)-v(x_1,0).
\end{equation}
As $x_M=0$, we deduce that, for every $x_1 \in \mathbb{R}$,
\begin{align*}
f(x_1,0)
& = f(x_M,0)+\int_{x_M}^{x_1} \partial_1f(s,0)\mathrm{d} s
\\ & = f(0,0)+\int_0^{x_1} \left( v(2x_0-s,0)-v(s,0) \right) \mathrm{d} s.
\end{align*}
By~\eqref{cosfxM}, there exists $n \in \mathbb{Z}$ such that $f(0,0)=2n\pi$, so that
\begin{align*}
f(x_1,0)
& = 2n\pi+\int_0^{x_1} \left( v(2x_0-s,0)-v(s,0) \right) \mathrm{d} s
\end{align*}
for every $x_1 \in \mathbb{R}$. Hence, the function $(x_1,x_2) \mapsto f(x_1,x_2)-\lambda x_2 -2n\pi$ is harmonic and bounded in $\mathbb{R}_+^2$ -- as solution of~\eqref{PNlambda} because $\sin(2n\pi)=0$ and $\cos(2n\pi)=1$ -- and its value \linebreak on $\partial(\mathbb{R}_+^2) \setminus \left\lbrace \infty \right\rbrace=\mathbb{R} \times \left\lbrace 0 \right\rbrace$ is
\begin{equation*}
\int_0^{x_1} \left( v(2x_0-s,0)-v(s,0) \right) \mathrm{d} s
\end{equation*}
for every $x_1 \in \mathbb{R}$. In order to extend the above function to $\overline{\mathbb{R}_+^2}$, we cite~\cite[Equation~(5.16)]{Tol97}:

\begin{lemme}
\label{lemg}
Assume $x_0=\pm\frac{\pi}{4\sigma}$. The function $g \colon \mathbb{R}_+^2 \rightarrow \mathbb{R}$ defined as
\begin{equation}
g(x_1,x_2) = \int_0^{x_1} \left( v \left( 2x_0-s,x_2 \right) -v \left( s,x_2 \right) \right) \mathrm{d} s
\end{equation}
is harmonic and bounded.
\end{lemme}

The functions $g$ defined in Lemma~\ref{lemg} and $(x_1,x_2) \mapsto f(x_1,x_2)-\lambda x_2-2n\pi$ are bounded harmonic functions in $\mathbb{R}_+^2$ which coincide on $\mathbb{R} \times \left\lbrace 0 \right\rbrace = \partial(\mathbb{R}_+^2) \setminus \left\lbrace 0 \right\rbrace$. By the Phragm\' en-Lindel\" of principe, they coincide in $\overline{\mathbb{R}_+^2}$ and thus, for every $(x_1,x_2) \in \overline{\mathbb{R}_+^2}$,
\begin{align*}
f(x_1,x_2)
& = 2n\pi + \int_0^{x_1} \left( v(2x_0-s,x_2)-v(s,x_2) \right) \mathrm{d} s + \lambda x_2.
\end{align*}
We now compute the integral explicitely in order to show that $f$ has the expression given in~\eqref{solPNperiodic} in Theorem~\ref{thmSolPNlambda}.

\begin{lemme}
\label{lemexprualpha}
Assume $x_0=\pm \frac{\pi}{4\sigma}$. For every $x_1 \in \mathbb{R} \setminus \left( \frac{\pi}{2\sigma} + \frac{\pi}{\sigma}\mathbb{Z} \right)$ and $x_2 \geqslant 0$,
\begin{align*}
& \int_0^{x_1} \left( v(2x_0-s,x_2)-v(s,x_2) \right) \mathrm{d} s
\\
& \qquad \qquad = 2 \left[ \arctan \left( \frac{\tan(\sigma x_1)}{\Gamma_\alpha(x_2)} \right)-\arctan \left( \Gamma_\alpha(x_2)\tan(\sigma x_1) \right) \right].
\end{align*}
\end{lemme}

\begin{proof} Let $x_2 \geqslant 0$ be fixed. By~\eqref{deltaGammagamma}, we have $\Gamma_\alpha(x_2)>0$. For every $x_1 \in (-\frac{\pi}{2\sigma},\frac{\pi}{2\sigma})$,
\begin{align*}
\ 
& \int_0^{x_1} \left( v(2x_0-s,x_2)-v(s,x_2) \right)\mathrm{d} s
\\ & \qquad = \int_0^{x_1} \left( \frac{2\sigma\Gamma_\alpha(x_2)}{\cos^2 \left( \pm \frac{\pi}{2}-\sigma s \right)+\Gamma_\alpha(x_2)^2 \sin^2 \left( \pm \frac{\pi}{2}-\sigma s \right)} \right.
\\ & \qquad \qquad \qquad \qquad \qquad \qquad\left. - \frac{2\sigma\Gamma_\alpha(x_2)}{\cos^2(\sigma s)+\Gamma_\alpha(x_2)^2\sin^2(\sigma s)} \right)\mathrm{d} s
\\ & \qquad = \int_0^{x_1} \left( \frac{2\sigma \Gamma_\alpha(x_2)}{\sin^2(\sigma s)+\Gamma_\alpha(x_2)^2\cos^2(\sigma s)} - \frac{2\sigma\Gamma_\alpha(x_2)}{\cos^2(\sigma s)+\Gamma_\alpha(x_2)^2\sin^2(\sigma s)} \right) \mathrm{d} s
\\ & \qquad = 2\int_0^{x_1} \left( \frac{\frac{\sigma}{\Gamma_\alpha(x_2) \cos^2(\sigma s)}}{1+\left( \frac{\tan (\sigma s)}{\Gamma_\alpha(x_2)} \right)^2} - \frac{\frac{\sigma\Gamma_\alpha(x_2)}{\cos^2(\sigma s)}}{1+(\Gamma_\alpha(x_2) \tan(\sigma s))^2} \right) \mathrm{d} s
\\ & \qquad = 2\arctan \left( \frac{\tan (\sigma x_1)}{\Gamma_\alpha(x_2)} \right) -2\arctan \left( \Gamma_\alpha(x_2) \tan (\sigma x_1) \right).
\end{align*}
By $\frac{\pi}{\sigma}$-periodicity of $v(\cdot,x_2)$, we extend this equalities to every $x_1 \in \mathbb{R} \setminus \left( \frac{\pi}{2\sigma}+\frac{\pi}{\sigma}\mathbb{Z} \right)$.
\end{proof}

\begin{remarque}
The function
\begin{equation*}
(x_1,x_2) \mapsto \arctan \left( \frac{\tan(\sigma x_1)}{\Gamma_\alpha(x_2)} \right) - \arctan \left( \Gamma_\alpha(x_2) \tan(\sigma x_1) \right),
\end{equation*}
defined for every $x_1 \in \mathbb{R} \setminus \left( \frac{\pi}{2\sigma}+\frac{\pi}{\sigma}\mathbb{Z} \right)$ and $x_2 \geqslant 0$, extends smoothly to $\overline{\mathbb{R}_+^2}$ by taking the value zero when $x_1 \in \frac{\pi}{2\sigma}+\frac{\pi}{\sigma}\mathbb{Z}$. In the following, we thus assume the above expression to be defined in $\overline{\mathbb{R}_+^2}$, keeping in mind that it is zero when $\tan(\delta x_1)$ is not well defined.
\end{remarque}

\begin{remarque}
Inverting the roles of $u$ and $v$, i.e. choosing firstly $u$ and secondly $v$, we get the same candidates $f$ in the first three cases. In Case~4, we only get a sign change in the relation~\eqref{dr1f_sign}. Hence, the function $f$ defined as
\begin{align*}
f(x_1,x_2)
& = 2n\pi - 2 \left[ \arctan \left( \frac{\tan(\sigma x_1)}{\Gamma_\alpha(x_2)} \right) - \arctan \left( \Gamma_\alpha(x_2) \tan(\sigma x_1) \right) \right] + \lambda x_2,
\end{align*}
for every $(x_1,x_2) \in \overline{\mathbb{R}_+^2}$, is also a candidate to be a solution of~\eqref{PNlambda}.
\end{remarque}

We summarize the possible configurations $(u,v)$ in the following table, with the possible expressions for $f$.

\begin{equation*}
\renewcommand{\arraystretch}{2}
\begin{array}{|c|l|}
\hline
(u,v) & f(x_1,x_2)
\\ \hline (u_0,u_0) & (2n+1)\pi+\lambda x_2
\\ \hline (u_0,u_2) & 2n\pi + 2\arctan \left( \frac{x_1}{1+x_2} \right) + \lambda x_2
\\ \hline (u_1,u_1) & 2n\pi + \lambda x_2
\\ \hline (u_2,u_0) & 2n\pi - 2\arctan \left( \frac{x_1}{1+x_2} \right) + \lambda x_2
\\ \hline (u_\alpha,v_\alpha) & 2n\pi \pm 2 \left[ \arctan \left( \frac{\tan(\sigma x_1)}{\Gamma_\alpha(x_2)} \right) - \arctan \left( \Gamma_\alpha(x_2)\tan(\sigma x_1) \right) \right] + \lambda x_2
\\ \hline
\end{array}
\end{equation*}

\begin{remarque}
Note that $(x_1,x_2) \in \overline{\mathbb{R}_+^2} \mapsto f(x_1,x_2)-\lambda x_2$ takes values in an interval of length less than $2\pi$.
\end{remarque}

Now it remains to check if all possible functions $f$ in the above table are solutions of~\eqref{PNlambda}. Let us recall
\begin{equation}
\tag{PN$_\lambda$}
\left\lbrace \begin{array}{l}
f \in C^\infty(\mathbb{R}_+^2) \cap C^1(\overline{\mathbb{R}_+^2}),
\\
(x_1,x_2) \mapsto f(x_1,x_2)-\lambda x_2 \text{ is bounded in } \mathbb{R}_+^2,
\\
\Delta f=0 \text{ in } \mathbb{R}_+^2,
\\
\partial_2 f-\lambda+\sin f = 0 \text{ on } \mathbb{R} \times \left\lbrace 0 \right\rbrace.
\end{array} \right.
\end{equation}
For any $n \in \mathbb{Z}$, the functions $(x_1,x_2) \mapsto (2n+1)\pi +\lambda x_2$ and $(x_1,x_2) \mapsto 2n\pi +\lambda x_2$ are clearly solutions of~\eqref{PNlambda}, in particular because $\sin(k\pi)=0$ for every $k \in \mathbb{Z}$. For any $n \in \mathbb{Z}$, the functions
\begin{equation*}
f^\pm \colon (x_1,x_2) \in \overline{\mathbb{R}_+^2} \mapsto 2n\pi \pm 2\arctan \left( \frac{x_1}{1+x_2} \right) + \lambda x_2
\end{equation*}
are harmonic in $\mathbb{R}_+^2$, of class $C^1$ in $\overline{\mathbb{R}_+^2}$ and the functions
\begin{equation*}
(x_1,x_2) \in \overline{\mathbb{R}_+^2} \mapsto f^\pm(x_1,x_2)-\lambda x_2
\end{equation*}
are bounded. Moreover, for every $x_1 \in \mathbb{R}$, $\partial_2f^\pm(x_1,0)-\lambda+\sin f^\pm(x_1,0)=0$. This shows that the functions $f^\pm$ are solutions of~\eqref{PNlambda}. To finish with, for $\alpha \in (1,2)$ and $n \in \mathbb{Z}$ fixed, we consider the function $f^\pm_\alpha \colon \overline{\mathbb{R}_+^2} \rightarrow \mathbb{R}$ defined as
\begin{align*}
f^\pm_\alpha (x_1,x_2) = 2n\pi \pm 2 \left[ \arctan \left( \frac{\tan(\sigma x_1)}{\Gamma_\alpha(x_2)} \right) - \arctan \left( \Gamma_\alpha (x_2)\tan(\sigma x_1) \right) \right] + \lambda x_2
\end{align*}
which is harmonic in $\mathbb{R}_+^2$ and of class $C^1$ in $\overline{\mathbb{R}_+^2}$ by Lemmas~\ref{lemg} and~\ref{lemexprualpha}. Moreover, the \linebreak function $(x_1,x_2) \in \overline{\mathbb{R}_+^2} \mapsto f^\pm_\alpha(x_1,x_2)-\lambda x_2$ is bounded. It remains to show that, for every $x_1 \in \mathbb{R}$,
\begin{align*}
\partial_2f^\pm_\alpha(x_1,0)-\lambda+\sin f^\pm_\alpha(x_1,0)&=0.
\end{align*}
First, we compute $\sin f_\alpha(\cdot,0)$. For every $x_1 \in \mathbb{R}$,
\begin{align*}
f^\pm_\alpha(x_1,0)
& = 2n\pi \pm 2 \left( \arctan \left( \frac{1}{\gamma}\tan(\sigma x_1) \right) - \arctan \left( \gamma \tan(\sigma x_1) \right) \right)
\end{align*}
since $\Gamma_\alpha(0)=\gamma$, and using trigonometric relations, we get
\begin{align*}
\sin f^\pm_\alpha(x_1,0)
& = \pm \left( \frac{1}{\gamma}-\gamma \right) \frac{2\tan(\sigma x_1)(1+\tan^2(\sigma x_1))}{\left( 1+\frac{1}{\gamma^2} \tan^2(\sigma x_1) \right) \left( 1+\gamma^2\tan^2(\sigma x_1) \right)}.
\end{align*}
On the other hand, computing $\partial_2f^\pm_\alpha(x_1,\cdot)$ for $x_1$ fixed, and using the identities $\Gamma_\alpha(0)=\gamma$ \linebreak and $\Gamma'_\alpha(0)=\sigma(1-\gamma^2)$, we get
\begin{align*}
\partial_2f^\pm_\alpha(x_1,0)-\lambda
& =\mp \sigma \left( 1-\gamma^2 \right) \left( 1+\frac{1}{\gamma^2} \right) \frac{2\tan(\sigma x_1)(1+\tan^2(\sigma x_1))}{\left( 1+\frac{1}{\gamma^2}\tan^2(\sigma x_1) \right) \left( 1+\gamma^2\tan^2(\sigma x_1) \right)}
\end{align*}
for every $x_1 \in \mathbb{R}$. Hence it remains to check that
\begin{align*}
\frac{1}{\gamma}-\gamma
& = \sigma \left( 1-\gamma^2 \right) \left( 1+\frac{1}{\gamma^2} \right).
\end{align*}
This is the following computation:
\begin{align*}
\frac{1}{\gamma}-\gamma-\sigma \left( 1-\gamma^2 \right) \left( 1+\frac{1}{\gamma^2} \right)
& = \left( \frac{1}{\gamma}-\gamma \right) \frac{\gamma-\sigma -\sigma \gamma^2}{\gamma}
\end{align*}
where, using~\eqref{deltaGammagamma},
\begin{align*}
\gamma-\sigma-\sigma\gamma^2
& = \frac{\alpha}{2\sigma}-\sigma-\frac{\alpha^2}{4\sigma}
 = \frac{2\alpha-4\sigma^2-\alpha^2}{4\sigma}
 = \frac{1}{4\sigma} \left( 2\alpha-\alpha(2-\alpha)-\alpha^2 \right)
 = 0.
\end{align*}
It confirms that $f^\pm_\alpha$ is solution of~\eqref{PNlambda}. The proof of Theorem~\ref{thmSolPNlambda} is complete.

\subsection{Proof of Theorem~\ref{thm_phir_minimizer}}
\label{subsection2.3}

We now study the link between solutions of~(PN$_\lambda$), given by Theorem~\ref{thmSolPNlambda}, and local minimizers of $E_{\varepsilon}^\delta$ in $\mathbb{R}_+^2$ in the sense of De Giorgi. More precisely, we get interested in the behaviour of the critical points $\varphi_\varepsilon$ of $E_\varepsilon^\delta$ near the boundary, in order to show the presence of boundary vortices (see~\cite{Kurzke06}). To go through this, we expect $(x_1,x_2) \mapsto \varphi_\varepsilon(x_1,x_2)-\delta_2 x_2$ to be nonconstant and nonperiodic.

Let $\lambda \in \mathbb{R}$. From Theorem \ref{thmSolPNlambda}, we deduce that the only nonconstant and nonperiodic solutions of~\eqref{PNlambda}, up to substracting $\lambda x_2$, are the functions
\begin{equation*}
(x_1,x_2) \in \overline{\mathbb{R}_+^2} \mapsto 2n\pi \pm 2 \arctan \left( \frac{x_1+a}{x_2+1} \right)+\lambda x_2
\end{equation*}
where $n \in \mathbb{Z}$ and $a \in \mathbb{R}$. Using~\eqref{phirepsilon} and~\eqref{systphirepsilon}, from the above solutions of~(PN$_\lambda$), we obtain the following corresponding solutions of~\eqref{systFeps} (see the proof of Proposition~\ref{prop_uminimizer} below for more details):
\begin{equation}
\label{criticEepsilonD}
(x_1,x_2) \in \overline{\mathbb{R}_+^2} \mapsto n\pi \pm \arctan \left( \frac{x_1+\varepsilon a}{x_2+\varepsilon} \right) +\delta_2x_2 - \frac{\pi}{2}
\end{equation}
where $n \in \mathbb{Z}$ and $a \in \mathbb{R}$.

This type of solutions is particularly relevant, because it illustrates a boundary vortex for the domain $\mathbb{R}_+^2$ at the point $(0,0)$ when taking the limits when $\varepsilon$ tends to zero (at scale $\varepsilon$, the vortex point is the point of coordinates $(-\varepsilon a,-\varepsilon)$). Indeed, assuming for a while that $\varepsilon$ and $\delta_2$ are negligible, we can test some combinations of $(n,\pm)$. The case $(n,\pm)=(1,-)$ gives a boundary vortex of degree $+1/2$ (see~\cite{BBH}, \cite{Kurzke06} or~\cite{IK21} for more information on the degree of $\mathbb{S}^1$-valued maps), and the magnetization $m=e^{i\varphi}$ seems like escaping from the point of coordinates $(-\varepsilon a,-\varepsilon)$ and behaves like $e^{i\theta}$ (see Figure~\ref{figure1}). The case $(n,\pm)=(0,+)$ gives a boundary vortex of degree $-1/2$, and the magnetization $m=e^{i\varphi}$ seems like converging to the point of coordinates $(-\varepsilon a,-\varepsilon)$ and behaves like $e^{-i\theta}$. 

\begin{figure}
\centering
\includegraphics[scale=0.5]{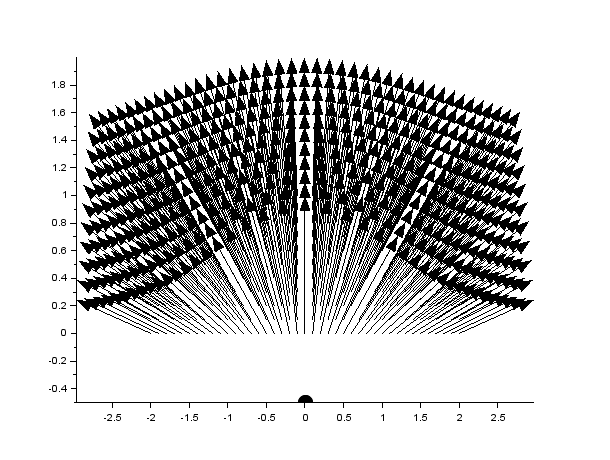}
\caption{Magnetization $m=e^{i\varphi}$ when $\varphi$ is of the form~\eqref{criticEepsilonD} with $n=1$, $\pm=-$, $a=0$, $\varepsilon=1/2$ and $\delta_2=1/10$.}
\label{figure1}
\end{figure}

The rest of this section is devoted to prove Theorem~\ref{thm_phir_minimizer}, i.e. we show that under the conditions~\eqref{conditionsminimizer}, the local minimizers $\varphi_\varepsilon$ of $E_\varepsilon^\delta$ in $\mathbb{R}_+^2$ in the sense of De Giorgi are the functions in~\eqref{criticEepsilonD} that correspond to the case $(n,\pm)=(1,-)$, which is the case of a boundary vortex of degree $+1/2$.

To begin with, we show in Proposition~\ref{prop_uminimizer} that any local minimizer of $E_\varepsilon^\delta$ in $\mathbb{R}_+^2$ in the sense of De Giorgi with conditions~\eqref{conditionsminimizer} must be of the expected form~\eqref{localminimizer}. Conversely, for proving that functions of the form~\eqref{localminimizer} are such local minimizers, we introduce the following definition.

\begin{definition}
\label{def_layer_function}
A function $\psi$ is a layer function associated to $E_\varepsilon^\delta$ in the sense of Cabr\'e and Sol\`a-Morales if it satisfies
\begin{equation*}
\left\lbrace \begin{array}{rcll}
\Delta \psi & = & 0 & \text{ in } \mathbb{R}_+^2,
\\
\partial_2\psi & = & \frac{1}{2\varepsilon}\sin(2\psi)+\delta_2 & \text{ on } \mathbb{R} \times \left\lbrace 0 \right\rbrace,
\end{array} \right.
\end{equation*}
and for every $x_1 \in \mathbb{R}$,
\begin{equation*}
\partial_1\psi(x_1,0)<0, \lim\limits_{x_1 \rightarrow +\infty} \psi(x_1,0)=0 \text{ and } \lim\limits_{x_1 \rightarrow -\infty} \psi(x_1,0)=\pi.
\end{equation*}
\end{definition}

These layer functions were studied by Cabr\'e and Sol\`a-Morales in~\cite{CSM05}. In that paper, properties in dimension two are given on the half plane $(0,+\infty) \times \mathbb{R}$, while we consider here $\mathbb{R}_+^2=\mathbb{R} \times (0,+\infty)$. The correspondence between a layer function $\psi$ in Definition~\ref{def_layer_function} and a layer function $u$ in~\cite{CSM05} is given by the relation
\begin{equation*}
\psi(x_1,x_2)=\frac{\pi}{2}\left( 1-u(-x_2,x_1) \right)
\end{equation*}
for any $(x_1,x_2) \in \mathbb{R}_+^2$.
In~\cite[Lemma~3.1]{CSM05}, Cabr\'e and Sol\`a-Morales prove that, if $r>0$ and $\psi$ is a layer function associated to $E_\varepsilon^\delta$, then $\psi$ is the unique weak solution of the problem
\begin{equation}
\label{pbCSM}
\left\lbrace \begin{array}{ll}
\Delta \varphi = 0 & \text{ in } B_r^+,
\\ \partial_2\varphi = \frac{1}{2\varepsilon}\sin(2\varphi)+\delta_2 & \text{ on } (-r,r) \times \left\lbrace 0 \right\rbrace,
\\ 0 \leqslant \varphi(x_1,x_2)-\delta_2x_2 \leqslant \pi & \text{ for every } (x_1,x_2) \in B_r^+,
\\ \varphi=\psi & \text{ on } \partial B_r^+ \cap \mathbb{R}_+^2.
\end{array} \right.
\end{equation}

\noindent
The method for proving this uniqueness property is the sliding method. This property will allow us to show that functions $\varphi_\varepsilon$ given by~\eqref{localminimizer} are local minimizers of $E_\varepsilon^\delta$ in $\mathbb{R}_+^2$ in the sense of De Giorgi that satisfy~\eqref{conditionsminimizer}. Our strategy of proof is the following: let $\varphi_\varepsilon$ be given by~\eqref{localminimizer}. Given $r>0$ and a minimizer $\widetilde{\varphi}_\varepsilon$ of $E_\varepsilon^\delta(\cdot;B_r)$ such that $\widetilde{\varphi}_\varepsilon=\varphi_\varepsilon$ on $\partial B_r^+ \cap \mathbb{R}_+^2$, we will prove in Proposition~\ref{prop_phirzeropi} that there exists a minimizer $\widehat{\varphi}_\varepsilon$ of $E_\varepsilon^\delta(\cdot;B_r)$ that satisfies~\eqref{pbCSM} with $\psi=\varphi_\varepsilon$ and $\varphi=\widehat{\varphi}_\varepsilon$. In Proposition~\ref{prop_phirlayer}, we show that $\varphi_\varepsilon$ is a layer function associated to $E_\varepsilon^\delta$ in the sense of Cabr\'e and Sol\`a-Morales. Hence by~\cite[Lemma~3.1]{CSM05}, it follows that $\varphi_\varepsilon=\widehat{\varphi}_\varepsilon$ is a minimizer of $E_\varepsilon^\delta(\cdot;B_r)$ with the boundary condition $\varphi_\varepsilon$ on $\partial B_r^+ \cap \mathbb{R}_+^2$. We finally conclude since $r>0$ is arbitrary.

\begin{proposition}
\label{prop_uminimizer}
Let $\varphi_\varepsilon \in H^1_\mathrm{loc}(\overline{\mathbb{R}_+^2})$ be a local minimizer of $E_\varepsilon^\delta$ in $\mathbb{R}_+^2$ in the sense of De Giorgi such that~\eqref{conditionsminimizer} is satisfied, i.e.
\begin{equation*}
\lim\limits_{x_1 \rightarrow +\infty} \varphi_\varepsilon(x_1,0)=0, \ 
\lim\limits_{x_1 \rightarrow -\infty} \varphi_\varepsilon(x_1,0)=\pi \text{ and } \left[ (x_1,x_2) \mapsto \varphi_\varepsilon(x_1,x_2)-\delta_2 x_2 \right] \in L^\infty(\mathbb{R}_+^2).
\end{equation*}
Then
\begin{equation*}
\varphi_\varepsilon(x_1,x_2)=\frac{\pi}{2}-\arctan \left( \frac{x_1+\varepsilon a}{x_2+\varepsilon} \right)+\delta_2x_2
\end{equation*}
for some $a \in \mathbb{R}$.
\end{proposition}

\begin{proof}
By Proposition~\ref{propsystFeps}, $\varphi_\varepsilon \in C^\infty(\overline{\mathbb{R}_+^2})$ and satisfies~\eqref{systFeps}. Set
\begin{equation*}
\phi_\varepsilon \colon (x_1,x_2) \in \overline{\mathbb{R}_+^2} \mapsto 2\varphi_\varepsilon(\varepsilon x_1, \varepsilon x_2)+\pi,
\end{equation*}
as in~\eqref{phirepsilon}. Then $\phi_\varepsilon$ satisfies~\eqref{systphirepsilon}, i.e.
\begin{equation*}
\left\lbrace \begin{array}{rcll}
\Delta \phi_\varepsilon&=&0&\text{ in } \mathbb{R}_+^2,
\\ \partial_2\phi_\varepsilon-\lambda_\varepsilon+\sin(\phi_\varepsilon) &=&0 & \text{ on } \mathbb{R} \times \left\lbrace 0 \right\rbrace,
\end{array} \right.
\end{equation*}
with $\lambda_\varepsilon=2 \varepsilon \delta_2$. Moreover $\phi_\varepsilon \in C^\infty(\overline{\mathbb{R}_+^2})$ and $(x_1,x_2) \mapsto \phi_\varepsilon(x_1,x_2)-\lambda_\varepsilon x_2$ is bounded in $\mathbb{R}_+^2$, because
\begin{equation*}
\phi_\varepsilon(x_1,x_2)-\lambda_\varepsilon x_2
= 2\varphi_\varepsilon(\varepsilon x_1,\varepsilon x_2)+\pi-2\varepsilon \delta_2 x_2
= 2 \underbrace{\left( \varphi_\varepsilon(\varepsilon x_1,\varepsilon x_2)-\delta_2 \varepsilon x_2 \right)}_{\in L^\infty(\mathbb{R}_+^2)}+\pi
\end{equation*}
for every $(x_1,x_2) \in \mathbb{R}_+^2$. It follows that $\phi_\varepsilon$ satisfies~(PN$_{\lambda_\varepsilon}$). By Theorem~\ref{thmSolPNlambda}, $\phi_\varepsilon$ must be one of the three following types of functions:
\begin{itemize}
\item[--] Firstly, the functions $(x_1,x_2) \mapsto n\pi+\lambda_\varepsilon x_2$ for some $n \in \mathbb{Z}$. However, on the boundary line $\mathbb{R} \times \left\lbrace 0 \right\rbrace$, this functions are constant (equal to $n\pi$). Hence, the boundary conditions $\lim_{x_1 \rightarrow +\infty} \varphi_\varepsilon(x_1,0)=0$ and $\lim_{x_1 \rightarrow -\infty} \varphi_\varepsilon(x_1,0)=\pi$ cannot be satisfied, and $\phi_\varepsilon$ is not of this first form.
\item[--] Secondly, the $x_1$-periodic functions given by~\eqref{solPNperiodic}. However, the boundary conditions \linebreak $\lim_{x_1 \rightarrow +\infty} \varphi_\varepsilon(x_1,0)=0$ and $\lim_{x_1 \rightarrow -\infty} \varphi_\varepsilon(x_1,0)=\pi$ are not compatible with the periodicity in the variable $x_1$. Thus $\phi_\varepsilon$ is not of this second form.
\item[--] Thirdly, the functions 
\begin{equation*}
(x_1,x_2) \mapsto 2n\pi \pm 2\arctan \left( \frac{x_1+a}{x_2+1} \right) + \lambda_\varepsilon x_2
\end{equation*}
for some $n \in \mathbb{Z}$ and $a \in \mathbb{R}$. Coming back to $\varphi_\varepsilon$ instead of $\phi_\varepsilon$, we get, for every $(x_1,x_2) \in \overline{\mathbb{R}_+^2}$,
\begin{equation*}
2\varphi_\varepsilon(\varepsilon x_1,\varepsilon x_2)+\pi
= 2n\pi \pm 2\arctan \left( \frac{x_1+a}{x_2+1} \right) + 2\varepsilon \delta_2 x_2,
\end{equation*}
i.e.
\begin{equation*}
\varphi_\varepsilon(x_1,x_2)= n\pi \pm \arctan \left( \frac{x_1+\varepsilon a}{x_2+\varepsilon} \right) +\delta_2 x_2 - \frac{\pi}{2}.
\end{equation*}
Since $\lim_{x_1 \rightarrow +\infty} \varphi_\varepsilon(x_1,0)=0$, we deduce that $n\pi \pm \frac{\pi}{2}-\frac{\pi}{2}=0$. It follows that the pair $(n,\pm)$ is either $(1,-)$ or $(0,+)$. As $\lim_{x_1 \rightarrow -\infty} \varphi_\varepsilon(x_1,x_2)=\pi$, we also have $n\pi \mp \frac{\pi}{2}-\frac{\pi}{2}=\pi$, hence the only possible pair $(n,\pm)$ is $(1,-)$. As a consequence, for every $(x_1,x_2) \in \overline{\mathbb{R}_+^2}$,
\begin{equation*}
\varphi_\varepsilon(x_1,x_2)= \frac{\pi}{2} - \arctan \left( \frac{x_1+\varepsilon a}{x_2+\varepsilon} \right) +\delta_2 x_2.
\end{equation*}
\end{itemize}
\end{proof}

\begin{proposition}
\label{prop_phirzeropi}
Let $r>0$ and $\varphi_\varepsilon$ be given by~\eqref{localminimizer}. Let $\widetilde{\varphi}_\varepsilon$ be a minimizer of $E_\varepsilon^\delta(\cdot;B_r)$ such that $\widetilde{\varphi}_\varepsilon=\varphi_\varepsilon$ on $\partial B_r^+ \cap \mathbb{R}_+^2$. Then there exists a minimizer $\widehat{\varphi}_\varepsilon$ of $E_\varepsilon^\delta(\cdot;B_r)$ that satisfies $\widehat{\varphi}_\varepsilon=\varphi_\varepsilon$ on $\partial B_r^+ \cap \mathbb{R}_+^2$ and $0 \leqslant \widehat{\varphi}_\varepsilon(x_1,x_2)-\delta_2x_2 \leqslant \pi$ for every $(x_1,x_2) \in \overline{B_r^+}$.
\end{proposition}

\begin{proof}
We define $\widehat{\varphi}_\varepsilon$ in $\overline{B_r^+}$ as follows:
\begin{equation}
\label{phirtildezeropi}
\widehat{\varphi}_\varepsilon(x_1,x_2) = \delta_2x_2 +
\left\lbrace \begin{array}{ll}
\widetilde{\varphi}_\varepsilon(x_1,x_2)-\delta_2x_2 & \text{ if } 0 \leqslant \widetilde{\varphi}_\varepsilon(x_1,x_2)-\delta_2x_2 \leqslant \pi,
\\
0 & \text{ if } \widetilde{\varphi}_\varepsilon(x_1,x_2)-\delta_2x_2 \leqslant 0,
\\
\pi & \text{ if } \widetilde{\varphi}_\varepsilon(x_1,x_2)-\delta_2x_2 \geqslant \pi.
\end{array}
\right.
\end{equation}
It is obvious that $0 \leqslant \widehat{\varphi}_\varepsilon(x_1,x_2)-\delta_2x_2 \leqslant \pi$ in $\overline{B_r^+}$. Moreover, on $\partial B_r^+ \cap \mathbb{R}_+^2$,
\begin{equation*}
\widetilde{\varphi}_\varepsilon(x_1,x_2)-\delta_2x_2 = \varphi_\varepsilon(x_1,x_2)-\delta_2x_2 \in [0,\pi],
\end{equation*} 
thus $\widehat{\varphi}_\varepsilon(x_1,x_2)
=\delta_2x_2+\widetilde{\varphi}_\varepsilon(x_1,x_2)-\delta_2x_2
=\widetilde{\varphi}_\varepsilon(x_1,x_2)
=\varphi_\varepsilon(x_1,x_2)$. For proving that $\widehat{\varphi}_\varepsilon$ is a minimizer of $E_\varepsilon^\delta(\cdot;B_r)$, it suffices to show that $E_\varepsilon^\delta(\widehat{\varphi}_\varepsilon;B_r) \leqslant E_\varepsilon^\delta(\widetilde{\varphi}_\varepsilon;B_r)$.

For the boundary integral on $B_r \cap (\mathbb{R} \times \left\lbrace 0 \right\rbrace)=(-r,r) \times \left\lbrace 0 \right\rbrace$, we note that on this line \linebreak segment, $\widehat{\varphi}_\varepsilon$ is equal either to $\widetilde{\varphi}_\varepsilon$, or to $0$, or to $\pi$. Thus $\sin^2 \widehat{\varphi}_\varepsilon \leqslant \sin^2 \widetilde{\varphi}_\varepsilon$, and
\begin{equation*}
\int_{(-r,r) \times \left\lbrace 0 \right\rbrace} \sin^2 \widehat{\varphi}_\varepsilon \ \mathrm{d} \mathcal{H}^1 \leqslant \int_{(-r,r) \times \left\lbrace 0 \right\rbrace} \sin^2 \widetilde{\varphi}_\varepsilon \ \mathrm{d} \mathcal{H}^1.
\end{equation*}

For the interior integral, we note that 
\begin{align*}
& \int_{B_r^+} \left( \left\vert \nabla \widehat{\varphi}_\varepsilon \right\vert^2 - 2\delta \cdot \nabla \widehat{\varphi}_\varepsilon \right) \mathrm{d} x
\\
& \qquad = 
\int_{B_r^+} \left( \left\vert \partial_1\widehat{\varphi}_\varepsilon \right\vert^2 - 2\delta_1 \partial_1\widehat{\varphi}_\varepsilon \right) \mathrm{d} x
+ \int_{B_r^+} \left\vert \partial_2\widehat{\varphi}_\varepsilon-\delta_2 \right\vert^2 \mathrm{d} x
- \delta_2^2 \left\vert B_r^+ \right\vert. 
\end{align*}
On the one hand, using~\eqref{phirtildezeropi}, we have
\begin{align*}
\left\vert \partial_1 \widehat{\varphi}_\varepsilon \right\vert
& =
\left\lbrace
\begin{array}{cl}
\left\vert \partial_1\widetilde{\varphi}_\varepsilon \right\vert
& \text{ if } 0 \leqslant \widetilde{\varphi}_\varepsilon(x_1,x_2)-\delta_2x_2 \leqslant \pi,
\\
0
& \text{ elsewhere},
\end{array}
\right.
\end{align*}
thus $\left\vert \partial_1 \widehat{\varphi}_\varepsilon \right\vert \leqslant \left\vert \partial_1 \widetilde{\varphi}_\varepsilon \right\vert$ in $B_r^+$, and $\int_{B_r^+} \left\vert \partial_1\widehat{\varphi}_\varepsilon \right\vert^2 \mathrm{d} x \leqslant \int_{B_r^+} \left\vert \partial_1\widetilde{\varphi}_\varepsilon \right\vert^2 \mathrm{d} x$. Moreover, 
\begin{equation*}
\int_{B_r^+} \partial_1 \widehat{\varphi}_\varepsilon \ \mathrm{d} x
= \int_{\partial B_r^+ \cap \mathbb{R}_+^2} \widehat{\varphi}_\varepsilon \nu_1 \ \mathrm{d} \mathcal{H}^1
= \int_{\partial B_r^+ \cap \mathbb{R}_+^2} \varphi_\varepsilon \nu_1 \ \mathrm{d} \mathcal{H}^1
= \int_{\partial B_r^+ \cap \mathbb{R}_+^2} \widetilde{\varphi}_\varepsilon \nu_1 \ \mathrm{d} \mathcal{H}^1
= \int_{B_r^+} \partial_1 \widetilde{\varphi}_\varepsilon \ \mathrm{d} x,
\end{equation*}
using that $\nu_1=0$ on $B_r \cap (\mathbb{R} \times \left\lbrace 0 \right\rbrace)$. On the other hand, using~\eqref{phirtildezeropi},
\begin{equation*}
\left\vert \partial_2 \widehat{\varphi}_\varepsilon-\delta_2 \right\vert =
\left\lbrace
\begin{array}{cl}
\left\vert \partial_2\widetilde{\varphi}_\varepsilon-\delta_2 \right\vert
& \text{ if } 0 \leqslant \widetilde{\varphi}_\varepsilon(x_1,x_2)-\delta_2x_2 \leqslant \pi,
\\
0
& \text{ elsewhere},
\end{array}
\right.
\end{equation*}
thus $\left\vert \partial_2 \widehat{\varphi}_\varepsilon-\delta_2 \right\vert \leqslant \left\vert \partial_2 \widetilde{\varphi}_\varepsilon-\delta_2 \right\vert$ in $B_r^+$, and $\int_{B_r^+} \left\vert \partial_2 \widehat{\varphi}_\varepsilon-\delta_2 \right\vert^2 \mathrm{d} x \leqslant \int_{B_r^+} \left\vert \partial_2 \widetilde{\varphi}_\varepsilon-\delta_2 \right\vert^2 \mathrm{d} x$. Combining the above inequalities comparing the boundary and interior parts of $E_\varepsilon^\delta(\widehat{\varphi}_\varepsilon;B_r)$ and $E_\varepsilon^\delta(\widetilde{\varphi}_\varepsilon;B_r)$, we deduce that $E_\varepsilon^\delta(\widehat{\varphi}_\varepsilon;B_r) \leqslant E_\varepsilon^\delta(\widetilde{\varphi}_\varepsilon;B_r)$.
\end{proof}

\begin{proposition}
\label{prop_phirlayer}
Let $\varphi_\varepsilon$ be given by~\eqref{localminimizer}, i.e.
\begin{equation*}
\varphi_\varepsilon \colon (x_1,x_2) \in \overline{\mathbb{R}_+^2} \mapsto \frac{\pi}{2} - \arctan \left( \frac{x_1+\varepsilon a}{x_2+\varepsilon} \right) + \delta_2x_2
\end{equation*}
for some $a \in \mathbb{R}$. Then $\varphi_\varepsilon$ is a layer function associated to $E_\varepsilon^\delta$ in the sense of Cabr\'e and Sol\`a-Morales.
\end{proposition}

\begin{proof}
It is clear that $\varphi_\varepsilon \in C^\infty(\overline{\mathbb{R}_+^2})$ and is a harmonic function in $\mathbb{R}_+^2$. The boundary condition
\begin{equation*}
\partial_2\varphi_\varepsilon(x_1,0)=\frac{1}{2\varepsilon}\sin(2\varphi_\varepsilon(x_1,0))+\delta_2
\end{equation*}
for every $x_1 \in \mathbb{R}$ follows from a standard calculation. Moreover, for every $x_1 \in \mathbb{R}$,
\begin{equation*}
\partial_1\varphi_\varepsilon(x_1,0)=\frac{-\varepsilon}{(x_1+\varepsilon a)^2+\varepsilon^2}<0,
\end{equation*}
and we also have $\lim_{x_1 \rightarrow +\infty} \varphi_\varepsilon(x_1,0)=\frac{\pi}{2}-\frac{\pi}{2}=0$ and $\lim_{x_1 \rightarrow -\infty} \varphi_\varepsilon(x_1,0)=\frac{\pi}{2}-\left(-\frac{\pi}{2}\right)=\pi$.
\end{proof}

\begin{proof}[Proof of Theorem~\ref{thm_phir_minimizer}]
Let $\varphi_\varepsilon \in H^1_\mathrm{loc}(\overline{\mathbb{R}_+^2})$ be a local minimizer of $E_\varepsilon^\delta$ in $\mathbb{R}_+^2$ in the sense of De Giorgi that satisfies~\eqref{conditionsminimizer}. By Proposition~\ref{prop_uminimizer}, $\varphi_\varepsilon$ is of the form~\eqref{localminimizer}.

It remains to check that functions $\varphi_\varepsilon$ of the form~\eqref{localminimizer} are local minimizers of $E_\varepsilon^\delta$ in $\mathbb{R}_+^2$ in the sense of De Giorgi that satisfy~\eqref{conditionsminimizer}. Let $\varphi_\varepsilon$ be given by~\eqref{localminimizer}. Then conditions~\eqref{conditionsminimizer} are directly satisfied and by Proposition~\ref{prop_phirlayer}, $\varphi_\varepsilon$ is a layer function associated to $E_\varepsilon^\delta$ in the sense of Cabr\'e and Sol\`a-Morales. Let $r>0$. By the direct method in the calculus of variations, $E_\varepsilon^\delta(\cdot;B_r)$ admits a minimizer $\widetilde{\varphi}_\varepsilon \in H_1(B_r^+)$ such that $\widetilde{\varphi}_\varepsilon=\varphi_\varepsilon$ on $\partial B_r^+ \cap \mathbb{R}_+^2$. By Proposition~\ref{prop_phirzeropi}, there exists a minimizer $\widehat{\varphi}_\varepsilon$ of $E_\varepsilon^\delta(\cdot;B_r)$ such that $\widehat{\varphi}_\varepsilon=\varphi_\varepsilon$ on $\partial B_r^+ \cap \mathbb{R}_+^2$ and $0 \leqslant \widehat{\varphi}_\varepsilon(x_1,x_2)-\delta_2x_2 \leqslant \pi$ for \linebreak every $(x_1,x_2) \in \overline{B_r^+}$. Moreover,
\begin{equation*}
\left\lbrace \begin{array}{rcll} \Delta \widehat{\varphi}_\varepsilon & = & 0 & \text{ in } B_r^+, \\ \partial_2 \widehat{\varphi}_\varepsilon & = & \frac{1}{2\varepsilon} \sin 2\widehat{\varphi}_\varepsilon + \delta_2 & \text{ on } (-r,r) \times \left\lbrace 0 \right\rbrace.
\end{array} \right.
\end{equation*}
similarly than in Proposition~\ref{propsystFeps}. It follows that \eqref{pbCSM} is satisfied with $\psi=\varphi_\varepsilon$ and $\varphi=\widehat{\varphi}_\varepsilon$. By~\cite[Lemma~3.1]{CSM05}, $\varphi_\varepsilon=\widehat{\varphi}_\varepsilon$, thus $\varphi_\varepsilon$ is a minimizer of $E_\varepsilon^\delta(\cdot;B_r)$ with the boundary condition $\varphi_\varepsilon$ \linebreak on $\partial B_r^+ \cap \mathbb{R}_+^2$. This fact being true for every $r>0$, $\varphi_\varepsilon$ is a local minimizer of $E_\varepsilon^\delta$ in $\mathbb{R}_+^2$ in the sense of De Giorgi.
\end{proof}

\begin{remarque}
In dimension greater than two, solving~\eqref{PNlambda} and finding local minimizers of $E_\varepsilon^\delta$ in the sense of De Giorgi is more difficult because there is no classification as for the Benjamin-Ono problem. In the case where $\delta=0$, Cabr\'e and Sol\`a-Morales~\cite{CSM05} introduce the notion of layer solution of~(PN$_0$) (or similar problems with a different nonlinearity instead of $\sin$) and show that a layer solution of~(PN$_0$) is a local minimizer of $E_\varepsilon^0$ in the sense of De Giorgi in any dimension.
\end{remarque}

\bibliographystyle{plain}

\end{document}